\documentclass[11pt,a4paper]{amsart}
\pdfoutput=1
\usepackage[utf8]{inputenc}
\usepackage{enumerate}
\usepackage{paralist}
\usepackage{mathrsfs}
\usepackage{color}
\usepackage[pdftex]{graphicx}
\usepackage{amssymb}
\usepackage{amsmath,amsthm}
\usepackage{hyperref}
\usepackage{soul}
\usepackage{datetime}
\hypersetup{
  colorlinks=false,
  plainpages=false,
}

\DeclareGraphicsExtensions{pdf, jpg, jpeg, png}

\newtheorem{theorem}{Theorem}
\newtheorem{lemma}[theorem]{Lemma}
\newtheorem{proposition}[theorem]{Proposition}
\newtheorem{corollary}[theorem]{Corollary}

\newtheorem{definition}[theorem]{Definition}

\newtheorem{remark}[theorem]{Remark}

\renewcommand{\paragraph}[1]{\medskip\noindent\textbf{#1}}

\newcommand{\ommit}[1]{
}

\newcommand{\R}{\mathbb{R}}
\newcommand{\C}{\mathbb{C}}

\newcommand{\Z}{\mathbb{Z}}

\newcommand{\RP}{\mathbb{R}\textup{P}}

\newcommand{\inc}[1]{\operatorname{inc}[#1]}

\newcommand{\crr}{\operatorname{cr}}

\renewcommand{\phi}{\varphi}
\renewcommand{\epsilon}{\varepsilon}
\newcommand{\longto}{\longrightarrow}

\newcommand{\inv}[1]{\mathcal{C}^{#1}}

\newenvironment{titleproof}[1]{\noindent{\bf #1.}}{\hfill\qed\par\medskip}

\usepackage{caption2}
\setcaptionwidth{.9\textwidth}

\dummycaptionstyle{longtable}{}  

\title[On Weingarten transformations of hyperbolic nets]{
	On Weingarten transformations of hyperbolic nets
}

\author{Emanuel Huhnen-Venedey \and Wolfgang K. Schief} %
\address{Emanuel Huhnen-Venedey\vspace{-5pt}}
\address{Institute of Mathematics, Secr. MA 8-3, TU Berlin, 10623 Berlin, Germany}%
\email{huhnen@math.tu-berlin.de}

\address{Wolfgang K. Schief\vspace{-5pt}} 
\address{School of Mathematics and Statistics, UNSW, NSW 2052, Sydney, Australia}
\email{w.schief@unsw.edu.au}

\thanks{The research of the first author was supported by the DFG
Collaborative Research Center TRR 109, “Discretization in Geometry and
Dynamics.” The first author also gratefully acknowledges the hospitality of the
School of Mathematics and Statistics at the University of New South Wales in
Sydney and the support of the German Academic Exchange Service (DAAD) during
his time as a Visiting Fellow at UNSW}



\begin{document}

\begin{abstract}
Weingarten transformations which, by definition, preserve the asymptotic lines
on smooth surfaces have been studied extensively in classical differential
geometry and also play an important role in connection with the modern
geometric theory of integrable systems. Their natural discrete analogues have
been investigated in great detail in the area of (integrable) discrete
differential geometry and can be traced back at least to the early 1950s. Here,
we propose a canonical analogue of (discrete) Weingarten transformations for
hyperbolic nets, that is, $C^1$-surfaces which constitute hybrids of smooth and
discrete surfaces ``parametrized'' in terms of asymptotic coordinates. We prove
the existence of Weingarten pairs and analyse their geometric and algebraic
properties. 
\end{abstract}

\maketitle

\section{Introduction}
\label{sec:intro}

The subject of the present paper is the determination and analysis of a canoncial class of
transformations associated with so-called hyperbolic nets.  The latter have
been introduced recently in \cite{Huhnen-VenedeyRoerig:2013:hyperbolicNets} and
constitute a discretization of surfaces in 3-space that are parametrized along
asymptotic lines.  A parametrization of a surface is called an asymptotic line
parametrization if, at each point of the surface, parameter lines follow the
distinguished directions of vanishing normal curvature. For smooth
surfaces, one has unique asymptotic line parametrizations (up to
reparametrization of parameter lines) around hyperbolic points, that is, around
points of negative Gaussian curvature \cite{Eisenhart:1960:Treatise}.  It is natural to discretize
parametrized surfaces by quadrilateral nets, also called quadrilateral meshes.
Compared with, e.g., discrete triangulated surfaces, quadrilateral nets do not
only discretize continuous surfaces understood as topological objects (point sets),
but also reflect the combinatorial structure of parameter lines. 
While unspecified quadrilateral nets discretize arbitrary parametrizations, the
discretization of distinguished types of parametrizations yields quadrilateral
nets with special geometric properties.  One of the most fundamental examples
is the discretization of conjugate parametrizations by quadrilateral nets with
planar faces. Discretizing more specific conjugate parametrizations then yields
planar quadrilateral nets with additional properties. However, as asymptotic
line parametrizations are not conjugate parametrizations, they are not modelled
by quadrilateral nets with planar faces. Instead, asymptotic line
parametrizations are properly discretized by quadrilateral nets with planar
vertex stars, that is, nets for which every vertex is coplanar with its nearest
neighbours.  We use the terminology of \cite{BobenkoSuris:2008:DDGBook}, calling
discrete nets with planar quadrilaterals \emph{Q-nets} and (skew) quadrilateral
nets with planar vertex stars \emph{A-nets}. Q-nets and A-nets as
discretizations of conjugate and asymptotic line parametrizations were already
introduced in~\cite{Sauer:1937:ProjLinienGeometrie}.

Various aspects of continuous asymptotic line parametrizations have been
discretized using A-nets.  For example, the discretization of surfaces of
constant negative Gaussian curvature as special A-nets, nowadays often called
K-surfaces, can be found in
\cite{Sauer:1950:Pseudosphaeren,Wunderlich:1951:K-Flaechen}. In the context of
the connections between geometry and integrability, the relation between
discrete K-surfaces and Hirota's \cite{Hirota:1977:DiscreteSineGordon}
algebraic discretization of the sine-Gordon equation was established much later
\cite{BobenkoPinkall:1996:DiscreteKandHirota}.  For a special instance of this
relation, see, for example, \cite{Hoffmann:1999:DiscreteAmsler} on discrete
Amsler-surfaces.  Discrete indefinite affine spheres
\cite{BobenkoSchief:1999:AffineSpheresIndefinite} are an example for the
discretization of a certain class of smooth A-nets within affine differential
geometry.  The discrete Lelieuvre representation of A-nets and the related
discrete Moutard equations are, for instance, treated in
\cite{Schief:1997:MoutardSuperposition,BobenkoSchief:1999:AffineSpheresDuality,
KonopelchenkoPinkall:2000:ProjectiveLelieuvre,Doliwa:2001:ANetsPluecker,
DoliwaNieszporskiSantini:2001:IntegrableReductionOfAnets,
Nieszporski:2002:DiscreteANets}.

Based on the discretization of asymptotic line parametrizations by A-nets,
\emph{hyperbolic nets} arise as an extension of A-nets in the sense that
elementary quadrilaterals of A-nets become extended to hyperbolic surface
patches.  More precisely, a hyperbolic net is a piecewise smooth surface
composed of \emph{hyperboloid patches}, where the latter refers to surface
patches that are taken from doubly ruled quadrics, i.e., one-sheeted
hyperboloids and hyperbolic paraboloids, by ``cutting along asymptotic lines''.
In order to obtain a hyperbolic net, hyperboloid patches are inserted into the
skew quadrilaterals of a supporting discrete A-surface such that the tangent
planes of edge-adjacent patches coincide along the common boundary
edge\footnote{ This is analogous to the discretization of curvature line
parametrized surfaces by cyclidic nets
\cite{BobenkoHuhnen-Venedey:2011:cyclidicNets}. A cyclidic net is composed of
surface patches that are taken from Dupin cyclides by ``cutting along curvature
lines'' and then glued along those cuts in a continuously differentiable way.}
(cf. Fig.~\ref{fig:equi_pringle_hypnet_intro}). Hence, hyperbolic nets are
$C^1$-surfaces which may be regarded as ``hybrids'' of smooth surfaces
parametrized in terms of asymptotic coordinates and their discrete
counterparts.

\begin{figure}[htb]
\begin{center}
\includegraphics[scale=.14]{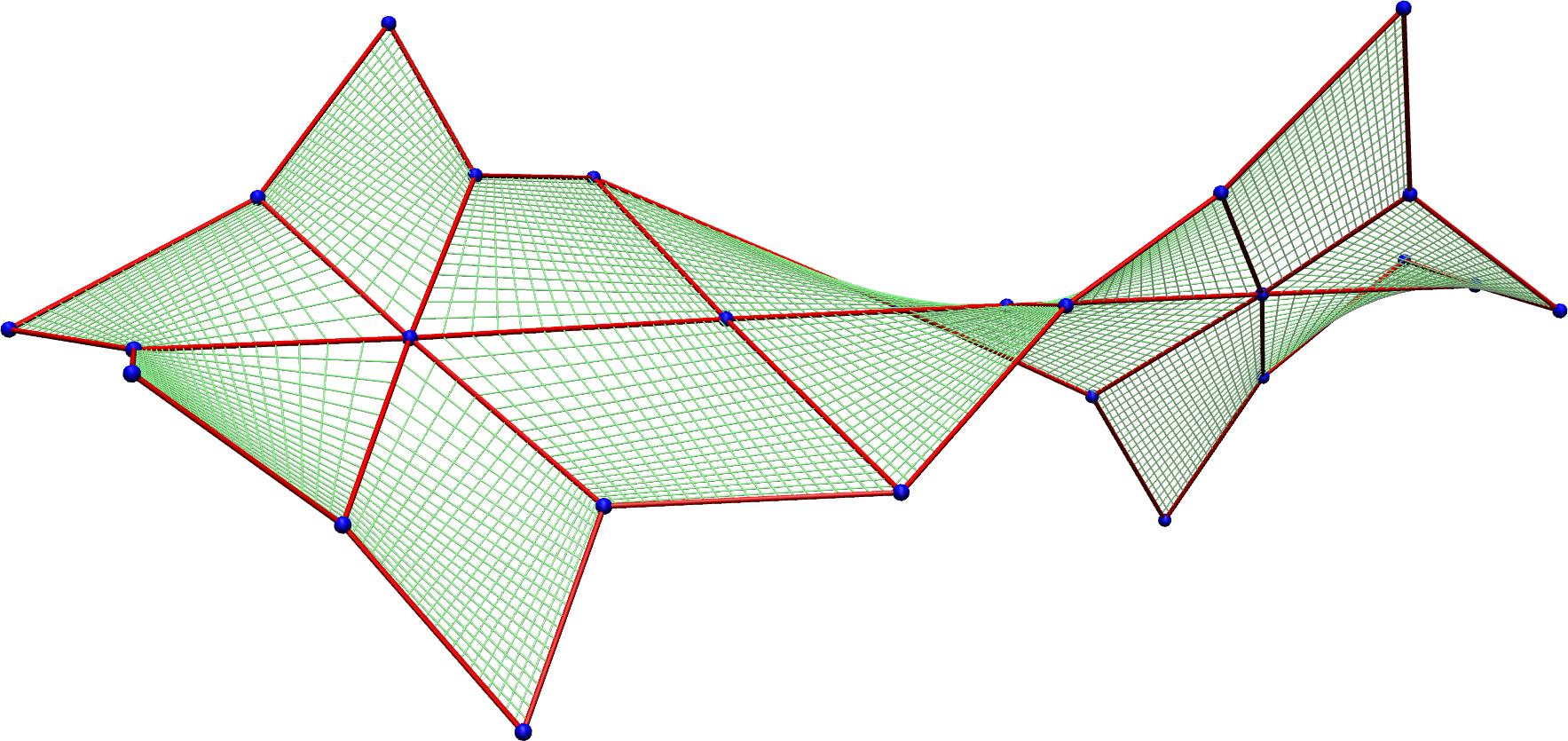}
\end{center}
\caption{A hyperbolic net. The red segments are the edges
of the supporting discrete A-surface, bounding the individual hyperboloid patches.}
\label{fig:equi_pringle_hypnet_intro}
\end{figure}

A specific subclass of hyperbolic nets, that is, hyperbolic nets that comprise
only surface patches taken from hyperbolic paraboloids, have already appeared
implicitly as discrete affine minimal surfaces in
\cite{Craizer:2010:AffineMinimalSurfaces}.  This relation is discussed in
detail in \cite{KaeferboeckPottmann:2012:DiscreteAffineMinimal}.  Aiming at the
application in the context of architectural geometry, a parametric description
of hyperbolic nets  in terms of rational bilinear patches has been given
recently in \cite{ShiWangPottmann:2013:RationalBilinearPatches}, wherein also
the approximation of a given negatively curved surface by hyperbolic nets is
investigated. The rational bilinear description is closely related to the elementary
geometric characterization of hyperbolic nets on which we rely in the present
paper.  While hyperbolic nets were introduced originally in the more abstract
setting of Pl\"ucker line geometry, the elementary description we use here
is formulated in terms of \emph{crisscrossed quadrilaterals}.  The latter are skew
quadrilaterals that are equipped with a pair of crossing lines, which uniquely
describes the extension of the supporting quadrilateral to a hyperboloid patch
that is bounded by the quadrilateral (cf.
Fig.~\ref{fig:crisscross_quad_extension}). Indeed, a crisscrossed quadrilateral
is a natural representative of a rational bilinear patch, the latter being a
rational bilinear parametrization of a hyperboloid patch over $[0,1]^2$ such
that the crossing line segments are the $\tfrac 1 2$-parameter lines.
Hyperbolic nets are then described as \emph{crisscrossed A-surfaces}, for which
crosses associated with edge-adjacent quadrilaterals have to satisfy an
incidence relation which guarantees that the corresponding hyperboloid patches
join smoothly along the common boundary edge.

\begin{figure}[htb]
\begin{center}
\parbox{.25\textwidth}{\includegraphics[width=\linewidth]{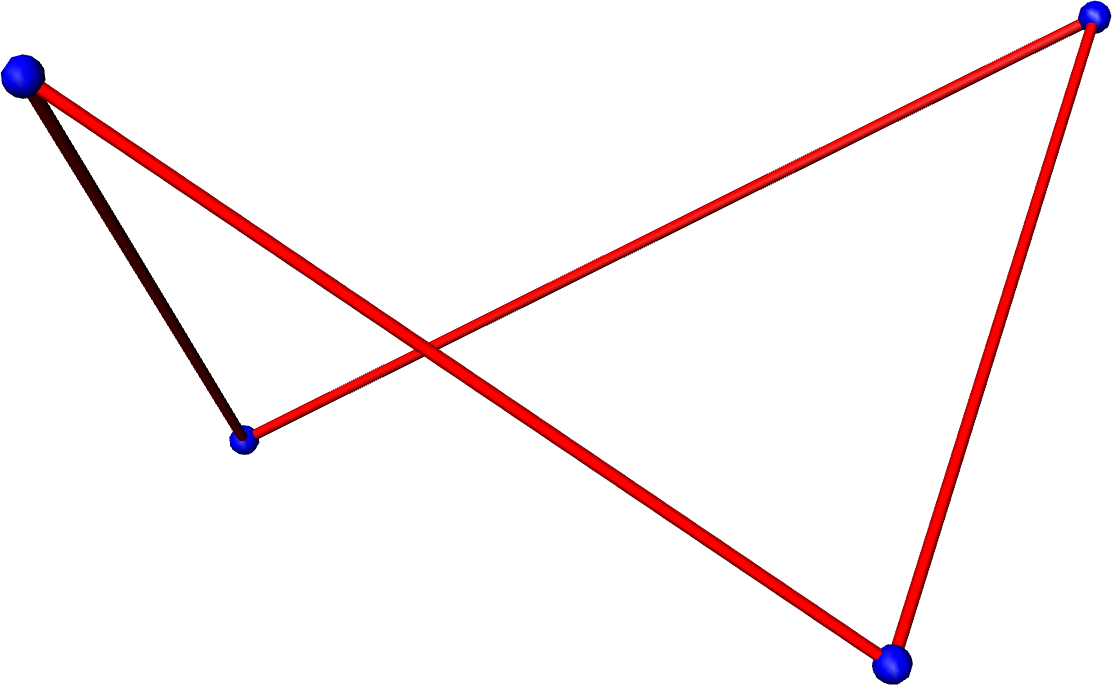}}
\hspace{.05\linewidth}
\parbox{.25\textwidth}{\includegraphics[width=\linewidth]{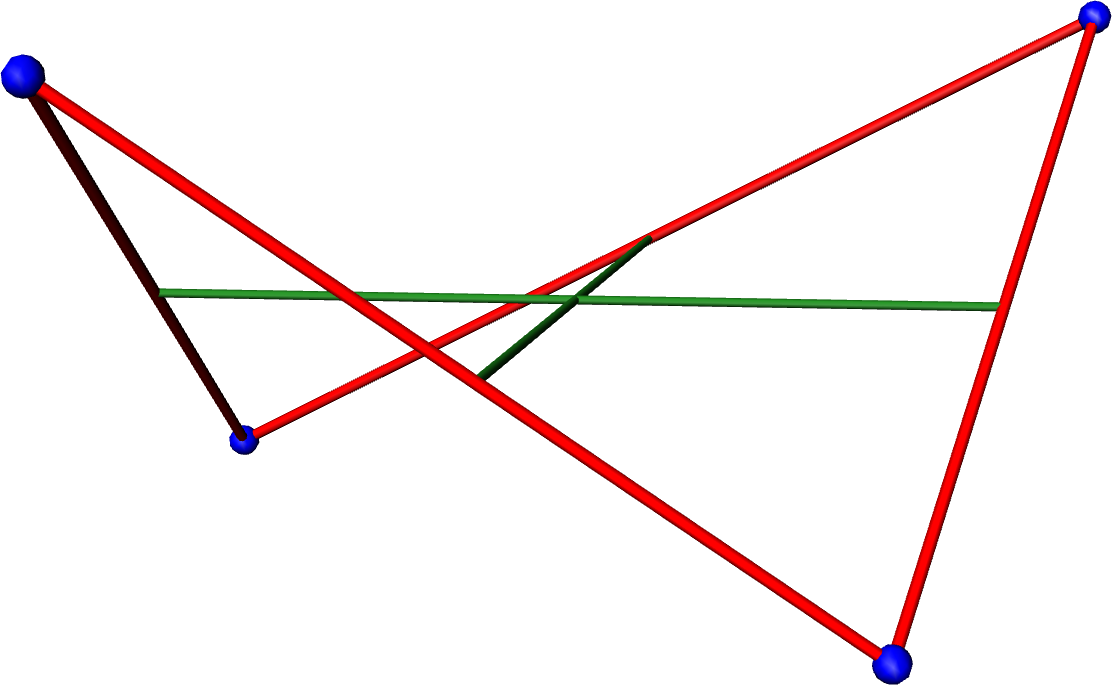}}
\hspace{.05\linewidth}
\parbox{.25\textwidth}{\includegraphics[width=\linewidth]{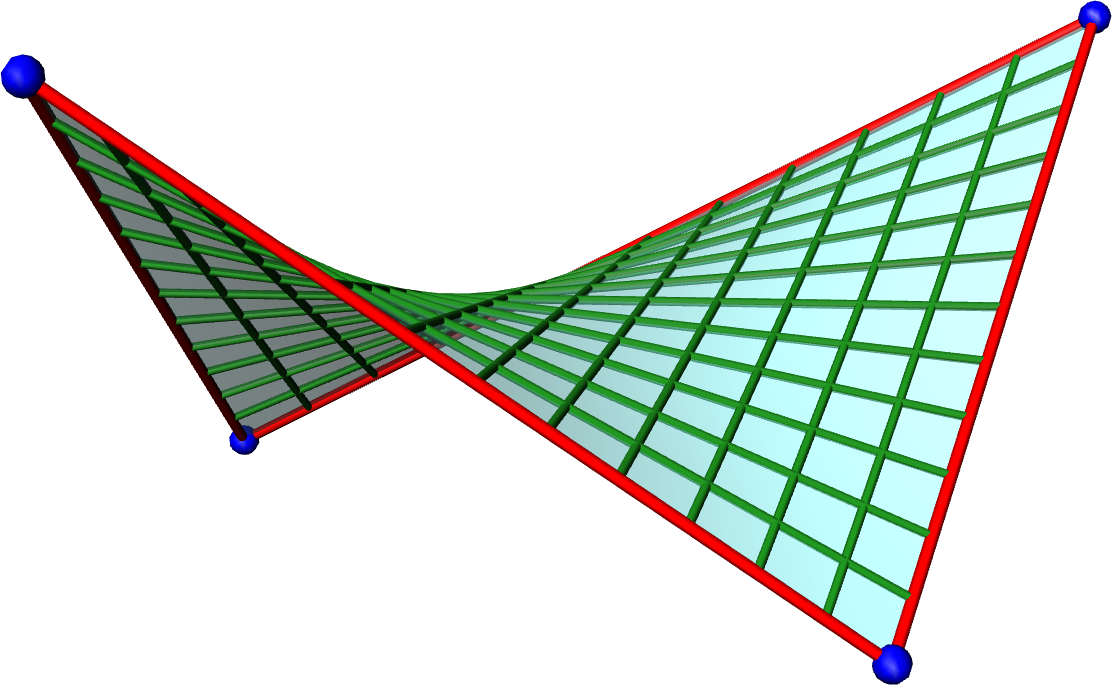}}
\end{center}
\caption{For a crisscrossed quadrilateral, there exists a unique doubly ruled quadric that contains the
two crossing lines and the four edges of the supporting quadrilateral.}
\label{fig:crisscross_quad_extension}
\end{figure}

Going beyond the discretization of individual surfaces, another issue is the
discretization of the class of transformations that is associated with
classical A-surfaces.  In general, the class of associated transformations and
related permutability theorems are an essential aspect of specific ``integrable''
surface parametrizations, that is, surface parametrizations which admit underlying
integrable structure \cite{SchiefRogers:2002:BaecklundDarboux}.  
In the case of A-nets, the associated
transformations are called \emph{Weingarten transformations} 
\cite{Eisenhart:1960:Treatise,SchiefRogers:2002:BaecklundDarboux}.  Two continuous
surfaces parametrized along asymptotic lines over the same domain are said to be Weingarten
transforms of each other if the line connecting corresponding points is the
intersection of the tangent planes to the two surfaces at these points.  This
relation carries over naturally to the setting of discrete A-nets in the
following way (see, e.g.,
\cite{Doliwa:2001:ANetsPluecker,DoliwaNieszporskiSantini:2001:IntegrableReductionOfAnets,
Nieszporski:2002:DiscreteANets}).
For a discrete A-surface $f$ and a vertex $x$ of $f$, the plane containing the
vertex star of $x$ is conveniently understood as the tangent plane to $f$ at
$x$.  Now let $\tilde f$ be another discrete A-surface with the same combinatorics
as $f$.  The A-surfaces $f$ and $\tilde f$ are said to form a discrete
Weingarten pair if the line connecting corresponding vertices $x$ and $\tilde
x$ is the intersection of the corresponding discrete tangent planes.
Equivalently, this relation may be described as follows.  Connecting
corresponding vertices of $f$ and $\tilde f$, one obtains a 3-dimensional
quadrilateral net $F$ that is composed of the two 2-dimensional layers $f$ and
$\tilde f$. The surfaces $f$ and $\tilde f$ form a discrete Weingarten pair if
and only if the net $F$ has planar vertex stars, which means that $F$ is a
3-dimensional A-net itself.  This illustrates a well established discretization
principle within discrete differential geometry, i.e., on the discrete
level, surfaces and their transformations should be described by the same
geometric or algebraic conditions. This approach reflects a deep and unifying
understanding of the classical relations between parametrized surfaces and
their transformations in the context of discrete integrability (see, e.g.,
\cite{BobenkoSuris:2008:DDGBook}). It is worth mentioning that, analogous to
the classical theory, two A-nets are discrete Weingarten transforms of each
other if and only if their discrete Lelieuvre normals are related by a discrete Moutard
transformation  (see, e.g., \cite{Doliwa:2001:ANetsPluecker,
DoliwaNieszporskiSantini:2001:IntegrableReductionOfAnets,
Nieszporski:2002:DiscreteANets}).

The aim of the present article is to develop, in the context of hyperbolic nets, 
a canonical analogue of the classical and modern theories of Weingarten 
transformations for smooth and discrete A-surfaces respectively. Since hyperbolic
nets possess the key features of both smooth and discrete A-surfaces, it is natural 
to demand that the same be true for their transformations. Thus, we here propose 
that two hyperbolic nets form a Weingarten pair if the
supporting A-surfaces form a discrete Weingarten pair and, additionally, the hyperboloid
patches associated with corresponding quadrilaterals are related by a classical
Weingarten transformation. It turns out that this definition is indeed admissible but that
the proof of this assertion is significantly more involved than the proof of the existence 
of both classical and discrete Weingarten transformations. Accordingly, we here confine
ourselves to the investigation of single applications of Weingarten transformations and
address the permutability properties (Bianchi diagram)
of Weingarten transformations of hyperbolic nets and related aspects in a separate
publication.

\paragraph{Structure and results of the present paper.} We begin in
Section~\ref{sec:anets} with an overview of the aspects of the theory of
discrete A-surfaces and their transformations which are relevant for our
purposes.  Subsequently, in Section~\ref{sec:hypnets_as_crisscrossed_anets},
the description of hyperbolic nets recorded in
\cite{Huhnen-VenedeyRoerig:2013:hyperbolicNets} is briefly reviewed and
reformulated in terms of crisscrossed quadrilaterals.  Particular attention is
given to a scalar function $\rho$ defined at the vertices of a supporting A-net
that describes crosses adapted to quadrilaterals of the support structure. In
the case that the crosses encapsulate a hyperbolic net, the relation between
this function $\rho$ and algebraic invariants composed of the discrete Moutard
coefficients of the A-net is revealed.\footnote{It turns out that the scalars
$\rho$ are exactly the weights used in the rational bilinear patch description
of hyperbolic nets of \cite{ShiWangPottmann:2013:RationalBilinearPatches}.} In
Section~\ref{sec:weingarten_trafos}, we develop the concept of Weingarten
transformations of hyperbolic nets which constitute a subclass of more general
B\"acklund transformations, the anlogues of which do not exist in the classical
and discrete cases.  We show how these may be characterized both geometrically
in terms of crosses and algebraically in terms of the function $\rho$. It
turns out that in the case of the generic B\"acklund transformation, the latter
key function is governed by a non-autonomous version of the master discrete BKP
(Miwa) equation of integrable systems theory
\cite{Miwa:1982:HirotaDifferenceEquations}.  Moreover, it is demonstrated that,
in the particular case of a Weingarten pair, $\rho$ may be identified with a
potential for a particular choice of Moutard coefficients of a Lelieuvre
representation associated with  the underlying 2-layer 3D A-net. Accordingly,
the above-mentioned BKP-type equation reduces to the standard discrete BKP
equation which, in turn, gives rise to a novel geometric interpretation of
Miwa's fundamental equation.

It is important to note that a discrete
A-surface may be extended to a hyperbolic net if and only if a certain condition on the
twist of quadrilateral strips is satisfied. However, any A-surface with 
$\Z^2$ combinatorics is extendable in an analogous sense if the elementary 
quadrilaterals are equipped with whole hyperboloids rather than hyperboloid
patches. Such nets, which still obey the tangency condition along edges, are 
termed {\em pre-hyperbolic nets}. Accordingly, our general approach is to introduce first
B\"acklund and Weingarten transformations for pre-hyperbolic nets and then derive the theory
for hyperbolic nets by taking into account the additional constraint on the quadrilateral strips. 
Here, the global existence of Weingarten pairs is proven by converting this constraint into a
condition on the aforementioned algebraic invariants. 

\section{Discrete A-nets}
\label{sec:anets}

In the following, we introduce the notion of discrete A-nets 
\cite{Sauer:1937:ProjLinienGeometrie,BobenkoSuris:2008:DDGBook} and summarize
different aspects of the related theory that are important for our purpose. We
start with the 2-dimensional case, i.e., discrete A-surfaces, and then move on
to the higher-dimensional case, which is conveniently understood as the
(integrable) theory of discrete A-surfaces and their associated
transformations.  This approach provides us with a structure which will be
used as a guide when developing the analogous theory of hyperbolic nets and their
transformations.

\paragraph{Notation.}
For a discrete map defined on $\Z^m$,
it is convenient to represent shifts in lattice directions by lower indices.
Accordingly, for $z = (z_1,\dots,z_m) \in \Z^m$ and a map $\phi$ on $\Z^m$, we
write
\begin{equation*}
\phi_1(z) := \phi(z_1+1,z_2,\dots,z_m), \quad
\phi_{11}(z) := \phi(z_1+2,z_2,\dots,z_m), 
\end{equation*}
\begin{equation*}
\phi_2(z) := \phi(z_1,z_2+1,z_3,\dots,z_m), \quad \text{etc.}
\end{equation*}
Usually, we omit the argument for discrete maps and write
\begin{equation*}
  \phi = \phi (z), \quad \phi_1 = \phi_1 (z) \quad \text{etc.}
\end{equation*}
For $k \in \left\{ 1,\dots,m \right\}$ denote by $\mathcal{S}^{i_1 \dots i_k}$
the $k$-dimensional subspace of $\Z^m$ that is spanned by directions
$i_1,\dots,i_k$,
\begin{equation*}
\mathcal{S}^{i_1 \dots i_k} = \operatorname{span}_\Z(e_{i_1},\dots,e_{i_k}),
\end{equation*}
where $e_i$ is the $i$-th unit vector in $\Z^m$.
Finally, for $x_1,\dots,x_n \in \R^m$ we denote by
\begin{equation*}
\inc{x_1,\dots,x_n} = \left\{ \sum_{i=1}^n \alpha_i x_i \mid \sum_{i=1}^n \alpha_i = 1 \right\}
\end{equation*}
the affine subspace spanned by $x_1,\dots,x_n$.

\subsection{Discrete A-surfaces.}
\label{subsec:a-surfaces}

\begin{definition}[Discrete A-surface]
 A map $x: \Z^2 \to \R^3$ is a called a \emph{2-dimensional discrete A-net} or
 \emph{discrete A-surface} if for each $z \in \Z^2$ the point $x(z)$ is
 coplanar with all its neighbours.  The points $x(\Z^2)$ are synonymously
 called \emph{lattice points} or \emph{vertices} of $x$.  The line segments
 connecting adjacent lattice points $x(z)$ and $x(\tilde z)$ are called
 \emph{edges} of $x$.  A vertex together with all its neighbours is called a
 \emph{vertex star} and we call a plane supporting a vertex star a \emph{vertex
 plane}.
\label{def:anet_2d}
\end{definition}

\begin{remark}
In order to describe an A-surface $x$ with more general combinatorics than
$\Z^2$, one uses quad-graphs, i.e., strongly regular cell decompositions of
topological surfaces with all 2-cells being quadrilaterals, as domain for $x$.
\end{remark}

\paragraph{Genericity assumption.}
We assume that the A-nets are generic, i.e., elementary quadrilaterals are skew
and each vertex star defines a unique vertex plane.

\paragraph{Lelieuvre representation of A-surfaces.}
Let $m$ be any normal field to the vertex planes of a discrete A-surface $x :
\Z^2 \to \R^3$.  Planarity of vertex stars implies that the edges of $x$ can be
described as
\begin{equation}
x_i - x = \alpha^i \ m_i \times m, \quad i = 1,2.
\label{eq:pre_lelieuvre}
\end{equation}
The compatibility condition of \eqref{eq:pre_lelieuvre} implies that
\begin{equation*}
\alpha^1 \alpha^1_2 = \alpha^2 \alpha^2_1,
\end{equation*}
which guarantees the existence of a potential $\xi$ such that
\begin{equation*}
\alpha^i = \xi_i \xi.
\end{equation*}
Introducing $n = \xi m$, the description \eqref{eq:pre_lelieuvre} of edges
simplifies to
\begin{equation}
x_i - x = n_i \times n, \quad i = 1,2.
\label{eq:lel_normal_2d}
\end{equation}
The map $n : \Z^2 \to \R^3$ is called a \emph{discrete
Lelieuvre normal field} and \eqref{eq:lel_normal_2d} are called \emph{discrete
Lelieuvre formulae}.  The corresponding simplified compatibiliy condition is the
discrete Moutard-type equation
\begin{equation}
n_{12} - n = a^{12} (n_2 - n_1),
\label{eq:moutard_minus}
\end{equation}
with scalars $a^{12}$ that are called (discrete) \emph{Moutard coefficients} 
(see, e.g., \cite{Schief:1997:MoutardSuperposition,BobenkoSchief:1999:AffineSpheresDuality,
Doliwa:2001:ANetsPluecker,DoliwaNieszporskiSantini:2001:IntegrableReductionOfAnets,
Nieszporski:2002:DiscreteANets}).

Lelieuvre normals are unique up to black-white rescaling.  This
means that given a Lelieuvre representation $n$ of an A-net $x: \Z^2 \to \R^3$, one
can colour the vertices of $\Z^2$ black and white such that adjacent vertices
are of different colour and for arbitrary $\alpha \ne 0$ define
\begin{equation}
\tilde n =
\begin{cases}
\alpha n & \text{on black vertices}, \\
\frac{1}{\alpha} n & \text{on white vertices}.
\end{cases}
\label{eq:bw_rescaling}
\end{equation}
Then $\tilde n$ is another Lelieuvre representation of the same A-net $x$.
Solutions of the Moutard equation \eqref{eq:moutard_minus} are in one-to-one
correspondence with discrete A-nets modulo global translation of $x$ and black-white
rescaling of $n$.

In a fixed Lelieuvre representation $n$, there are four related Moutard coefficients
associated with each elementary quadrilateral. They correspond to different
equivalent reformulations of \eqref{eq:moutard_minus}. We identify those coefficients with
combinatorial pictures as shown in Fig.~\ref{fig:moutard_coeff}.

\begin{figure}[htb]
\begin{center}
 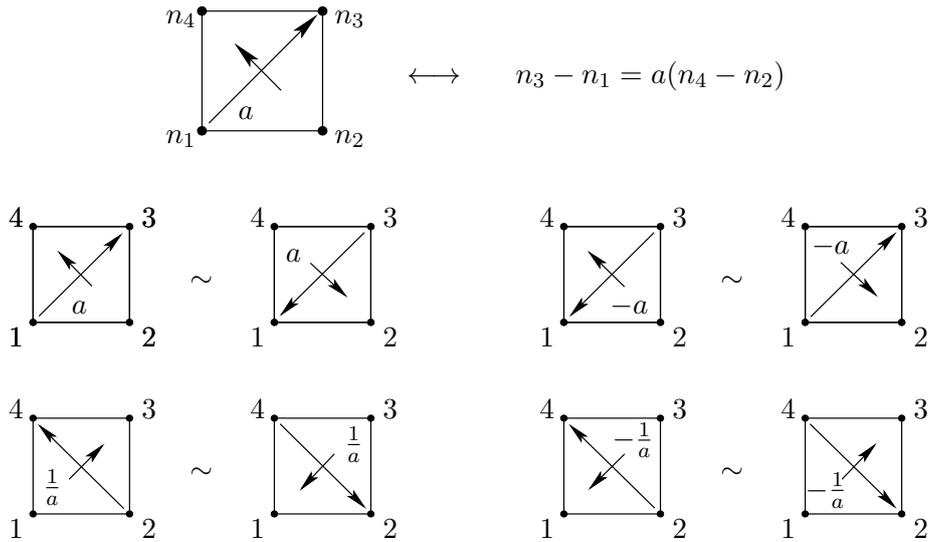 
\end{center}
\caption{Identification of Moutard coefficients with combinatorial pictures.
The long oriented diagonal represents the vector which is written as a scalar
multiple of the short oriented diagonal. Changing the orientation of one vector
corresponds to a sign change of the Moutard coefficient, $a \to -a$.
Interchanging the long and the short diagonal results in the reciprocal value,
$a \to \frac{1}{a}$.}
\label{fig:moutard_coeff}
\end{figure}

\paragraph{Invariants associated with pairs of edge-adjacent quadrilaterals of an A-net.}
Changing the Lelieuvre representation of an A-net, i.e., performing a
black-white rescaling \eqref{eq:bw_rescaling} of a given Lelieuvre normal field
$n$, changes the Moutard coefficients as indicated in
Fig.~\ref{fig:moutard_rescaling}.  Note that the sign of the Moutard
coefficient is preserved.

\begin{figure}[htb]
\begin{center}
 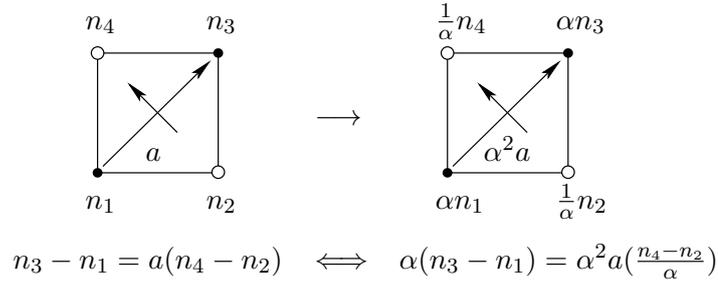 
\end{center}
\caption{Rescaling of Moutard coefficients induced by a b/w-rescaling of Lelieuvre normals.}
\label{fig:moutard_rescaling}
\end{figure}

A Moutard coefficient becomes rescaled by $\alpha^2$ or $\frac{1}{\alpha^2}$,
depending on the type (black-black or white-white) of the associated long
diagonal.  This yields algebraic invariants associated with edge-adjacent
quadrilaterals of a discrete A-net as certain products of Moutard coefficients.
One type of invariant that turns out to be crucial for our purpose is
characterized by the following

\begin{definition}[Parallel invariants]
Let $a$ and $\tilde a$ be Moutard coefficients associated with edge-adjacent
quadrilaterals of an A-net. If, in the symbolic representation of
Fig.~\ref{fig:moutard_coeff}, the coefficients $a$ and $\tilde a$ are related
by a ``parallel transport'' then the algebraic invariant $a \tilde a$ of the A-net
is called a \emph{parallel invariant} (cf. Fig.~\ref{fig:moutard_invariants}).
\label{def:parallel_invariant}
\end{definition}

\begin{figure}[htb]
\begin{center}
 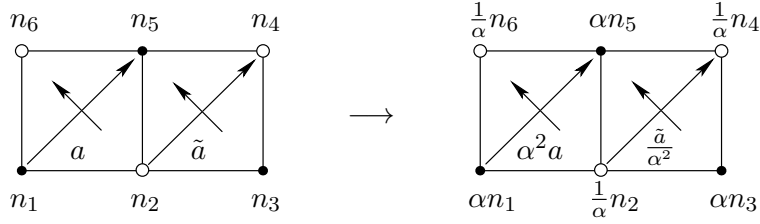 
\end{center}
\caption{Products $a \tilde a$ of Moutard coefficients that are related by a
parallel transport are algebraic invariants associated with pairs of
edge-adjacent quadrilaterals of an A-net.}
\label{fig:moutard_invariants}
\end{figure}

\begin{remark}
Moutard coefficients $a^{12}$ and $a^{12}_i, i=1,2$ associated
with edge-adjacent quadrilaterals of an A-surface $x : \Z^2 \to \R^3$ yield
parallel invariants $a^{12}a^{12}_i$.
\label{rem:parallel_invariants_aij}
\end{remark}

\paragraph{Cauchy problem for A-surfaces.}
The Lelieuvre representation provides a very convenient
description of Cauchy problems for A-surfaces.  Admissible Cauchy data for an
A-surface $x : \Z^2 \to \R^3$ are, for example,
\begin{equation}
n (\mathcal{S}^i),\ i = 1,2;
\quad
a^{12} (\Z^2);
\quad
x_0.
\label{eq:cauchy_data_a_surface}
\end{equation}
Thus, Moutard coefficients $a^{12}$ may be prescribed for the whole surface. This allows to
determine the entire Lelieuvre normal field from initial values of $n$ along
the coordinate axes $\mathcal{S}^1$ and $\mathcal{S}^2$, using
\eqref{eq:moutard_minus}. Due to \eqref{eq:lel_normal_2d}, the A-surface $x$ 
is then determined up to translation so that only one vertex $x_0$ of $x$
is needed to complete the Cauchy data.

\paragraph{Continuum limit.}
According to the classical theory, a surface $x:\R^2 \to \R^3$ parametrized
along asymptotic lines can be described by its Lelieuvre normal $n:\R^2 \to \R^3$
as stated by the Lelieuvre formulae
\begin{equation}
\partial_1 x = \partial_1 n \times n,
\quad
\partial_2 x = n \times \partial_2 n.
\label{eq:lel_normal_2d_smooth}
\end{equation}
In the continuous case, the  Lelieuvre normal is unique up to sign and does not
allow a rescaling as in the discrete case. The compatibility condition of
\eqref{eq:lel_normal_2d_smooth} is the classical Moutard equation
\begin{equation}
\partial_1 \partial_2 n = q^{12} n.
\label{eq:moutard_smooth}
\end{equation}
To obtain \eqref{eq:lel_normal_2d_smooth},\eqref{eq:moutard_smooth}
as a continuum limit of \eqref{eq:lel_normal_2d},\eqref{eq:moutard_minus}, one
first has to change the orientation of the discrete Lelieuvre normals according to,
for example, $n \to (-1)^{z_2} n.$
This converts \eqref{eq:lel_normal_2d}
and \eqref{eq:moutard_minus} into
\begin{equation}
x_1 - x = n_1 \times n,
\quad
x_2 - x = n \times n_2,
\quad
n_{12} + n = a^{12} (n_2 + n_1),
\label{eq:lel_moutard_flipped}
\end{equation}
which leads to \eqref{eq:lel_normal_2d_smooth} and \eqref{eq:moutard_smooth}
by expressing equations \eqref{eq:lel_moutard_flipped} in terms of difference quotients 
and then taking the limit.
Indeed, it is easily verified that without suitable flipping of Lelieuvre normals,
the system \eqref{eq:lel_normal_2d} does not possess a simultaneous
continuum limit.

\subsection{Higher-dimensional A-nets}
\label{subsec:anets_md}

There are two philosophically different approaches to introducing
higher-dimensional A-nets.  The first approach generalizes the incidence
geometric structure, i.e., it generalizes the idea of planar vertex stars to an
$m$-dimensional lattice with $m \ge 3$.  The second approach emphasizes the
relation between A-surfaces and their transformations. Starting with the notion
of 2-dimensional discrete A-surfaces, one imposes planarity of vertex stars
only on 2-dimensional layers of an $m$-dimensional lattice.  A multidimensional
A-net is then understood as a family of A-surfaces which are interrelated
according to the same geometric property that characterizes the surfaces
themselves.

However, it is not difficult to see that the seemingly weaker condition of
planar vertex stars in every 2-dimensional layer is equivalent to planarity of
the whole vertex stars. Indeed, it is noted that, for a generic net, three
consecutive points along a discrete coordinate line are not collinear.
Therefore, three such vertices already span the vertex plane at the middle
vertex for all 2-dimensional sublattices that contain this coordinate line.
Applying this argument repeatedly, one finds that, at a fixed vertex, all
vertex planes associated with different 2-dimensional coordinate planes through
that vertex coincide.

It is easy to verify that for a higher-dimensional lattice with all 2-cells
being quadrilaterals, planarity of vertex stars implies that the whole lattice
is contained in the 3-dimensional space that is spanned by the vertices of one
arbitrary elementary quadrilateral. Therefore, it is no restriction do define
A-nets of arbitrary dimension as maps $x: \Z^m \to \R^3$ with planar vertex stars.
Verifying the existence of 2-dimensional A-nets is straight forward since it is
not difficult to perform an iterative geometric construction of A-surfaces
which contains sufficiently many degrees of freedom at each step.  But, in the
higher-dimensional case, it is not obvious that the condition of planar vertex
stars can be imposed consistently even on a 3D lattice. Already for a single
hexahedron of a 3D lattice one obtains a closure condition. While it is clear
that one can choose 7 points associated with 7 vertices of a 3D cube such that
all 7 vertex stars are planar, the 7 points determine 4 planes that have to
intersect in a single point, i.e., the missing eighth vertex.  The existence of
this unique intersection point is guaranteed by

\begin{theorem}[Cox' theorem]
Let $P_1,P_2,P_3,P_4$ be four planes in $\RP^3$ which intersect in a point $x$.
Let $x_{ij} \in P_i \cap P_j, i \ne j$, be six points on the lines of intersection
of these planes and define four new planes $P_{ijk} = \inc{x_{ij},x_{jk},x_{ik}}$.
Then, the four planes $P_{123},P_{124},P_{134},P_{234}$ intersect in one point
$x_{1234}$ (cf. Fig.~\ref{fig:cox_gitter}).
\label{thm:cox}
\end{theorem}

For a proof of Theorem \ref{thm:cox}, see, e.g., \cite{BobenkoSuris:2008:DDGBook}.

\begin{figure}[htb]
\begin{center}
 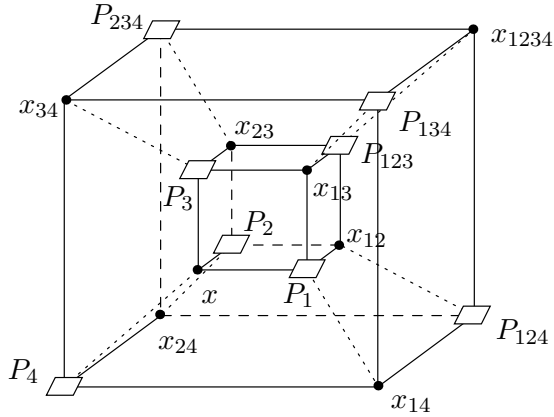 
\end{center}
\caption{A Cox configuration of points $x$ and planes $P$ that are associated
with vertices of a 4D cube.
An edge connecting a point and a plane represents incidence. There are four
planes passing through each point and each plane contains four points.  There
are four ways of interpreting a 4D Cox configuration as an elementary
hexahedron of an A-net, each corresponding to the contraction of all edges of
one coordinate direction.}
\label{fig:cox_gitter}
\end{figure}

According to the previous considerations, we say that \emph{A-nets are governed
by a 3D system} in the following sense: feasible data at seven vertices of an
elementary hexahedron determine the data at the remaining vertex uniquely,
where ``feasible data'' refers to lattice points that satisfy the condition of
planar vertex stars.\footnote{Equivalently, one one may use the dual
description of A-nets in terms of their vertex planes, where the condition
on adjacent planes is that they all have to intersect in a single point.} As a consequence,
feasible initial data along three intersecting coordinate planes of $\Z^3$
determine the net on the whole of $\Z^3$.

Moreover, the 3D system governing discrete A-nets is \emph{multidimensionally
consistent}, i.e., it can be imposed consistently on higher-dimensional lattices
$\Z^m$. The most elementary building block for this is a 4D cube as in
Fig.~\ref{fig:4D_consistency}, right.  Prescribing feasible initial data at the
11 vertices $z,z_1,\dots,z_{34}$ yields, in a first step, the data at the four
vertices $z_{123},z_{124},z_{134},z_{234}$.  Subsequently, there exist four
different ways of determining the data at $z_{1234}$ as this vertex is the
intersection of four different 3D cubes. The fact that the potentially
different data at $z_{1234}$ coincide for arbitrary feasible initial data is
called \emph{4D consistency} of the 3D system. In general,
($m+1$)D consistency of an $m$D system implies consistency in arbitrary dimension.
Multidimensional consistency is
understood as discrete integrability and we say that the 3D system governing
A-nets is \emph{discrete integrable} (see \cite{BobenkoSuris:2008:DDGBook} and 
references therein).

\begin{figure}[htb]
\begin{center}
 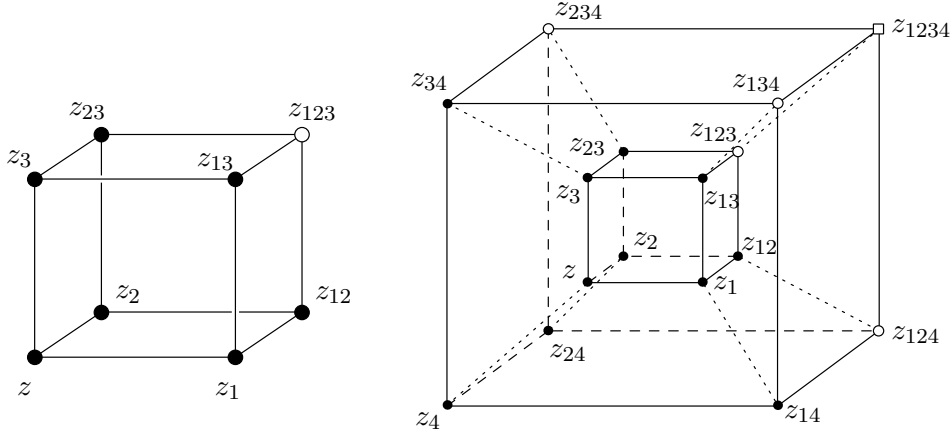 
\end{center}
\caption{ Left: The data of a discrete A-net at seven vertices
$z,z_1,\dots,z_{23}$ of a 3D cube determine the data at the eighth vertex
$z_{123}$.  Right: The 3D system describing A-nets can be imposed consistently
on a 4D cube.}
\label{fig:4D_consistency}
\end{figure}

\paragraph{Algebraic description of higher-dimensional A-nets.}
As in the 2-dimensional case, discrete A-nets $x: \Z^m \to \R^3$ can be described
by their Lelieuvre normals $n : \Z^m \to \R^3$,
\begin{equation}
x_i - x = n_i \times n, \quad i = 1,\dots,m.
\label{eq:lelieuvre_mD}
\end{equation}
In the multidimensional case, Lelieuvre normals satisfy a system of
discrete Moutard equations \cite{Schief:1997:MoutardSuperposition}
\begin{equation}
\label{eq:moutard_minus_mD}
n_{ij} - n = a^{ij} (n_j - n_i), \quad i \ne j,
\end{equation}
with skew-symmetric Moutard coefficients $a^{ji} = -a^{ij}$.
The Moutard coefficients are not independent but, since A-nets are described by
a 3D system, satisfy the following relation (compatibility condition)
on each elementary hexahedron of $\Z^m$
\begin{equation}
a^{ij}_k = - \frac{a^{ij}}{a^{ij}a^{jk} + a^{jk}a^{ki} + a^{ki}a^{ij}},
\quad i \ne j \ne k \ne i.
\label{eq:ste}
\end{equation}
The coefficients $a^{ij}$ are understood as fields on elementary quadrilaterals
of $(i,j)$-coordinate planes, where a lower index $k$ represents a shift of the
variable $a^{ij}$ in the $k$-th coordinate direction.  The multidimensional
consistency of A-nets can be stated on an algebraic level as the
multidimensional consistency of equation \eqref{eq:ste}.

The Moutard coefficients of a multidimensional A-net can be parametrized by a
function $\tau$ at vertices.  More precisely, choosing an ordering for each
pair of distinct lattice directions, for example lexicographic ordering $i <
j$, selects one type of Moutard coefficients for each coordinate plane. Then,
there exists a (non-unique) function $\tau : \Z^m \to \R$ such that the
selected Moutard coefficients can be written as
\begin{equation}
a^{ij} = \frac{\tau_i \tau_j}{\tau \tau_{ij}}.
\label{eq:moutard_param_by_tau}
\end{equation}
It is a necessary and sufficient condition for the existence of a potential
$\tau$ that satisfies \eqref{eq:moutard_param_by_tau} that for each 3D cube the
ratios of Moutard coefficients associated with opposite faces coincide.  This
is the case, since \eqref{eq:ste} implies that
\begin{equation}
\frac{a^{ij}}{a^{ij}_k} =
\frac{a^{jk}}{a^{jk}_i} =
\frac{a^{ki}}{a^{ki}_j} =
- (a^{ij}a^{jk} + a^{jk}a^{ki} + a^{ki}a^{ij}).
\label{eq:moutard_constant_ratio}
\end{equation}
Indeed, if we regard \eqref{eq:moutard_param_by_tau} as a definition of $\tau_{ij}$,
it is not difficult to see that the associated compatibility conditions
\begin{equation*}
(\tau_{ij})_k = (\tau_{ik})_j = (\tau_{jk})_i
\end{equation*}
are satisfied modulo $a^{ij}/a^{ij}_k = a^{jk}/a^{jk}_i = a^{ki}/a^{ki}_j$.

The system \eqref{eq:ste} for Moutard coefficients on a 3-dimensional
sublattice is equivalent to a discrete BKP (Miwa) equation 
\cite{Miwa:1982:HirotaDifferenceEquations} 
for $\tau$ on that sublattice.
In the lexicographic case, i.e., Moutard coefficients parametrized 
according to \eqref{eq:moutard_param_by_tau}
with $1 \le i < j \le m$,
system \eqref{eq:ste} reduces to the Miwa equation in the form
\begin{equation}
\tau \tau_{ijk} - \tau_i \tau_{jk} + \tau_j \tau_{ik} - \tau_k \tau_{ij} = 0,
\quad i < j < k.
\label{eq:bkp_lexicographic}
\end{equation}
Different sets of Moutard coefficients parametrized by $\tau$ may yield
different relative signs in \eqref{eq:bkp_lexicographic}, and, in general, one
obtains different signs for different 3-dimensional sublattices. Having
observed this, it is worth mentioning that
equation~\eqref{eq:bkp_lexicographic} is a multidimensionally consistent
equation, i.e., it can be imposed simultaneously on each 3-dimensional
sublattice of a lattice $\Z^m$ of arbitrary dimension.

\paragraph{Cauchy problem for multidimensional A-nets.}
Using \eqref{eq:ste} as evolution equation for Moutard coefficients,
it is clear how to extend Cauchy data \eqref{eq:cauchy_data_a_surface} for A-surfaces 
to Cauchy data for multidimensional A-nets $x : \Z^m \to \R^3$.
One obtains, for example,
\begin{equation}
n (\mathcal{S}^i),\ i = 1,\dots,m;
\quad
a^{ij} (\mathcal{S}^{ij}),\ 1 \le i < j \le m;
\quad
x_0.
\label{eq:cauchy_data_a_net_md}
\end{equation}

\paragraph{Continuum limit.}
For an A-net $x : \Z^m \to \R^3, m \ge 3$, it is only possible to take the
continuum limit in at most two coordinate directions.
Recalling the 2-dimensional case, for fixed $i,j$, it
is necessary to flip Lelieuvre normals such that, e.g.,
\begin{equation*}
x_i - x = n_i \times n,
\quad
x_j - x = n \times n_j
\end{equation*}
in order to obtain a continuum limit in the $(i,j)$-coordinate planes.  To this
end, start with a discrete Lelieuvre normal $n$ satisfying
\eqref{eq:lelieuvre_mD} and perform a flip of every second Lelieuvre normal in
direction $j$,
\begin{equation*}
n \to (-1)^{z_j} n.
\end{equation*}
In the continuum limit of $i$ and $j$ directions, one obtains classical
A-surfaces as $(i,j)$-coordinate planes of the resulting semi-discrete
$m$-dimensional A-net.  However, it is not possible to perform another
continuum limit in a direction $k \ne i,j$ since, either in the $(i,k)$-planes
or in the $(j,k)$-planes, the limit does not exist.  In fact, there do not
exist higher-dimensional continuous A-nets beyond A-surfaces.

\subsection{Weingarten transformations of discrete A-surfaces}
\label{subsec:wtrafos_anets}

An essential aspect of privileged surface parametrizations such as conjugate,
curvature, or asymptotic line parametrizations is the corresponding class
of transformations.  The transformations associated with a specific class of
parametrization preserve that type of parametrization.  In the context of
surfaces parametrized along asymptotic lines, the corresponding transformations are
called \emph{Weingarten transformations} \cite{Eisenhart:1960:Treatise,
SchiefRogers:2002:BaecklundDarboux}. Two continuous A-surfaces are said to
be Weingarten transforms of each other if the line connecting corresponding
points is the intersection of the tangent planes to the two surfaces at these
points.  A literal discretization of classical Weingarten transformations
(see, e.g., \cite{Doliwa:2001:ANetsPluecker,
DoliwaNieszporskiSantini:2001:IntegrableReductionOfAnets,
Nieszporski:2002:DiscreteANets}) yields
\begin{definition}[Weingarten transformation of discrete A-surfaces / Weingarten property]
Two discrete A-surfaces $f, \tilde f : \Z^2 \to \R^3$ are related by a
\emph{Weingarten transformation} if for every $z \in \Z^2$ the
line $\inc{f(z),\tilde f(z)}$ is the intersection of the vertex planes of $f$
and $\tilde f$ at the points $f(z)$ and $\tilde f(z)$, respectively.  The net
$\tilde f$ is called a \emph{Weingarten transform} of the net $f$ (and vice
versa) and $f,\tilde f$ are said to form a \emph{Weingarten pair.} We say that
the \emph{Weingarten property} is satisfied at pairs $f(z),\tilde f(z)$ of
corresponding points.
\label{def:wtrafo_anets}
\end{definition}
It is a remarkable fact that many classes of
special surface parametrizations and their associated transformations can be
unified at the discrete level.  This means that surfaces and their transformations
become indistinguishable in the sense that they are described by the same
geometric properties or, algebraically, by the same equations.
Definition~\ref{def:wtrafo_anets} clearly illustrates this unification:
Discrete A-surfaces $f,\tilde f : \Z^2 \to \R^3$ form a Weingarten pair
if and only if $F: \Z^2 \times \left\{ 0,1 \right\}$ composed of the
layers $F(\cdot,0) = f$ and $F(\cdot,1) = \tilde f$ constitutes a 3-dimensional A-net.

\section{Hyperbolic nets in terms of crisscrossed quadrilaterals}
\label{sec:hypnets_as_crisscrossed_anets}

We begin with the introduction of hyperboloids and hyperboloid patches before
explaining the notion of hyperbolic nets and recapitulating previous results.
Subsequently, we give an elementary geometric description of those results,
which will be the starting point for our discussion of transformations of
hyperbolic nets.

\subsection{Hyperboloids and hyperboloid patches}
\label{subsec:hyperboloids}

A hyperboloid in our sense is a doubly ruled quadric in $\R^3$, i.e., a
hyperboloid of one sheet or a hyperbolic paraboloid.  This terminology is
justified by projective geometry since, in $\RP^3$, there exists only one type
of doubly ruled quadric.  Referring to an affine chart which embeds $\R^3
\subset \RP^3$, one may say that a doubly ruled quadric in $\RP^3$ appears as a
hyperbolic paraboloid in the affine part $\R^3$ if it is tangent to the ideal
plane at infinity, otherwise it appears as a hyperboloid of one sheet.  In
general, if a surface contains a straight line, obviously this line is an
asymptotic line, following a constant direction of vanishing normal curvature.
Moreover, it is an essential fact of elementary projective geometry that any
three mutually skew lines determine a unique hyperboloid.

\begin{definition}[Hyperboloid patch / ruling / regulus]
A \emph{hyperboloid patch} is a (parametrized) surface patch obtained by
restricting an asymptotic line parametrization $f:D \to \R^3$ of a hyperboloid to a
closed rectangle.  We call an asymptotic line of a hyperboloid also a
\emph{ruling}. Each of the two families of rulings that cover a hyperboloid
is called a \emph{regulus}.
\label{def:hyperboloid_patch}
\end{definition}

Geometrically, a hyperboloid patch is a piece of a hyperboloid cut out along
four asymptotic lines (cf.  Fig.~\ref{fig:hyperboloid_patch}, left).  Note that
not any four asymptotic lines of a hyperboloid bound a finite hyperboloid
patch.  More precisely, four asymptotic lines, two from each regulus, divide
each other into several line segments, four of them being finite.  There exists
a patch that is bounded by those finite segments if and only if a ruling of the
hyperboloid that intersects one finite segment also intersects the opposite
finite segment (see Fig.~\ref{fig:hyperboloid_patch}).

\begin{figure}[htb]
  \begin{center}
    \includegraphics[scale=.15]{cross_extension_3.png}
    \hspace{50pt}
     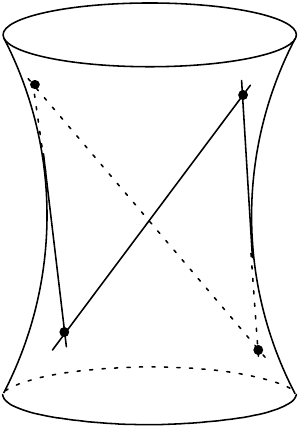 
  \end{center}
  \caption{Left: A finite hyperboloid patch.
	Right: A finite skew quadrilateral on a hyperboloid that does not bound a hyperboloid patch.}
  \label{fig:hyperboloid_patch}
\end{figure}

\begin{definition}[Adapted hyperboloids / Tangency or $C^1$-condition]
We call a hyperboloid (patch) \emph{adapted} to a skew
quadrilateral if the edges of the quadrilateral are asymptotic lines of the
hyperboloid (patch).  Moreover, we say that two hyperboloids
(hyperboloid patches) adapted to edge adjacent skew quadrilaterals satisfy the
\emph{tangency condition}, or \emph{$C^1$-condition} for short, if the tangent
planes of the two surfaces coincide along the common asymptotic line.
\label{def:adapted_and_c1}
\end{definition}

\subsection{Previous results}
\label{subsec:previous_results}

In the following, we give a brief overview of 
the work \cite{Huhnen-VenedeyRoerig:2013:hyperbolicNets}, which
introduced \emph{hyperbolic nets} as a novel discretization of smooth A-surfaces.

\begin{definition}[Hyperbolic and pre-hyperbolic nets]
\emph{Hyperbolic nets} are piecewise smooth surfaces which are composed of
hyperboloid surface patches that are adapted to the skew quadrilaterals of a
supporting A-net and satisfy the $C^1$-condition. A \emph{pre-hyperbolic net},
in turn, consists of complete adapted hyperboloids that satisfy the $C^1$-condition.
\end{definition}

\begin{remark}
Hyperbolic nets may be regarded as ``$C^1$-versions'' of smooth A-surfaces,
whereby, for convenience, we do not exclude the occurrence of two adjacent
hyperboloid patches forming a cusp. Indeed, cusps are common singularities of
pseudospherical surfaces which form an important class of A-surfaces in the
sense that these are naturally parametrized in terms of asymptotic coordinates.
Furthermore, as seen in Fig~\ref{fig:hyperbolic_net}, the
discrete A-surface which becomes extended to a (pre-)hyperbolic net may be of
more general quad-graph combinatorics than $\Z^2$.
\end{remark}

\begin{figure}[htb]
\begin{center}
\parbox{.4\textwidth}{
\includegraphics[scale=.13]{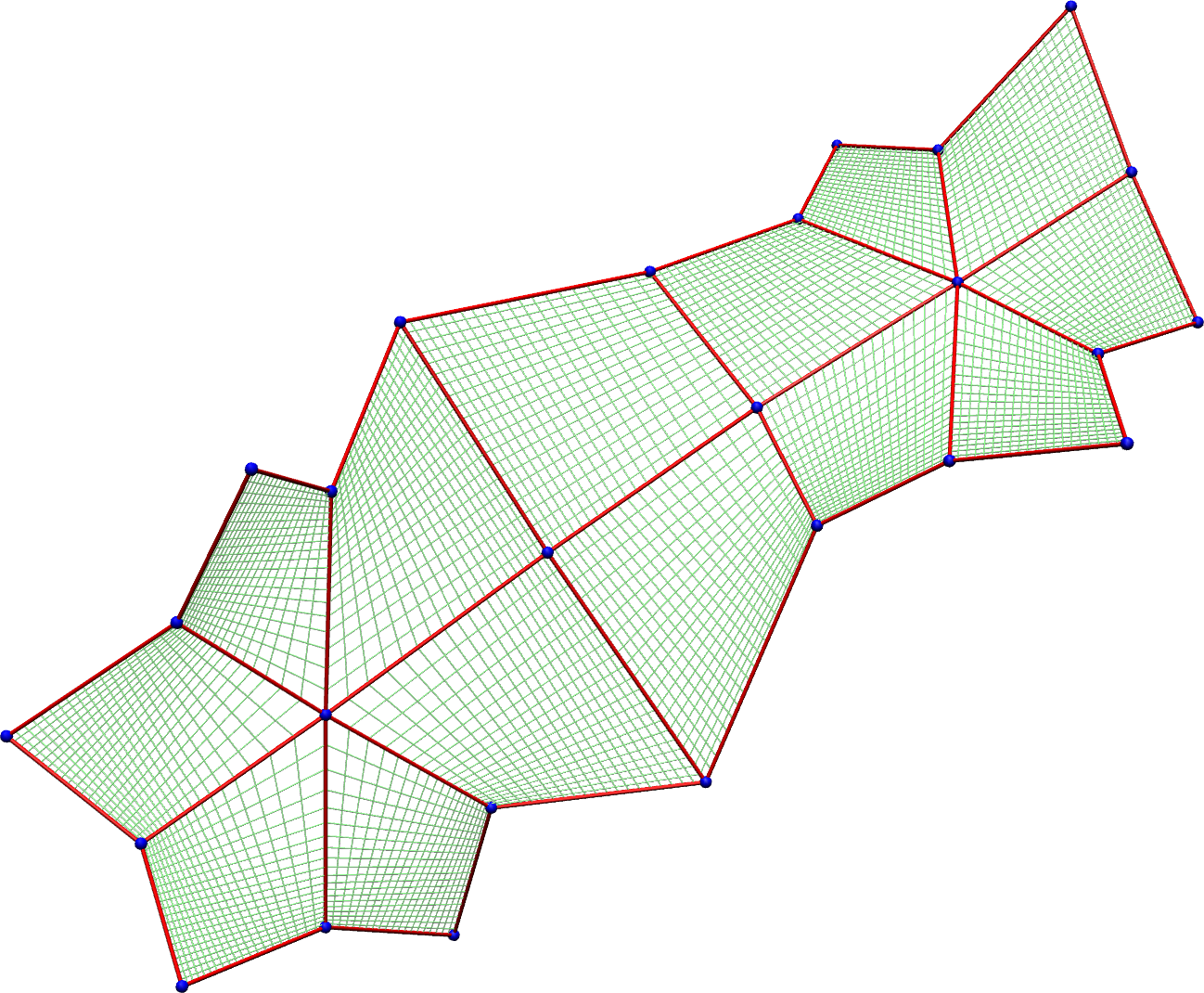}
}
\hspace{10pt}
\parbox{.4\textwidth}{
\includegraphics[scale=.12]{2x6hypnet_2.png}
}
\end{center}
\caption{Two perspectives of a hyperbolic net with two vertices of degree six.}
\label{fig:hyperbolic_net}
\end{figure}

Given two edge adjacent skew quadrilaterals and a hyperboloid adapted to one of
them, the $C^1$-condition determines a unique hyperboloid adapted to the other
quadrilateral. Accordingly, for a given A-net one may choose one inital adapted
hyperboloid and then propagate this hyperboloid to all other quadrilaterals of
the net by imposing the $C^1$-condition on adjacent hyperboloids.  The question
is whether this propagation is globally consistent, i.e., path-independent, so
that a supporting A-surface can be extended to a well-defined pre-hyperbolic
net.  It turns out that a simply connected discrete A-surface is extendable to
a pre-hyperbolic net if and only if all interior vertices are of even degree.
If we regard ``consecutive'' edges of a discrete A-surface as discrete
asymptotic lines then this is consistent with the classical theory since
asymptotic lines on continuous A-surfaces do not terminate.

Now, if we proceed from pre-hyperbolic nets to hyperbolic nets then the essential
difference is the following. In the context of pre-hyperbolic nets, the propagation of
adapted hyperboloids according to the $C^1$-condition is always possible
locally, while for hyperboloid patches this is not true. More precisely, given
two edge-adjacent skew quadrilaterals $Q,\tilde Q$ and a hyperboloid
$\mathcal{H}$ adapted to $Q$, the $C^1$-condition yields a unique hyperboloid
$\tilde{\mathcal{H}}$ adapted to $\tilde Q$. But, as explained in
Section~\ref{subsec:hyperboloids}, not every hyperboloid adapted to a skew
quadrilateral can be restricted to a patch that is bounded by the
quadrilateral. It may happen that $\mathcal{H}$ can be restricted to a patch
bounded by $Q$, but that the quadrilateral $\tilde Q$ does not bound a patch on
$\tilde{\mathcal{H}}$. In \cite{Huhnen-VenedeyRoerig:2013:hyperbolicNets} it
was shown that, assuming that $Q$ bounds a patch on $\mathcal{H}$, the
hyperboloid $\tilde{\mathcal{H}}$ can be restricted to a patch bounded by
$\tilde Q$ if and only if the quadrilaterals $Q$ and $\tilde Q$ are
\emph{equi-twisted}.  Roughly speaking, the \emph{twist of a pair of opposite
edges} of a skew quadrilateral indicates in which direction an edge turns if
it is transported into the opposite edge along the two remaining edges.
(The twists of the two pairs of edges are always complementary.)
Two edge-adjacent quadrilaterals are then called equi-twisted if the twist
of corresponding pairs of edges coincides. Accordingly, the notion of equi-twist
gives rise to \emph{equi-twisted quadrilateral strips} (cf.
Fig.~\ref{fig:equi_twisted_strip}).

\begin{figure}[htb]
\begin{center}
\includegraphics[scale=.1]{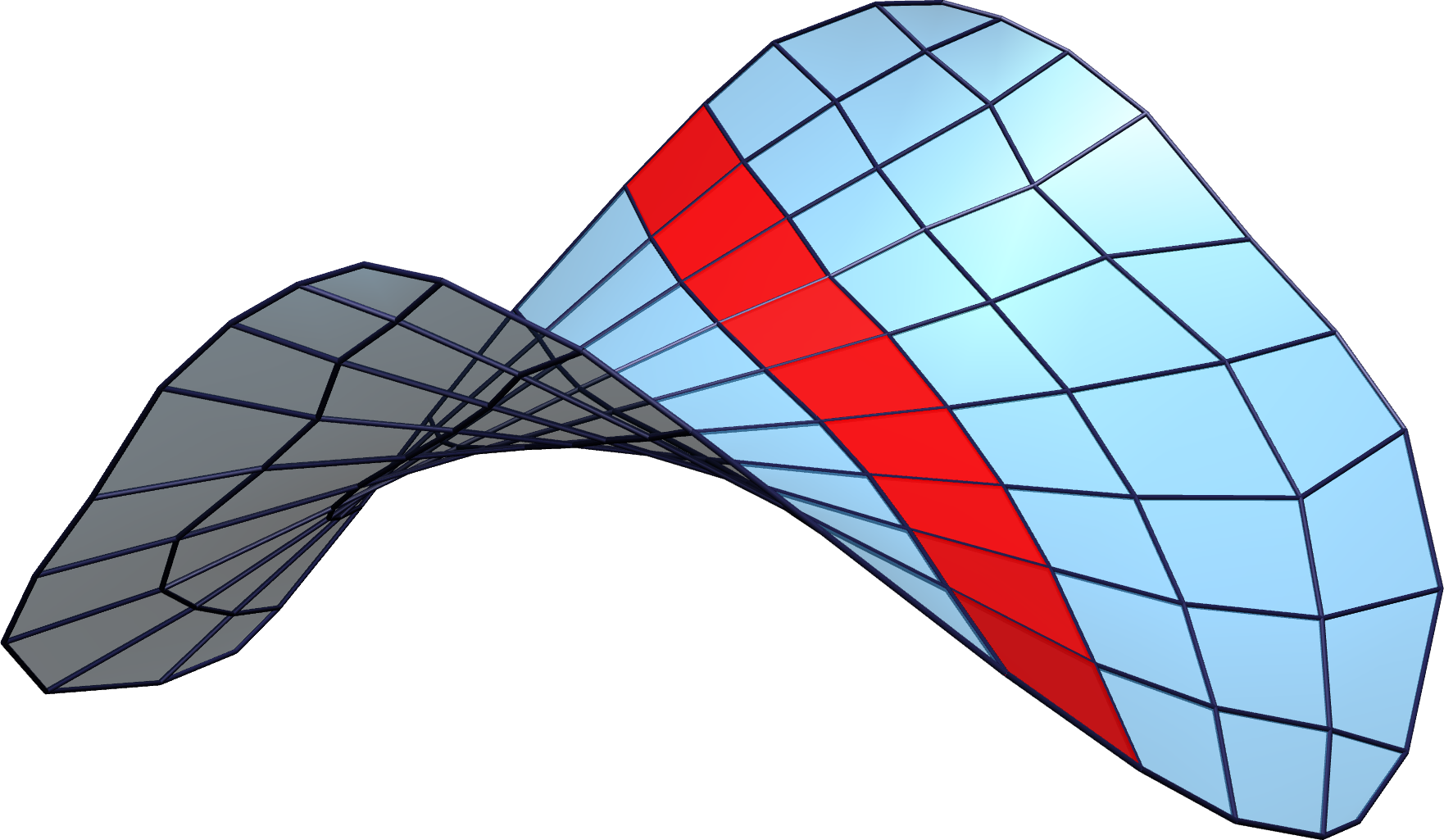}
\end{center}
\caption{An equi-twisted quadrilateral strip of an A-net.}
\label{fig:equi_twisted_strip}
\end{figure}

A discrete A-surface for which all quadrilateral strips are equi-twisted is
called equi-twisted for brevity. It is not difficult to see that, for an
equi-twisted A-surface, all interior vertices are of even degree.  It follows
that a simply connected discrete A-surface can be extended to a hyperbolic net
if and only if it is equi-twisted.  Moreover, for any skew quadrilateral there
exists a 1-parameter family of adapted hyperboloid patches.  Since the
$C^1$-propagation is unique modulo the initial patch, one obtains a 1-parameter
family of adapted hyperbolic nets for an equi-twisted A-surface.

In \cite{Huhnen-VenedeyRoerig:2013:hyperbolicNets}, A-nets, hyperboloids and
the extension of A-nets to (pre-)hyperbolic nets are all described within the
projective model of Pl\"ucker line geometry. In that model, lines in $\RP^3$
are represented by points on the Pl\"ucker quadric, which is a 4-dimensional
quadric embedded in a 5-dimensional projective space. The key feature of the
model is that two lines in $\RP^3$ intersect if and only if their
representatives in the Pl\"ucker quadric are polar with respect to the quadric.
In the Pl\"ucker setting, A-nets are discrete line congruences in the Plücker
quadric and the reguli of hyperboloids appear as non-degenerate conic sections
of the Pl\"ucker quadric with 2-planes.  For the purpose of this paper, it is
now appropriate to re-establish the theory of hyperbolic nets in affine $\R^3$
on a purely elementary geometric level.

\subsection{Hyperboloids and hyperboloid patches as skew quadrilaterals equipped
with crisscrossing lines.}
\label{subsec:crisscrossed_quads}
A skew quadrilateral in $\R^3$ consists of four points in general position that
are connected by finite edges.  A \emph{crisscrossed quadrilateral} is a skew
quadrilateral $Q$ that is equipped with a pair of intersecting lines as shown
in Fig.~\ref{fig:crossed_quadrilateral}.  A pair of such lines, which we call a
\emph{cross}, is determined by the corresponding quadruple of coplanar
intersection points with the extended edges of $Q$.  We refer to those
intersection points as \emph{cross vertices} and call the intersection point of
the two lines the \emph{centre} of the cross.
If a cross vertex is not only contained in an extended edge but
in the edge itself, we call it an \emph{internal cross vertex}.
If all four vertices of a cross are internal, we say that $Q$ is equipped with
an \emph{internal cross} as in the example of Fig.~\ref{fig:crossed_quadrilateral}.
It is noted that coplanarity of the cross vertices implies that the number
of internal cross vertices is always even.

\begin{figure}[htb]
  \begin{center}
     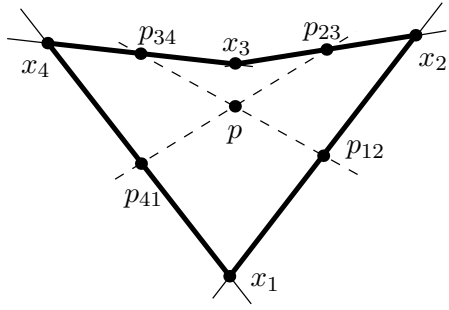 
  \end{center}
  \caption{A skew quadrilateral that is equipped with an 	internal cross.}
  \label{fig:crossed_quadrilateral}
\end{figure}

For any given skew quadrilateral $Q$, there exists a 1-parameter family of
adapted hyperboloids (cf. Definition~\ref{def:adapted_and_c1}).  Since a
hyperboloid is determined by three skew lines, extending $Q$ to a crisscrossed
quadrilateral determines a unique adapted hyperboloid, where a 2-parameter
family of crosses belongs to the same adapted hyperboloid.  Not every
hyperboloid adapted to $Q$ can be restricted to a patch bounded by $Q$, as
indicated in Fig.~\ref{fig:hyperboloid_patch}, right.  The restriction is
possible if and only if there exists an internal cross that encodes the
adapted hyperboloid. Therefore, hyperboloid patches adapted to $Q$ can be
conveniently described by internal crosses.

The extension of a skew quadrilateral $Q=(x_1,x_2,x_3,x_4)$ to a crisscrossed
quadrilateral is determined by the choice of three cross vertices on three
extended edges, say $p_{12},p_{23},p_{34}$ in the notation of
Fig.~\ref{fig:crossed_quadrilateral}.  The fourth vertex $p_{41}$ is then
obtained as the intersection of the plane $\inc{p_{12},p_{23},p_{34}}$ with the
fourth extended edge.  The three vertices $p_{12},p_{23},p_{34}$ can be
described as affine combinations of the points $x_1,x_2,x_3,x_4$. For this
purpose, introduce scalars $\rho_1,\rho_2,\rho_3,\rho_4$ at vertices such
that
\begin{equation}
p_{12} = \frac{\rho_1 x_1 + \rho_2 x_2}{\rho_1 + \rho_2}, \quad
p_{23} = \frac{\rho_2 x_2 + \rho_3 x_3}{\rho_2 + \rho_3}, \quad
p_{34} = \frac{\rho_3 x_3 + \rho_4 x_4}{\rho_3 + \rho_4}.
\label{eq:p_12_23_34_rho}
\end{equation}
The $\rho_i$ are unique up to homogeneous scaling, $\rho_i \to \alpha
\rho_i,\ \alpha \ne 0$. Ratios of these scalars correspond to ratios of
oriented lengths\footnote{ For two points $A,B$ in an affine metric space, one
can introduce the \emph{oriented length} $l(A,B) = -l(B,A)$ of the segment
$\overrightarrow{AB}$ with respect to a chosen orientation of the line
spanned by $A$ and $B$. Depending on the orientation, one says $A \le B$ or $B \le A$ and defines
\begin{equation*}
l(A,B) =
\left\{
\begin{array}{rl}
d(A,B) & \text{if } A \le B \\
-d(A,B) & \text{if } A \ge B.
\end{array}
\right.
\end{equation*}
}
that involve adjacent vertices of $Q$ and the cross vertex on the
corresponding extended edge.
For example
\begin{eqnarray*}
&& p_{12} = x_1 + \frac{\rho_2}{\rho_1 + \rho_2} (x_2 - x_1)
	= x_2 + \frac{\rho_1}{\rho_1 + \rho_2} (x_1 - x_2) \\
&\iff& x_2 = p_{12} + \frac{\rho_1}{\rho_1 + \rho_2} (x_2 - x_1).
\end{eqnarray*}
Therefore,
\begin{equation*}
l(x_1,p_{12}) = \frac{\rho_2}{\rho_1 + \rho_2} l(x_1,x_2)
\quad \text{and} \quad
l(p_{12},x_2) = \frac{\rho_1}{\rho_1 + \rho_2} l(x_1,x_2),
\end{equation*}
which yields
\begin{equation*}
\frac{l(x_1,p_{12})}{l(p_{12},x_2)} = \frac{\rho_2}{\rho_1}.
\end{equation*}
In the same manner, one obtains
\begin{equation*}
\frac{l(x_2,p_{23})}{l(p_{23},x_3)} = \frac{\rho_3}{\rho_2}, \quad
\frac{l(x_3,p_{34})}{l(p_{34},x_4)} = \frac{\rho_4}{\rho_3}.
\end{equation*}
The point $p_{41}$ lies in the plane spanned by $p_{12},p_{23}$, and $p_{34}$
if and only if it is related to $x_1$ and $x_4$ in an analogous manner, that is,
\begin{equation}
\frac{l(x_4,p_{41})}{l(p_{41},x_1)} = \frac{\rho_1}{\rho_4}
\quad \iff \quad
p_{41} = \frac{\rho_4 x_4 + \rho_1 x_1}{\rho_4 + \rho_1}.
\label{eq:p_41_rho}
\end{equation}
This assertion is a consequence of the following theorem for
$n=4$.

\begin{theorem}[Generalized	Menelaus Theorem]
Let	$x_1,...,x_{n+1}$	be $n+1$ points in general position in $\R^n$, i.e.,
$\inc{x_1,\dots,x_{n+1}} = \R^n$.  Let $p_{i,i+1}$ be some points on the lines
$\inc{x_i,x_{i+1}}$ different from $x_i, x_{i+1}$ (indices are taken modulo
$n+1$). The $n+1$ points $p_{i,i+1}$ lie in an affine hyperplane if and only if
the following relation for the quotients of directed lengths holds:
\begin{equation*}
M(x_1,p_{1,2},\dots,x_{n+1},p_{n+1,1}) := \prod_{i=1}^{n+1} \frac{l(x_i, p_{i,i+1})}{l(p_{i,i+1},x_{i+1})} = (-1)^{n+1}.
\end{equation*}
\label{thm:menelaus}
\end{theorem}

For a proof of Theorem \ref{thm:menelaus} see, e.g.,
\cite{BobenkoSuris:2008:DDGBook}.

On use of \eqref{eq:p_12_23_34_rho} and \eqref{eq:p_41_rho}, it is easily
verified that the centre $p$ of the cross is given by
\begin{equation*}
p = \inc{p_{12},p_{34}} \cap \inc{p_{23},p_{41}}
	= \frac{\rho_1 x_1 + \rho_2 x_2 + \rho_3 x_3 + \rho_4 x_4}
{\rho_1 + \rho_2 + \rho_3 + \rho_4}.
\end{equation*}

It is important to note that $\rho_1,\rho_2,\rho_3,\rho_4$ describe an internal
cross if and only if they have the same sign, which follows from the
observation that a cross vertex is internal if and only if the two
corresponding $\rho$s have the same sign.
More general, an adapted hyperboloid can be restricted to a patch that
is bounded by the supporting quadrilateral if for each pair of opposite cross vertices
either both vertices are internal, or both vertices are external.  If this
holds for one pair, it automatically holds for the other pair as well.
Therefore, the scalars $\rho_1,\rho_2,\rho_3,\rho_4$ determine a restrictable
hyperboloid if and only if $\rho_1 \rho_2 \rho_3 \rho_4 > 0$.

{Denote by
\begin{equation*}
\crr(a,b,c,d) =
	\frac{l(a,b)}{l(b,c)}	
	\frac{l(c,d)}{l(d,a)}
\end{equation*}
the cross-ratio of four collinear points and let $\tilde p_{ij}$ be four additional
cross vertices that are determined by scalars $\tilde \rho_i$.  It is a fact of
elementary projective geometry that the lines $\inc{p_{12},p_{34}}$ and
$\inc{\tilde p_{12},\tilde p_{34}}$ determine the same adapted hyperboloid if
and only if the cross-ratios
\begin{equation*}
\crr(x_1,p_{12},x_2,\tilde p_{12}) = \frac{\rho_2}{\rho_1} \cdot \frac{\tilde \rho_1}{\tilde \rho_2}
\quad \text{and} \quad
\crr(x_4,p_{34},x_3,\tilde p_{34}) = \frac{\rho_3}{\rho_4} \cdot \frac{\tilde \rho_4}{\tilde \rho_3}
\end{equation*}
coincide.} Summing up the previous considerations, we have established

\begin{lemma}
The extension of a skew quadrilateral $Q=(x_1,x_2,x_3,x_4)$ in $\R^3$ to a crisscrossed
quadrilateral (cf. Fig.~\ref{fig:crossed_quadrilateral})
corresponds to the choice of scalars $\rho_1,\rho_2,\rho_3,\rho_4$
associated with the vertices of $Q$. The vertices of the cross are then parametrized by
\begin{equation*}
p_{12} = \frac{\rho_1 x_1 + \rho_2 x_2}{\rho_1 + \rho_2}, \quad
p_{23} = \frac{\rho_2 x_2 + \rho_3 x_3}{\rho_2 + \rho_3}, \quad
p_{34} = \frac{\rho_3 x_3 + \rho_4 x_4}{\rho_3 + \rho_4}, \quad
p_{41} = \frac{\rho_4 x_4 + \rho_1 x_1}{\rho_4 + \rho_1}.
\end{equation*}
The centre $p$ of the cross is given by
\begin{equation*}
p = \frac{\rho_1 x_1 + \rho_2 x_2 + \rho_3 x_3 + \rho_4 x_4}
{\rho_1 + \rho_2 + \rho_3 + \rho_4}.
\end{equation*}
Given the cross vertices, the scalars $\rho_i$ are unique up to homogeneous
scaling, $\rho_i \to \alpha \rho_i,\ \alpha \ne 0$.  They determine an internal
cross if and only if all $\rho_i$ have the same sign.\footnote{One also obtains
a well-defined cross if exactly one $\rho_i$ equals zero so that two opposite
vertices of the quadrilateral become cross vertices. This corresponds to the limiting case of the
adapted hyperboloids degenerating to a pair of intersecting planes, each plane
being spanned by two adjacent edges of the supporting quadrilateral.}
{Adapted hyperboloids determined by scalars $\rho_i,\tilde \rho_i$ coincide if and only if
\begin{equation}
\frac{\rho_1\rho_3}{\rho_2 \rho_4}
=
\frac{\tilde \rho_1 \tilde \rho_3}{\tilde \rho_2 \tilde \rho_4}.
\label{eq:rho_tilde_rho_same_hyperboloid}
\end{equation}
}
\label{lem:cross_as_rhos}
\end{lemma}

It is convenient to identify a hyperboloid (patch) with
the 2-parameter family of corresponding (internal) crossesthat are
related according to \eqref{eq:rho_tilde_rho_same_hyperboloid}.

{
\begin{remark}
The case $\rho_1 \rho_3/\rho_2 \rho_4=1$ corresponds to adapted
hyperbolic paraboloids since these are characterized by the property that
for each regulus all rulings are parallel to a plane (see, e.g.,
\cite{KaeferboeckPottmann:2012:DiscreteAffineMinimal}).
\label{rem:hyperbolic_paraboloids_shape_parameter}
\end{remark}
}

\subsection{Hyperbolic nets as crisscrossed A-surfaces}
\label{subsec:crisscrossed_anets}

The extension of a discrete A-surface to a hyperbolic net (or a pre-hyperbolic
net) can be understood as equipping elementary quadrilaterals of the A-net with
crosses such that hyperboloid patches (or hyperboloids) associated with
edge-adjacent quadrilaterals satisfy the $C^1$-condition. Without loss of generality,
we may assume that crosses representing hyperboloids adapted to edge-adjacent 
quadrilaterals share their cross vertex on the common extended edge (cf.
Fig.~\ref{fig:c1_condition_coplanar_points}, right) and call an A-net equipped
with such crosses a \emph{crisscrossed A-net}.  According to
Lemma~\ref{lem:cross_as_rhos}, the extension of an A-net to a crisscrossed
A-net corresponds to prescribing a discrete function $\rho$ defined at lattice
points so that we may label a crisscrossed A-net by a pair $(x,\rho)$. Two
functions $\rho$ and $\tilde\rho$ describe the same crosses if and only if they
differ by a constant factor, $\tilde \rho = c \rho$. All crosses are internal,
i.e., they describe hyperboloid patches, if $\rho$ is stricly positive or
strictly negative.  The analysis of crisscrossed A-nets can be done either
geometrically via incidence theorems or algebraically in terms of the discrete
scalar function $\rho$.

\paragraph{Description of the $C^1$-condition in terms of crisscrossed
quadri\-laterals.}
Since hyperboloids are quadratic surfaces, tangency of two hyperboloids
(hyperboloid patches) along a common asymptotic line is guaranteed if the
hyperboloids (patches) are tangent at three points of this line.\footnote{This
can be verified easily, e.g., in the Pl\"ucker geometric setting as done in
\cite{Huhnen-VenedeyRoerig:2013:hyperbolicNets}.}
Now, consider two edge-adjacent quadrilaterals of an A-net as in
Fig.~\ref{fig:c1_condition_coplanar_points}.  The planarity of vertex
stars of an A-net and the definition of adapted hyperboloids implies that
any two adapted hyperboloids are tangent at the two points $b$ and $d$. In order to have
tangency along the common asymptotic line $\inc{b,d}$, it is therefore
sufficient to require tangency at the common cross vertex $y$.
This means that the planes $\inc{a,b,d}$ and $\inc{b,c,d}$ have to coincide,
i.e., the points $a, b, c, d$ have to be coplanar.

\begin{figure}[htb]
\begin{center}
 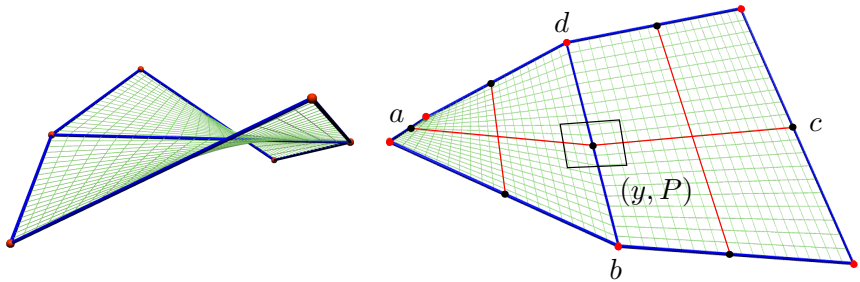 
\end{center}
\caption{Two perspectives of adjacent skew quadrilaterals of an A-net with
hyperboloid patches adapted to those quadrilaterals.}
\label{fig:c1_condition_coplanar_points}
\end{figure}

The previous considerations establish

\begin{lemma}[$C^1$-condition]
Consider two edge-adjacent quadrilaterals of a crisscrossed A-net, using the
notation of Fig.~\ref{fig:c1_condition_coplanar_points}. The corresponding
adapted hyperboloids are tangent along the common asymptotic line $\inc{b,d}$
if and only if the points $a,b,c,d$ are coplanar.  As for the corresponding
surfaces, we say that two such crisscrossed quadrilaterals, or the crosses
themselves, satisfy the \emph{(local) $C^1$-condition}.
\label{lem:local_c1}
\end{lemma}

\begin{remark}
One can interpret Lemma~\ref{lem:local_c1} as follows.  The two patches in
Fig.~\ref{fig:c1_condition_coplanar_points} satisfy the $C^1$-condition if and
only if the points $a$ and $c$ are related by a projection through the line
$\inc{b,d}$. In particular, $a$ and $c$ are independent of the common cross
vertex $y$.
\label{rem:local_c1_projection}
\end{remark}

The $C^1$-condition may be expressed algebraically in
terms of the function $\rho$ at vertices.

\begin{lemma}
For two edge-adjacent quadrilaterals of a crisscrossed A-net that are labelled
as in Fig.~\ref{fig:local_c1_rho}, the four points 
\begin{equation*}
p=\frac{\rho_1 x_1 + \rho_6 x_6}{\rho_1 + \rho_6},\ x_2,\
q=\frac{\rho_3 x_3 + \rho_4 x_4}{\rho_3 + \rho_4},\ x_5
\end{equation*}
are coplanar if and only if,
with respect to the depicted parallel invariant $a\tilde a$,
\begin{equation}
\frac{\rho_3 \rho_6}{\rho_1 \rho_4} = a \tilde a.
\label{eq:local_c1_rho}
\end{equation}
\label{lem:local_c1_rho}
\end{lemma}

\begin{figure}[htb]
\begin{center}
 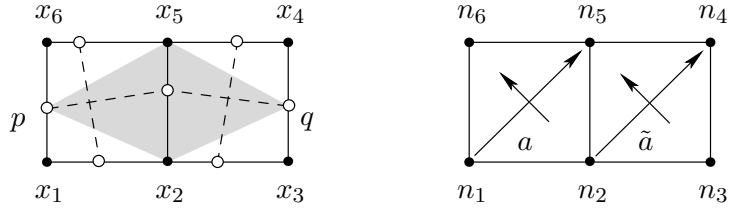 
\end{center}
\caption{
Left: Two crosses satisfy the $C^1$-condition if and only if the points
$p,x_2,q,x_5$ are coplanar.  Right: Moutard coefficients $a$ and $\tilde a$ are
related by a parallel transport and yield the parallel invariant $a \tilde a$.
}
\label{fig:local_c1_rho}
\end{figure}

\begin{remark}
Lemma~\ref{lem:local_c1_rho} relates the $C^1$-condition to the parallel
invariant $a \tilde a$ that is depicted in Fig.~\ref{fig:local_c1_rho}.
Considering  Moutard coefficients which are obtained from $a$ and $\tilde a$ by
interchanging the long and the short diagonals yields the reciprocal parallel
invariant $b \tilde b = (a \tilde a)^{-1}$ (cf.  Fig.~\ref{fig:moutard_coeff}).
Thus, in terms of the parallel invariant $b \tilde b$,
relation~\eqref{eq:local_c1_rho} adopts the form
\begin{equation*}
\frac{\rho_1 \rho_4}{\rho_3 \rho_6} = b \tilde b.
\end{equation*}
\label{rem:local_c1_rho_second_invariant}
\end{remark}

\begin{titleproof}{Proof of Lemma~\ref{lem:local_c1_rho}}
The Moutard coefficients $a$ and $\tilde a$ belong to a certain Lelieuvre representation
$n$, where
\begin{equation*}
n_5 - n_1 = a (n_6 - n_2),
\quad
n_4 - n_2 = \tilde a (n_5 - n_3).
\end{equation*}
The parallel invariant $a\tilde a$, in turn, is independent of the chosen representation
(cf. Section~\ref{subsec:a-surfaces}).

Any plane that contains the edge $[x_2,x_5]$ has a normal vector
$m$ of the form
\begin{equation*}
m = \mu n_2 + \nu n_5
	= \mu n_2 + \nu (n_1 + a (n_6 - n_2))
\iff
\langle m , x_5 - x_2 \rangle = 0
\end{equation*}
and such a plane additionally contains $p$ and $q$ if and only if 
\begin{equation}
\langle m , x_2 - p \rangle = 0 = \langle m , x_2 - q \rangle.
\label{eq:pq_in_plane}
\end{equation}
The edges of the quadrilaterals can be written as cross-products of the Lelieuvre normals
\begin{equation*}
x_2 - x_1 = n_2 \times n_1,
\quad
x_6 - x_1 = n_6 \times n_1,
\quad \text{etc.}
\end{equation*}
Therefore, the vector $x_2 - p$ may be expressed as
\begin{equation*}
x_2 - p = x_2 - x_1 - \alpha (x_6 - x_1) = n_2 \times n_1 - \alpha n_6 \times n_1,
\quad
\alpha = \frac{\rho_6}{\rho_1 + \rho_6}.
\end{equation*}
We have
\begin{eqnarray*}
\langle m , x_2 - p \rangle 
	&=& \langle \mu n_2 + \nu (n_1 + a (n_6 - n_2)) , n_2 \times n_1 - \alpha n_6 \times n_1 \rangle \\
	&=&
	\alpha (\nu a - \mu) \langle n_2 , n_6 \times n_1 \rangle
	+
	\nu a \langle n_6, n_2 \times n_1 \rangle
\end{eqnarray*}
and therefore
\begin{equation*}
\langle m , x_2 - p \rangle = 0
\iff
\nu a + \alpha (\mu - \nu a) = 0
\end{equation*}
which yields
\begin{equation*}
\mu = \frac{\alpha - 1}{\alpha} \nu a = - \frac{\rho_1}{\rho_6} \nu a.
\end{equation*}
Accordingly,
a plane through $\inc{x_2,x_5}$ with normal $m = \mu n_2 + \nu n_5$ contains the point $p$
if and only if
\begin{equation*}
m \sim \rho_6 n_5 - \rho_1 a n_2.
\end{equation*}
For reasons of symmetry, the same holds with respect to $q$ if and only if
\begin{equation*}
m \sim \tilde a \rho_4 n_5 - \rho_3 n_2.
\end{equation*}
Thus, there exists a plane through $\inc{x_2,x_5}$ that contains both $p$ and $q$
if and only if
\begin{equation*}
a \tilde a = \frac{\rho_3 \rho_6}{\rho_1 \rho_4}.
\end{equation*}
\end{titleproof}

We can use Lemmas~\ref{lem:cross_as_rhos} and \ref{lem:local_c1_rho} to
characterize those crisscrossed A-nets that are (pre-)hyperbolic nets.
\begin{proposition}
A crisscrossed A-surface $(x,\rho)$ constitutes a pre-hyperbolic net if, for
any two edge-adjacent quadrilaterals in the coordinate free notation of
Fig.~\ref{fig:local_c1_rho}, the function $\rho$ at vertices of $x$ satisfies
condition \eqref{eq:local_c1_rho}. If, additionally, $\rho$ is strictly positive
or strictly negative, all crosses are internal and therefore describe adapted
hyperboloid patches that form a hyperbolic net.
\label{prop:prehyp_net_in_terms_of_rho}
\end{proposition}

\paragraph{Geometric interpretation of parallel invariants.}
Lemma~\ref{lem:local_c1_rho} gives rise to the following geometric interpretation of
parallel invariants $a \tilde a$. With respect to the notation of
Fig.~\ref{fig:invariant_as_length_ratios}, we introduce the ratios of oriented
lengths
\begin{equation}
l_p = \frac{l(x_1,p)}{l(p,x_6)},
l_r = \frac{l(x_1,r)}{l(r,x_3)},
l_q = \frac{l(x_4,q)}{l(q,x_3)},
l_s = \frac{l(x_4,s)}{l(s,x_6)},
\label{eq:length_ratios_associated_with_arrows}
\end{equation}
which are represented in Fig.~\ref{fig:invariant_as_length_ratios} by the four arrows.
Reversing the direction of an arrow corresponds to taking the inverse of the
associated ratio.

\begin{figure}[htb]
\begin{center}
 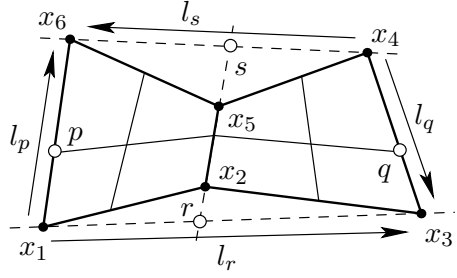 
\end{center}
\caption{Planar vertex stars imply that the line $\inc{x_2,x_5}$ intersects each of the
lines $\inc{x_1,x_3}$ and $\inc{x_4,x_6}$, which yields intersection points $r$ and $s$.}
\label{fig:invariant_as_length_ratios}
\end{figure}

Now, knowing that $p,q,r,s$ are coplanar, we can apply the generalized Menelaus Theorem and obtain 
\begin{equation}
l_p l_s^{-1} l_q l_r^{-1} = 1,
\label{eq:menelaus_c1}
\end{equation}
which yields
\begin{equation*}
a \tilde a = \frac{\rho_6}{\rho_1} \frac{\rho_3}{\rho_4} = l_p l_q = l_r l_s.
\end{equation*}
While the product $l_p l_q$ refers to the additional structure provided by the crosses,
the relation $a \tilde a = l_r l_s$ refers to the geometry of the underlying A-net only.
In particular, the parallel invariant $a \tilde a$ is positive if and only
if either both or none of the points $r,s$ are contained in the line segments $[x_1,x_3]$
and $[x_4,x_6]$ respectively.

\paragraph{Internal crosses that satisfy the $C^1$-condition.}
Lemma~\ref{lem:local_c1_rho} reveals in which case edge-adjacent skew
quadrilaterals can be equipped with internal crosses that satisfy the
$C^1$-condition. In the notation of Fig.~\ref{fig:local_c1_rho}, we assume
positive initial data $\rho_1,\rho_2,\rho_3,\rho_5,\rho_6$, which describe an
internal cross for the quadrilateral $(x_1,x_2,x_5,x_6)$ as well as an internal
cross vertex on the edge $[x_2,x_3]$. According to \eqref{eq:local_c1_rho}, the
value $\rho_4$ is then obtained as $\rho_4 = \rho_3 \rho_6/\rho_1 a
\tilde{a}$ and the resulting cross for the quadrilateral $(x_2,x_3,x_4,x_5)$ is
internal if and only if $\rho_4 > 0$, i.e., if and only if the parallel
invariant $a \tilde a$ is positive.  In
\cite{Huhnen-VenedeyRoerig:2013:hyperbolicNets} it was proven that it is
possible to equip two adjacent skew quadrilaterals of an A-net with adapted
hyperboloid patches that satisfy the $C^1$-condition if and only if the
quadrilaterals are equi-twisted (cf.  Section~\ref{subsec:previous_results}).
Thus, positivity of parallel invariants is an algebraic description of
equi-twist. Hence, we have established

\begin{lemma}
The propagation of cross vertices described in
Remark~\ref{rem:local_c1_projection} maps internal cross vertices to internal
cross vertices if and only if the quadrilaterals in question are equi-twisted,
which, in turn, is equivalent to positivity of the corresponding parallel
invariants.
\label{lem:propagation_of_internal_vertices}
\end{lemma}

\paragraph{Extension of A-surfaces to pre-hyperbolic nets.}
Locally, it is always possible to propagate a hyperboloid adapted to a skew
quadrilateral via the $C^1$-condition to an edge-adjacent quadrilateral.
If we try to extend an entire A-surface to a pre-hyperbolic net then the question
arises as to whether the propagation is consistent, i.e., path-independent.  To answer
this question, one has to examine whether the propagation along closed cycles
composed of edge-adjacent quadrilaterals is consistent.  If we restrict
ourselves to simply connected A-surfaces, a basis for those cycles is given by
the elementary cycles of quadrilaterals around single vertices and
it is sufficient to investigate whether the
propagation of hyperboloids around inner vertices of an A-surface is
consistent.  In \cite{Huhnen-VenedeyRoerig:2013:hyperbolicNets}, this was
done in terms of Pl\"ucker geometry and it was shown that the propagation
around a vertex is consistent if and only if the vertex is of even degree.
In the following, we will give an algebraic and a geometric proof of the
corresponding consistency statement for a regular vertex of degree four in the
setting of crisscrossed quadrilaterals.

\begin{lemma}
Let $Q_1,\dots,Q_4$ be four quadrilaterals of a crisscrossed A-net around a
vertex of degree four as in Fig.~\ref{fig:menelaus_consistency}.  If the 
$C^1$-condition is satisfied at three interior edges then it is also satisfied
at the fourth interior edge.
\label{lem:4_cc_quads_consistency}
\end{lemma}

\begin{figure}[htb]
\begin{center}
\includegraphics[scale=.37]{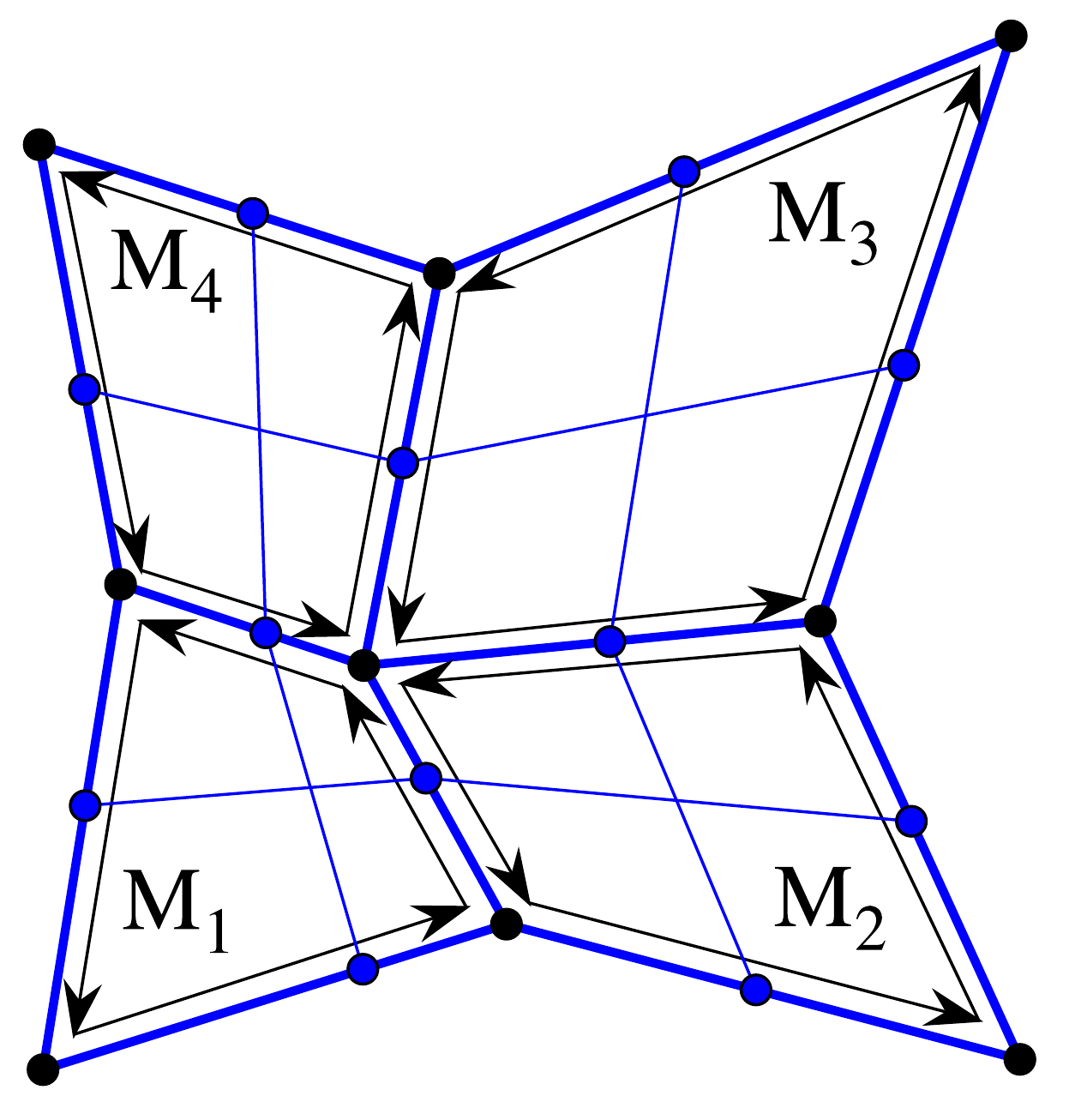}
\hspace{.2cm}
\includegraphics[scale=.37]{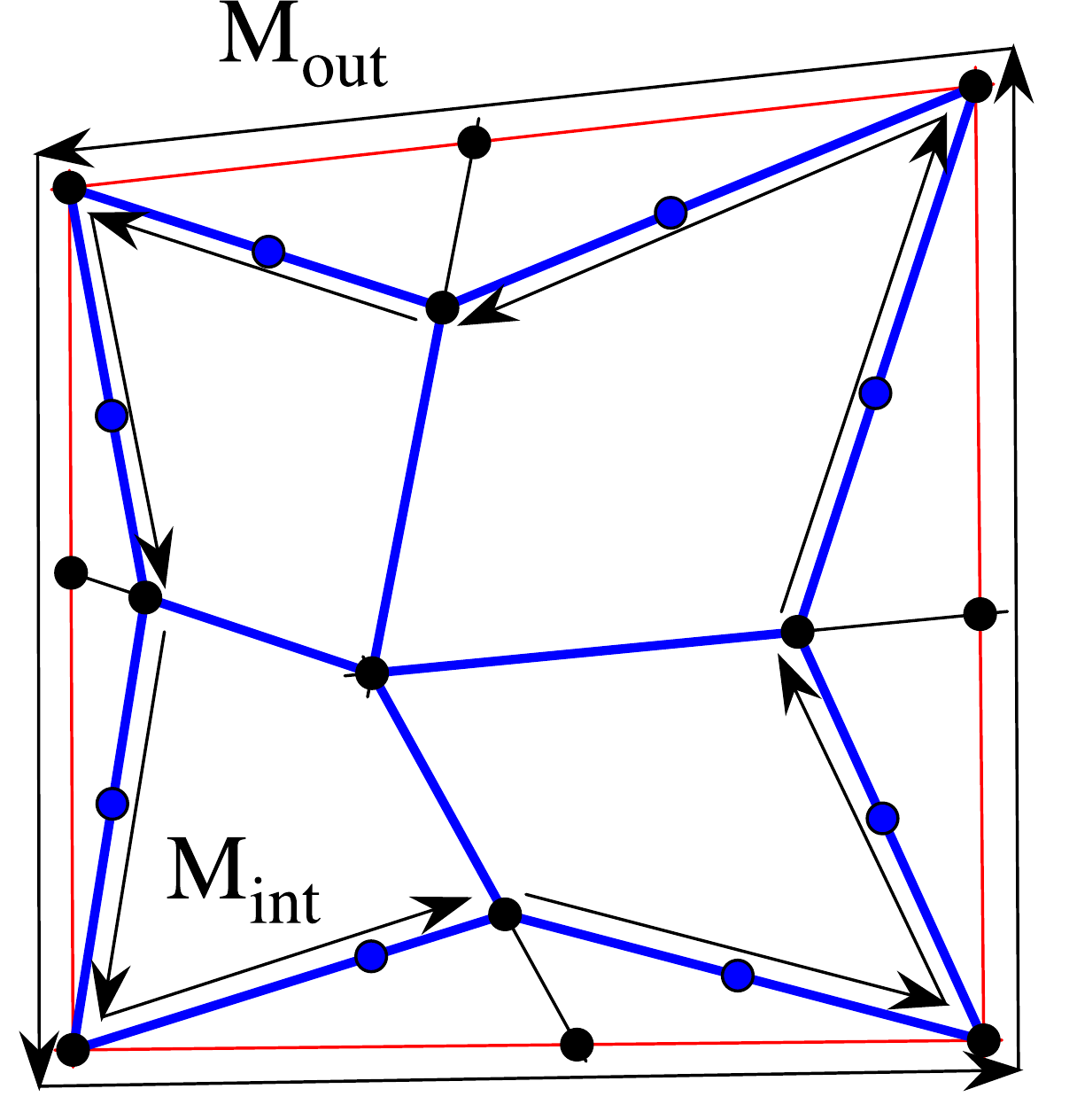}
\hspace{.2cm}
\includegraphics[scale=.37]{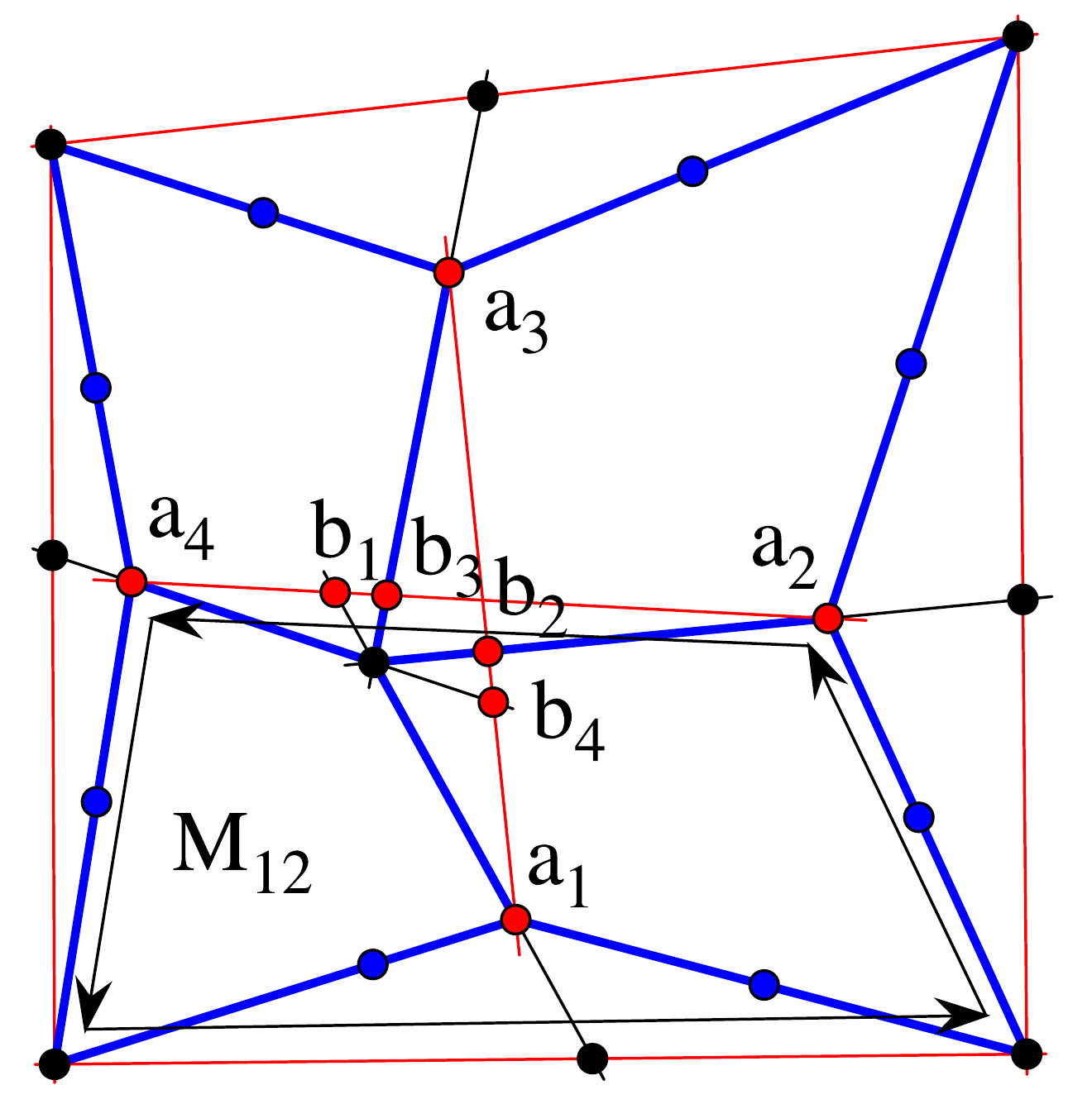}
\end{center}
\caption{Four crisscrossed quadrilaterals around a regular vertex of degree
four.  Arrows represent ratios of oriented lengths. The multiratios $M_\ast$
arise as the depicted products of ratios associated with arrows that form
oriented polygons.}
\label{fig:menelaus_consistency}
\end{figure}

\begin{titleproof}{Geometric proof of Lemma~\ref{lem:4_cc_quads_consistency}}
We will use the generalized Menelaus Theorem (Theorem~\ref{thm:menelaus})
several times.  The ratios involved are indicated by arrows in
Fig.~\ref{fig:menelaus_consistency}, analogous to the arrows in
Fig.~\ref{fig:invariant_as_length_ratios} representing the ratios
\eqref{eq:length_ratios_associated_with_arrows}. We obtain several relevant
multiratios $M_\ast$ (as defined in Theorem~\ref{thm:menelaus}) as products of
ratios that are associated with arrows that form closed polygons as depicted in
Fig.~\ref{fig:menelaus_consistency}. 

Reformulation of Lemma~\ref{lem:local_c1} using Menelaus' theorem
yields the following.  Crosses adapted to edge-adjacent quadrilaterals $i$ and
$j$ satisfy the $C^1$-condition if and only if the corresponding multiratio
$M_{ij} = 1$, where $M_{ij}$ involves those points that correspond
to $p,r,q,s$ in Fig.~\ref{fig:invariant_as_length_ratios}.
(In the context of Fig.~\ref{fig:invariant_as_length_ratios}, $M_{ij}=1$ corresponds
to \eqref{eq:menelaus_c1}.)
We will now show that
\begin{equation*}
M_{12} M_{23} M_{34} M_{41} = 1,
\end{equation*}
which proves the lemma. If we
regroup the factors of the multiratios $M_{ij}$ then we obtain
\begin{equation*}
M_{12} M_{23} M_{34} M_{41} 
	= M_{int} M_{out} 
	\frac{l(a_2,b_1)}{l(b_1,a_4)}
	\frac{l(a_3,b_2)}{l(b_2,a_1)}
	\frac{l(a_4,b_3)}{l(b_3,a_2)}
	\frac{l(a_1,b_4)}{l(b_4,a_3)}.
\end{equation*}
Since the vertices of a cross are coplanar, the generalized Menelaus Theorem
gives $M_1 = M_2 = M_3 = M_4 = 1$ and, hence,
\begin{equation*}
M_{int} = M_1 M_2 M_3 M_4 = 1.
\end{equation*}
Moreover, planarity of the vertex star of the central vertex immediately yields
$M_{out} = 1$.
It remains to show that 
\begin{equation*}
	\frac{l(a_2,b_1)}{l(b_1,a_4)}
	\frac{l(a_3,b_2)}{l(b_2,a_1)}
	\frac{l(a_4,b_3)}{l(b_3,a_2)}
	\frac{l(a_1,b_4)}{l(b_4,a_3)} = 1
	\iff
	\crr(a_2,b_1,a_4,b_3) = \crr(b_2,a_1,b_4,a_3).
\end{equation*}
This assertion is indeed true since the two quadruples of points are related by a
projection through the central vertex of the four quadrilaterals.
\end{titleproof}

\begin{titleproof}{Algebraic proof of Lemma~\ref{lem:4_cc_quads_consistency}}
According to Lemma~\ref{lem:local_c1_rho} and
Remark~\ref{rem:local_c1_rho_second_invariant}, imposing the 
$C^1$-condition on a crisscrossed A-surface $(x,\rho) : \Z^2 \to \R^3 \times
\R$ means requiring
\begin{equation*}
\frac{\rho_{11} \rho_2}{\rho \rho_{112}} = a^{12} a^{12}_1,
\qquad
\frac{\rho_{22} \rho_1}{\rho \rho_{122}} = a^{12} a^{12}_2,
\end{equation*}
which is equivalent to
\begin{equation}
\rho_{112} = \frac{\rho_{11} \rho_2}{\rho a^{12} a^{12}_1},
\qquad
\rho_{122} = \frac{\rho_{22} \rho_1}{\rho a^{12} a^{12}_2}.
\label{eq:rho_evolution_2d}
\end{equation}
It is straight forward to verify that the evolution equations
\eqref{eq:rho_evolution_2d} are compatible, i.e.,
\begin{equation}
(\rho_{112})_2 = (\rho_{122})_1
\label{eq:rho_consistency_2d}
\end{equation}
modulo \eqref{eq:rho_evolution_2d}, which proves the lemma.
\end{titleproof}

Lemmas~\ref{lem:propagation_of_internal_vertices} and \ref{lem:4_cc_quads_consistency}
imply

\begin{proposition}
A discrete A-surface $x : \Z^2 \to \R^3$ can be equipped with internal crosses
that satisfy the $C^1$-condition, i.e., it can be extended to a hyperbolic net,
if and only if all parallel invariants are positive.
\label{prop:hypnet_extension_positive_invariants}
\end{proposition}

\paragraph{Cauchy problems for hyperbolic and pre-hyperbolic nets.}
\label{par:cauchy_rho_2d}
Imposition of the $C^1$-condition on a crisscrossed A-surface $(x,\rho) : \Z^2
\to \R^3 \times \R$ yields the evolution equations \eqref{eq:rho_evolution_2d}
and sets up a 2-dimensional Cauchy problem for $\rho$ (see
Fig.~\ref{fig:cauchy_data_rho_2d}).  Given the supporting A-net, Cauchy data
for $\rho$ are obtained, for instance, by prescribing $\rho$ along coordinate
axes and on one suitable quadrilateral,
\begin{equation}
\rho(\mathcal{S}^i),\ i=1,2;
\quad \rho(\{0,1\}^2).
\label{eq:cauchy_data_rho_2d}
\end{equation}
Accordingly, Cauchy data for a pre-hyperbolic net comprise, for instance,
Cauchy data \eqref{eq:cauchy_data_a_surface} for the supporting A-net
supplemented by Cauchy data \eqref{eq:cauchy_data_rho_2d} for $\rho$.

\begin{figure}[htb]
\begin{center}
 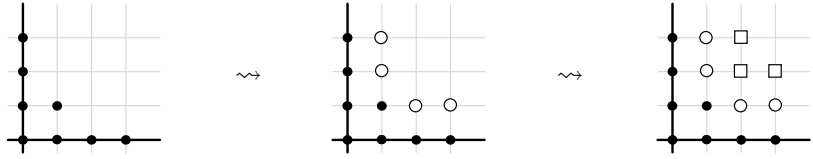 
\end{center}
\caption{Cauchy problem for the extension of a given A-surface by crosses that
satisfy the $C^1$-condition. The values of $\rho$ at the black points
constitute the Cauchy data. The values of $\rho$ at the white points and
squares are determined by the evolution equations \eqref{eq:rho_evolution_2d}.
The uniqueness of $\rho$ at the white squares is due to
\eqref{eq:rho_consistency_2d}.}
\label{fig:cauchy_data_rho_2d}
\end{figure}

In the case of hyperbolic nets, we have to ensure that the supporting A-surface
is equi-twisted and that all crosses are internal.
The description of equi-twist as positivity of parallel invariants immediately
gives rise to the Cauchy problem for equi-twisted A-surfaces as a
specialization of the Cauchy problem \eqref{eq:cauchy_data_a_surface}.  For
instance, admissible Cauchy data for an equi-twisted A-surface $x : \Z^2 \to
\R^3$ are given by
\begin{equation}
n (\mathcal{S}^i),\ i = 1,2;
\quad
a^{12} (\Z^2) > 0;
\quad
x_0.
\label{eq:cauchy_data_et_a_surface}
\end{equation}
Without loss of generality, internal crosses are described by strictly positive
$\rho$.  Accordingly, Cauchy data for a hyperbolic net consists, for instance,
of Cauchy data \eqref{eq:cauchy_data_et_a_surface} for the supporting
equi-twisted A-surface combined with positive Cauchy data
\begin{equation}
\rho(\mathcal{S}^i) > 0,\ i=1,2;
\quad \rho(\{0,1\}^2) > 0
\label{eq:cauchy_data_rho_2d_internal}
\end{equation}
for $\rho$,
which is sufficient according to Lemma~\ref{lem:propagation_of_internal_vertices}.

\section{Weingarten transformations of hyperbolic nets}
\label{sec:weingarten_trafos}

Our goal is to establish a theory of Weingarten transformations of hyperbolic
nets that extends the notion of Weingarten transformations of discrete A-surfaces.
This means that for a Weingarten pair $f = (x,\rho)$ and $\tilde f = (\tilde
x,\tilde \rho)$ of hyperbolic nets, also the supporting A-nets $x$ and $\tilde
x$ should form a Weingarten pair. The task is to relate the patches
of $\tilde f$ to the patches of $f$ in a canoncial manner.  Ideally, we would
like to ensure that corresponding patches are classical Weingarten transforms of
each other and it turns out that this is indeed possible modulo certain equi-twist 
requirements.

As alluded to in Section~\ref{subsec:anets_md}, a generalization of discrete
A-surfaces to higher-dimensional A-nets may be obtained by imposing planarity
of vertex stars on every 2-dimensional layer of an $m$-dimensional lattice, $m
\ge 3$.  If we interpret the 2-dimensional layers as discrete A-surfaces then
higher-dimensional A-nets may be regarded as families of A-surfaces related by
discrete Weingarten transformations (cf.  Section~\ref{subsec:wtrafos_anets}).
Our idea is to transfer this approach to the setting of pre-hyperbolic nets,
i.e., to impose the $C^1$-condition on 2-dimensional layers of multidimensional
crisscrossed A-nets and, in this way, derive a notion of Weingarten transforms.
However, while this approach is shown to be consistent and
yields a class of B\"acklund transformations \cite{SchiefRogers:2002:BaecklundDarboux}
for pre-hyperbolic nets, it turns out to be too flexible for our purpose.
This may be resolved by imposing additional $C^1$-conditions which lead to
B\"acklund transformations with the desired additional properties, i.e.,
Weingarten transformations of pre-hyperbolic nets.
Weingarten transformations of hyperbolic nets are induced in the case that all
crosses of two pre-hyperbolic nets forming a Weingarten pair are internal,
i.e., describe hyperboloid patches.  These Weingarten pairs may be characterized both
geometrically and algebraically in terms of equi-twist properties of multidimensional 
A-nets and the positivity of the corresponding parallel invariants respectively.

\paragraph{Higher-dimensional crisscrossed lattices and Blaschke cubes.}
Before turning to the specific setting of A-nets, we will briefly discuss the
extension of arbitrary lattices $x: \Z^m \to \R^3$.  According to
Lemma~\ref{lem:cross_as_rhos}, any such lattice $x$ can be extended to a
crisscrossed lattice by means of a function $\rho \ne 0$ at vertices.  Two
crisscrossed lattices $(x,\rho)$ and $(x,\tilde\rho)$ coincide geometrically if
and only if $\tilde\rho = c\rho$.

The description of crosses in terms of $\rho$ immediately shows that for an
elementary hexahedron of $x$, compatible crosses for all but one quadrilateral
always determine a unique compatible cross for the last quadrilateral.
Combinatorial cubes with crisscrossed skew quadrilateral faces and the property
that any two adjacent crosses meet at a point on the common supporting edge
have been investigated in detail in \cite{BlaschkeBol:1938:Gewebe}.
Accordingly, we refer to such cubes as \emph{Blaschke cubes} (cf.
Fig.~\ref{fig:blaschke_cube}).

\begin{figure}[htb]
\begin{center}
\includegraphics[scale=.19]{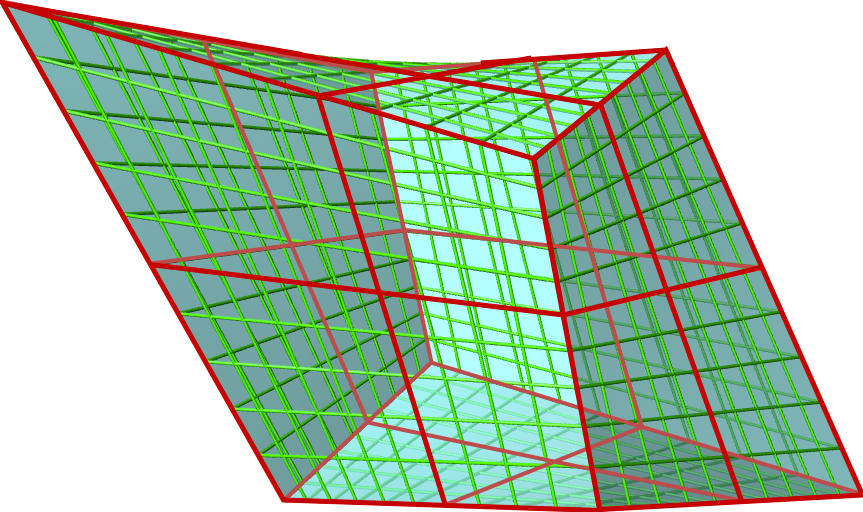}
\end{center}
\caption{A Blaschke cube is a skew cube whose faces are crisscrossed skew
quadrilaterals. If all crosses are internal, they determine an adapted
hyperboloid patch for each face, otherwise they may be understood as adapted
hyperboloids.}
\label{fig:blaschke_cube}
\end{figure}

\begin{lemma}[Blaschke cubes] \hfill
\begin{enumerate}[i)]
\item 
Let $C$ be a cube in $\RP^3$ with skew quadrilateral faces. Furthermore,
consider twelve additional points, one on each extended edge of $C$. Given five
faces of $C$, if, for each face, the four points on the edges are coplanar then
the four points associated with the remaining face are also coplanar. In other
words, the extension of five faces to crisscrossed quadrilaterals is consistent
and yields a unique cross for the sixth face.
\item 
The three lines connecting the centres of opposite crosses of a Blaschke
cube are concurrent.
\end{enumerate}
\label{lem:blaschke_cube}
\end{lemma}
\begin{proof}
Both claims follow directly from the description of crosses in terms of $\rho$
(cf. Lemma~\ref{lem:cross_as_rhos}).  To verify part ii), we label the vertices of
the cube by $x_1,\ldots, x_8$ and note that the point 
\begin{equation*}
p = \frac{\rho_1 x_1 + \rho_2 x_2 + \rho_3 x_3
	+ \rho_4 x_4 + \rho_5 x_5 + \rho_6 x_6 + \rho_7 x_7 + \rho_8 x_8}
	{\rho_1 + \rho_2 + \rho_3 + \rho_4 + \rho_5 + \rho_6 + \rho_7 + \rho_8}
\end{equation*}
is contained in each of the lines connecting opposite centres.
\end{proof}

\subsection{Imposing the $C^1$-condition on multidimensional crisscrossed A-nets}
\label{subsec:imposing_local_c1_crisscrossed_anets}

While the previous considerations show that it is always possible to extend an
A-net $x : \Z^m \to \R^3$ to a crisscrossed A-net, it is not obvious that the
$C^1$-condition can be imposed consistently on each 2-dimensional layer of $x$.
We show that this is possible by first describing the corresponding Cauchy
problem algebraically in terms of $\rho$ and then giving a geometric proof for
the consistency of the evolution equations.

\paragraph{Cauchy problem for crisscrossed A-nets that satisfy the 
$C^1$-condition in 2D coordinate planes.}
If we assume that the $C^1$-condition holds for coordinate planes of higher-dimensional
A-nets then one obtains an extension of the system \eqref{eq:rho_evolution_2d} for
A-surfaces, namely
\begin{equation}
\rho_{iij} = \frac{\rho_{ii} \rho_j}{\rho a^{ij} a^{ij}_i},
\quad i,j \in \{1,\dots,m\},\ i \ne j.
\label{eq:rho_evolution_md}
\end{equation}
For a given supporting A-net, in order to be able to propagate initial values
$\rho$ to all vertices of $\Z^m$ according to \eqref{eq:rho_evolution_md}, one
has to prescribe Cauchy data along the coordinate axes and a suitable unit
hypercube, for instance
\begin{equation}
\rho(\mathcal{S}^i),\
i=1,\dots,m;
\quad
\rho(\left\{ 0,1 \right\}^m).
\label{eq:cauchy_data_rho_md}
\end{equation}
It remains to show that for data \eqref{eq:cauchy_data_rho_md} the propagation
according to \eqref{eq:rho_evolution_md} is consistent. The 2-dimensional
compatibility conditions $(\rho_{iij})_j = (\rho_{ijj})_i$ are satisfied by
virtue of Lemma~\ref{lem:4_cc_quads_consistency}.  The elementary 3-dimensional
compatibility condition is captured by
\begin{lemma}
Imposition of the $C^1$-condition \eqref{eq:rho_evolution_md} on the four pairs of
adjacent quadrilaterals of two neighbouring cubes of an A-net, as
depicted in Fig.~\ref{fig:stacked_cubes_combinatorics}, is consistent. This
means, if we start with the Cauchy data
$\rho,\rho_j,\rho_k,\rho_{jk},\rho_{ii}$ and successively apply
\eqref{eq:rho_evolution_md} then the compatibility conditions
\begin{equation}
(\rho_{iij})_k = (\rho_{iik})_j
\label{eq:rho_consistency_3d}
\end{equation}
are satisfied so that $\rho_{iijk}$ is well defined.
\label{lem:weak_c1_two_cubes_consistency}
\end{lemma}

\begin{figure}[htb]
\begin{center}
 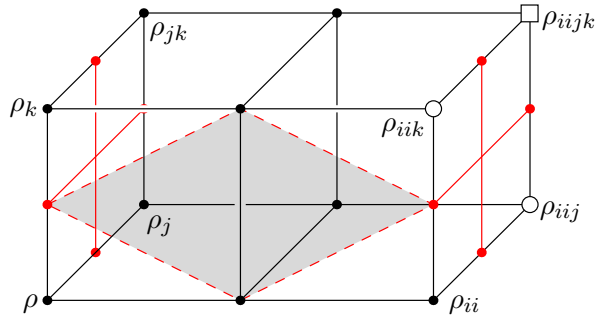 
\end{center}
\caption{The $C^1$-condition can be consistently imposed on coordinate
planes of a crisscrossed A-net locally, i.e., for two face-adjacent cubes.}
\label{fig:stacked_cubes_combinatorics}
\end{figure}
\begin{remark}
Lemma~\ref{lem:weak_c1_two_cubes_consistency} guarantees that if the 
$C^1$-condition is satisfied on three of the ``compound long faces'' in
Fig.~\ref{fig:stacked_cubes_combinatorics}, then it is also satisfied on the
remaining compound face.
\label{rem:weak_c1_on_three_sides}
\end{remark}

\begin{titleproof}{Geometric proof of Lemma~\ref{lem:weak_c1_two_cubes_consistency}}
We give a Menelaus-type proof of \eqref{eq:rho_consistency_3d},
analogous to the geometric proof of Lemma~\ref{lem:4_cc_quads_consistency}.
In terms of the notation of Fig.~\ref{fig:stacked_cubes_geometry}, left,
we define the multi-ratios
\begin{equation*}
M =
	\frac{l(x_1,p_{41})}{l(p_{41},x_4)}
	\frac{l(x_4,p_{34})}{l(p_{34},x_3)}
	\frac{l(x_3,p_{23})}{l(p_{23},x_2)}
	\frac{l(x_2,p_{12})}{l(p_{12},x_1)},
\quad
\tilde M =
	\frac{l(\tilde x_1,\tilde p_{12})}{l(\tilde p_{12},\tilde x_2)}
	\frac{l(\tilde x_2,\tilde p_{23})}{l(\tilde p_{23},\tilde x_3)}
	\frac{l(\tilde x_3,\tilde p_{34})}{l(\tilde p_{34},\tilde x_4)}
	\frac{l(\tilde x_4,\tilde p_{41})}{l(\tilde p_{41},\tilde x_1)}
\end{equation*}
and
\begin{equation*}
M_{ij} =
	\frac{l(x_i,p_{ij})}{l(p_{ij},x_j)}
	\frac{l(x_j,s_j)}{l(s_j,\tilde x_j)}
	\frac{l(\tilde x_j,\tilde p_{ij})}{l(\tilde p_{ij},\tilde x_i)}
	\frac{l(\tilde x_i,r_i)}{l(r_i,x_i)},
\quad j = i+1 \mod 4.
\end{equation*}
\begin{figure}[htb]
\begin{center}
\parbox{.47\textwidth}{ 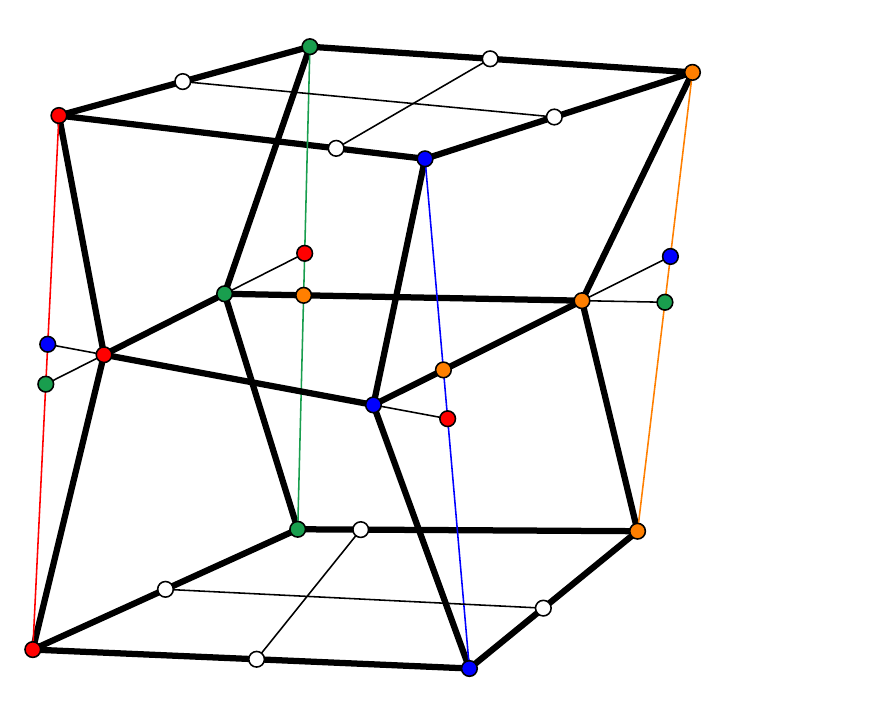 }
\quad
\parbox{.43\textwidth}{ 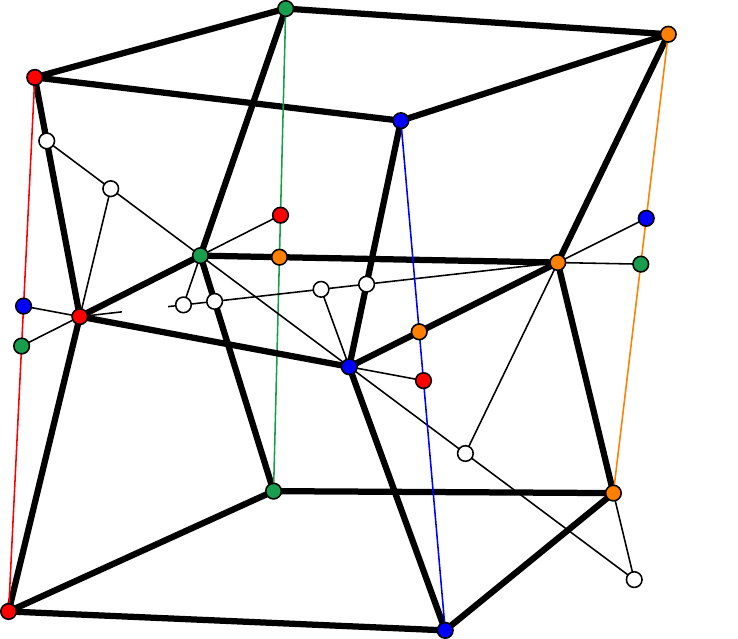 }
\end{center}
\caption{Combinatorial picture of two face-adjacent cubes of an A-net. Points
and lines of the same colour are coplanar and bold quadrilateral faces are skew.}
\label{fig:stacked_cubes_geometry}
\end{figure}
The generalized Menelaus Theorem (Theorem~\ref{thm:menelaus}) implies that
the condition $M = 1$ is equivalent to the coplanarity of $p_{12},p_{23},p_{34},p_{41}$,
and the analogous statement holds for $\tilde M = 1$.
Combining Menelaus' theorem with Lemma~\ref{lem:local_c1}, we see that
$M_{ij} = 1$ is equivalent to $p_{ij}$ and $\tilde p_{ij}$ being related
by the $C^1$-condition.
Assuming $M = M_{12} = M_{23} = M_{34} = M_{41} = 1$, we have to show
that $\tilde M = 1$, which is done by demonstrating that
\begin{equation}
M \tilde M M_{12} M_{23} M_{34} M_{41} = 1.
\label{eq:stacked_cubes_multiratios_identity}
\end{equation}
In terms of the cross-ratio of four collinear points, we see that
\begin{equation*}
M_{12} M_{23} M_{34} M_{41} =
M^{-1} \tilde M^{-1}
\prod_{i=1}^4
\crr(x_i,s_i,\tilde x_i, r_i)
\end{equation*}
and \eqref{eq:stacked_cubes_multiratios_identity} becomes
\begin{equation}
\prod_{i=1}^4 \crr(x_i,s_i,\tilde x_i, r_i) = 1.
\label{eq:stacked_cubes_multiratios_identity_crrs}
\end{equation}
To verify $\eqref{eq:stacked_cubes_multiratios_identity_crrs}$, we project the
lines $\inc{x_i,\tilde x_i}$ onto the diagonals of the middle quadrilateral
$(y_1,y_2,y_3,y_4)$ (see Fig.~\ref{fig:stacked_cubes_combinatorics}, right).
Each line $\inc{x_i,\tilde x_i}$ is projected through $y_i$ onto the diagonal
$\inc{y_j,y_k}$ that does not contain $y_i$. Note that $y_i$ and the lines $\inc{x_i,\tilde
x_i}$ and $\inc{y_j,y_k}$ are coplanar due to planarity
of the vertex star of $y_i$.  One obtains (indices taken
modulo 4)
\begin{equation}
\crr(x_i,s_i,\tilde x_i, r_i) = \crr(z_i,y_{i-1},\tilde z_i, y_{i+1}).
\label{eq:stacked_cubes_crr_projection}
\end{equation}
Using $\crr(c,b,a,d) = \crr(a,b,c,d)^{-1}$ and
\eqref{eq:stacked_cubes_crr_projection}, equation
\eqref{eq:stacked_cubes_multiratios_identity_crrs} becomes
\begin{eqnarray*}
&&\crr(z_1,y_4,\tilde z_1,y_2) \crr(z_3,y_2,\tilde z_3,y_4)
=
\crr(\tilde z_2,y_1,z_2,y_3) \crr(\tilde z_4,y_3,z_4,y_1) \\
&\iff& \crr(z_1,y_4,z_3,y_2) \crr(\tilde z_1,y_2,\tilde z_3,y_4)
=
\crr(y_1,z_4,y_3,z_2) \crr(y_1,\tilde z_2,y_3,\tilde z_4).
\end{eqnarray*}
Due to the geometry of an A-net (see Fig.~\ref{fig:cross_ratios_a_cube}) we have
\begin{equation*}
\crr(z_1,y_4,z_3,y_2) = \crr(y_1,z_4,y_3,z_2),
\quad
\crr(\tilde z_1,y_2,\tilde z_3,y_4) = \crr(y_1,\tilde z_2,y_3,\tilde z_4),
\end{equation*}
which finally proves the claim.
\end{titleproof}
\begin{figure}[htb]
\begin{center}
 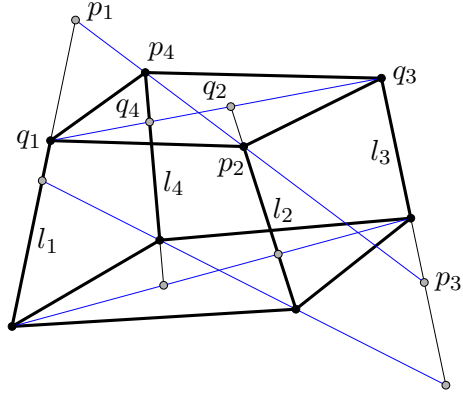 
\end{center}
\caption{
For an elementary hexahedron of a generic A-net, both the four 
extended edges $l_1,\dots,l_4$ and the four blue diagonals are skew.  Since each of those
quadrupels of lines intersects the other quadrupel, they are contained in the
two complementary reguli of a hyperboloid. Therefore, each quadrupel has a well-defined
cross-ratio, e.g., $\crr(l_1,l_2,l_3,l_4) = \crr(p_1,p_2,p_3,p_4) =
\crr(q_1,q_2,q_3,q_4)$.
}
\label{fig:cross_ratios_a_cube}
\end{figure}

We obtain

\begin{proposition}
The $C^1$-condition can be imposed consistently on 2-dimensional
layers of a crisscrossed A-net $(x,\rho) : \Z^m \to \R^3 \times \R$.
\label{prop:weak_c1_consistency}
\end{proposition}
\begin{proof}
Combining Lemma~\ref{lem:4_cc_quads_consistency} and
Remark~\ref{rem:weak_c1_on_three_sides}, it is evident that
\eqref{eq:rho_evolution_md} can be imposed consistently on $\Z^3$.  To verify
that this implies consistency in all higher dimensions, consider $z \in
\Z^{m+1}, m \ge 3$ and assume consistency on each $m$-dimensional sublattice.
The point $z$ is the intersection of $m$-dimensional sublattices
$Z_1,\dots,Z_{m+1}$, with $Z_i \cong \Z^m$.  The values of $\rho(z)$ induced by
different sublattices $Z_i$ and $Z_j$ with $\dim(Z_i \cap Z_j) \ge 2$ coincide,
because the evolution equations \eqref{eq:rho_evolution_md} for $\rho$ are
2-dimensional equations. This shows that all sublattices $Z_i$ induce the same
value for $\rho(z)$ by considering all possible pairs $Z_i, Z_j$.
\end{proof}

\paragraph{General solution $\rho$ for a crisscrossed A-net that satisfies the
$C^1$-condition in 2D coordinate planes.} We will now investigate in more
detail those functions $\rho:\Z^m \to \R$ which correspond to crosses that
satisfy the $C^1$-condition in each coordinate plane of a given A-net $x:\Z^m
\to \R^3, m \ge 3$.  In particular, in the 3-dimensional case, we derive
explicitly the general form of $\rho$ in terms of a potential $\tau$ for the
Moutard coefficients associated with $x$ and demonstrate that $\rho$ satisfies
a non-autonomous version of the discrete BKP equation
(\ref{eq:bkp_lexicographic}).  It is noted that the latter will subsequently be
shown to reduce to the standard discrete BKP equation in the context of
Weingarten pairs of (pre-)hyperbolic nets, leading to the remarkable relation
$\rho = \tau$ (see Remark~\ref{rem:bkp_geometric} and
Theorem~\ref{thm:bkp_geometric_positive}).

Consider the evolution equations \eqref{eq:rho_evolution_md} for $\rho$
satisfying the $C^1$-condition in coordinate planes
and firstly note that these equations are invariant with respect to the change
$a^{ij} \leftrightarrow a^{ji} = - a^{ij}$ of Moutard coefficients.  If we fix
one family of Moutard coefficients $a^{ij}$ for every $(i,j)$-coordinate plane
of $\Z^m$, e.g., by the condition $i < j$, then there exist potentials $\tau
: \Z^m \to \R$ that satisfy \eqref{eq:moutard_param_by_tau}.  With respect to
such a potential, the system \eqref{eq:rho_evolution_md} may be written as
\begin{equation}
\frac{\rho_j \rho_{ii}}{\rho \rho_{iij}} = 
a^{ij}a^{ij}_i =
\frac{\tau_j \tau_{ii}}{\tau \tau_{iij}},
\quad i,j \in \{1,\dots,m\},\ i \ne j.
\label{eq:weak_c1_in_tau}
\end{equation}
If we define
\begin{equation}
\xi = \frac{\rho}{\tau}
\label{eq:xi_rho_tau}
\end{equation}
then the system \eqref{eq:weak_c1_in_tau} may be stated as
\begin{equation}
\frac{\xi_j \xi_{ii}}{\xi \xi_{iij}} = 1,
\quad i,j \in \{1,\dots,m\},\ i \ne j
\label{eq:weak_c1_in_xi}
\end{equation}
and introduction of the quantities
\begin{equation}
q^{ij} = q^{ji} = \frac{\xi_i \xi_j}{\xi \xi_{ij}},
\quad
i,j \in \{1,\dots,m\},\ i \ne j
\label{eq:q_in_xi}
\end{equation}
transforms the system \eqref{eq:weak_c1_in_xi} into
\begin{equation}
q^{ij}q^{ij}_i = 1,
\quad
i,j \in \{1,\dots,m\},\ i \ne j.
\label{eq:weak_c1_in_q}
\end{equation}
The solutions $q^{ij}$ of \eqref{eq:weak_c1_in_q} may be expressed
in terms of $\xi$ according to \eqref{eq:q_in_xi} if the corresponding compatibility
conditions
\begin{equation}
\frac{q^{ij}}{q^{ij}_k}
= \frac{q^{jk}}{q^{jk}_i}
= \frac{q^{ki}}{q^{ki}_j},
\quad
i,j,k \in \{1,\dots,m\},\ i \ne j \ne k \ne i
\label{eq:q_equal_ratio_ijk}
\end{equation}
are satisfied. Thus, the general solution $\xi$ of \eqref{eq:weak_c1_in_xi}
is encoded in the compatible relations \eqref{eq:q_in_xi} with the $q^{ij}$
satisfying the coupled system \eqref{eq:weak_c1_in_q}, \eqref{eq:q_equal_ratio_ijk}.

Now, denote
\begin{equation*}
\inv{k} \alpha = \alpha^{(-1)^k}
\end{equation*}
and let $z=(z_1,\dots,z_m) \in \Z^m$ be coordinates of $\Z^m$.
The general solution of \eqref{eq:weak_c1_in_q} is given by
\begin{equation}
q^{ij} = \inv{z_i}\inv{z_j} \alpha^{ij}
	= \inv{z_i+z_j} \alpha^{ij},
\label{eq:q_ratio_inv_representation}
\end{equation}
where the functions $\alpha^{ij}$ are independent of $z_i$ and $z_j$ but
otherwise arbitrary.  In order to demonstrate how the remaining relations
\eqref{eq:q_equal_ratio_ijk} determine the functions $\alpha^{ij}$, we here
consider the case $m=3$ and merely state that an analogous procedure applies
for $m>3$.  In the 3-dimensional case, \eqref{eq:q_ratio_inv_representation}
becomes
\begin{equation}
q^{12} = \inv{z_1+z_2} \alpha^{12}(z_3), \quad
q^{23} = \inv{z_2+z_3} \alpha^{23}(z_1), \quad
q^{31} = \inv{z_3+z_1} \alpha^{31}(z_2)
\label{eq:q_in_alpha_3d}
\end{equation}
and the relations \eqref{eq:q_equal_ratio_ijk} read
\begin{equation*}
\frac{q^{12}}{q^{12}_3}
= \frac{q^{23}}{q^{23}_1}
= \frac{q^{31}}{q^{31}_2}.
\end{equation*}
Therefore,
\begin{equation}
\inv{z_1+z_2} \left( \frac{\alpha^{12}}{\alpha^{12}_3} \right)
= \inv{z_2+z_3} \left( \frac{\alpha^{23}}{\alpha^{23}_1} \right)
= \inv{z_3+z_1} \left( \frac{\alpha^{31}}{\alpha^{31}_2} \right)
\label{eq:constant_ratio_alphas}
\end{equation}
and, under the change of variables
\begin{equation}
\alpha^{12} = \inv{z_3} \tilde\alpha^{12}, \quad
\alpha^{23} = \inv{z_1} \tilde\alpha^{23}, \quad
\alpha^{31} = \inv{z_2} \tilde\alpha^{31},
\label{eq:alpha_tilde_alpha}
\end{equation}
we may write \eqref{eq:constant_ratio_alphas} as
\begin{equation*}
\inv{z_1+z_2+z_3} (\tilde\alpha^{12} \tilde\alpha^{12}_3)
= \inv{z_1+z_2+z_3} (\tilde\alpha^{23} \tilde\alpha^{23}_1)
= \inv{z_1+z_2+z_3} (\tilde\alpha^{31} \tilde\alpha^{31}_2).
\end{equation*}
Since $\tilde \alpha^{ij}$ is independent of $z_i$ and $z_j$,
we see that
\begin{equation*}
\tilde\alpha^{12} \tilde\alpha^{12}_3
= \tilde\alpha^{23} \tilde\alpha^{23}_1
= \tilde\alpha^{31} \tilde\alpha^{31}_2
= \text{const},
\end{equation*}
which allows us to introduce constants
$\beta^{12},\beta^{23},\beta^{31},\beta \in \C$, such that
the $\tilde a^{ij}$ are given by
\begin{equation}
\tilde \alpha^{12} = (\inv{z_3} (\beta^{12})^4) \beta^4, \quad
\tilde \alpha^{23} = (\inv{z_1} (\beta^{23})^4) \beta^4, \quad
\tilde \alpha^{31} = (\inv{z_2} (\beta^{31})^4) \beta^4.
\label{eq:general_solution_tilde_alpha_ij}
\end{equation}
Combining \eqref{eq:q_in_alpha_3d}, \eqref{eq:alpha_tilde_alpha}, and
\eqref{eq:general_solution_tilde_alpha_ij} yields
\begin{equation}
\begin{split}
q^{12} = \left( \inv{z_1+z_2} (\beta^{12})^4 \right) \left( \inv{z_1+z_2+z_3} \beta^4 \right),\\
q^{23} = \left( \inv{z_2+z_3} (\beta^{23})^4 \right) \left( \inv{z_1+z_2+z_3} \beta^4 \right),\\
q^{31} = \left( \inv{z_3+z_1} (\beta^{31})^4 \right) \left( \inv{z_1+z_2+z_3} \beta^4 \right).
\end{split}
\label{eq:q_in_alpha_tilde}
\end{equation}
Performing the variable substitution
\begin{equation}
\xi = \tilde \xi 
\left( \inv{z_1+z_2} (\beta^{12})^{-1} \right)
\left( \inv{z_2+z_3} (\beta^{23})^{-1} \right)
\left( \inv{z_3+z_1} (\beta^{31})^{-1} \right)
\left( \inv{z_1+z_2+z_3} \beta^{-1} \right),
\label{eq:xi_in_tilde_xi}
\end{equation}
relations \eqref{eq:q_in_xi} become
\begin{equation*}
q^{ij} = \frac{\xi_i \xi_j}{\xi \xi_{ij}}
	= \frac{\tilde \xi_i \tilde \xi_j}{\tilde \xi \tilde \xi_{ij}}
		\left( \inv{z_i+z_j} (\beta^{ij})^4 \right) \left( \inv{z_1+z_2+z_3} \beta^4 \right)
\end{equation*}
so that \eqref{eq:q_in_alpha_tilde} implies that
\begin{equation}
	\frac{\tilde \xi_i \tilde \xi_j}{\tilde \xi \tilde \xi_{ij}} = 1.
	\label{eq:xi_tilde_product_1}
\end{equation}
The general solution $\tilde \xi$ of \eqref{eq:xi_tilde_product_1} is given by
\begin{equation}
\tilde \xi(z) = f^{(1)}(z_1) f^{(2)}(z_2) f^{(3)}(z_3),
\quad f^{({i})}:\Z \to \C^*.
\label{eq:general_solution_tilde_xi}
\end{equation}
Combination of equations \eqref{eq:xi_rho_tau}, \eqref{eq:xi_in_tilde_xi},
\eqref{eq:general_solution_tilde_xi} yields the general solution $\rho$ of
\eqref{eq:weak_c1_in_tau} in terms of $\tau$, namely
\begin{equation}
\rho = (\inv{z_1+z_2} \gamma^{12}) (\inv{z_2+z_3} \gamma^{23}) (\inv{z_3+z_1} \gamma^{31})
	(\inv{z_1+z_2+z_3} \gamma) f^{(1)}(z_1)f^{(2)}(z_2)f^{(3)}(z_3) \ \tau,
\label{eq:general_solution_rho_weak_c1_3d}
\end{equation}
where the constants $\gamma,\gamma^{ij} \in \C$ and functions $f^{({i})}:\Z \to \C$
have to be chosen in such a manner that $\rho \in \R^*$ but may otherwise be arbitrary.

\begin{remark}
Equation~\eqref{eq:ste} for Moutard coefficients expressed in terms of a
potential $\tau$ that satisfies \eqref{eq:moutard_param_by_tau} becomes a
discrete BKP (dBKP) equation
\begin{equation}
\tau \tau_{123} + \varepsilon^1 \tau_1 \tau_{23} + \varepsilon^2 \tau_2 \tau_{13} + \varepsilon^3 \tau_3 \tau_{12} = 0
\label{eq:bkp_123_general}
\end{equation}
with $\varepsilon^i = \pm 1$ depending on which Moutard coefficients are chosen
to be parametrized by $\tau$.  Accordingly, $\rho$ defined by
\eqref{eq:general_solution_rho_weak_c1_3d} satisfies the corresponding 
(integrable) non-autonomous BKP-type equation
\begin{equation}
\rho \rho_{123}
+ \kappa^1 \rho_{1}\rho_{23}
+ \kappa^2 \rho_{2}\rho_{13}
+ \kappa^3 \rho_{3}\rho_{12} = 0,
\label{eq:modified_bkp_for_rho}
\end{equation}
where
\begin{gather*}
\kappa^1 = \varepsilon^1 (\inv{z_1+z_2} \gamma^{12})^4 (\inv{z_3+z_1} \gamma^{31})^4,\\
\kappa^2 = \varepsilon^2 (\inv{z_1+z_2} \gamma^{12})^4 (\inv{z_2+z_3} \gamma^{23})^4,\\
\kappa^3 = \varepsilon^3 (\inv{z_2+z_3} \gamma^{23})^4 (\inv{z_3+z_1} \gamma^{31})^4.
\end{gather*}
\end{remark}

\begin{remark}
Given a potential $\tau$ that parametrizes a certain choice of Moutard
coefficients, a function $\tilde \tau = \tilde \xi \tau$ is another potential
if and only if
\begin{equation*}
\frac{\tilde \tau_i \tilde \tau_j}{\tilde \tau \tilde \tau_{ij}}
=
\frac{\tilde \xi_i \tilde \xi_j}{\tilde \xi \tilde \xi_{ij}}
\frac{\tau_i \tau_j}{\tau \tau_{ij}}
=
\frac{\tau_i \tau_j}{\tau \tau_{ij}},
\end{equation*}
i.e., if and only if $\tilde \xi$ satisfies \eqref{eq:xi_tilde_product_1}.
Therefore, in the 3-dimensional case, the general potential $\tilde \tau$
which parametrizes the same Moutard coefficients as $\tau$ is given by
\begin{equation*}
\tilde \tau = f^{(1)}(z_1)f^{(2)}(z_2)f^{(3)}(z_3) \ \tau,
\end{equation*}
with $f^{(i)}:\Z \to \C^*$.  Taking into account black-white rescaling of the
Lelieuvre normals and the according rescaling of Moutard coefficients, one
obtains an equivalence class of Moutard coefficients that depends only on the
geometry of the underlying A-net. In terms of the potentials, any two
representatives $\tau$ and $\tilde \tau$ of this equivalence class are related
by
\begin{equation} \tilde \tau = (\inv{z_1+z_2+z_3} \gamma)
f^{(1)}(z_1)f^{(2)}(z_2)f^{(3)}(z_3) \ \tau \label{eq:general_solution_tau_3d}
\end{equation}
with $\gamma \in \C$ such that $\gamma^4 \in \R_+$.  Comparing
\eqref{eq:general_solution_tau_3d} with
\eqref{eq:general_solution_rho_weak_c1_3d} shows that, roughly speaking, in the
3-dimensional case and for a fixed potential $\tau$ the general solution $\rho$
of \eqref{eq:weak_c1_in_tau} may be decomposed into the general potential
$\tilde \tau$ that parametrizes Moutard coefficients in the equivalence class
of the given supporting A-net and a factor containing three additional
parameters $\gamma^{12},\gamma^{23},\gamma^{31}$.
\label{rem:relation_rho_tau}
\end{remark}

\paragraph{A class of canonical B\"acklund transformations for pre-hyperbolic nets.}
Proposition~\ref{prop:weak_c1_consistency} allows us to construct
\emph{B\"acklund transforms} of a pre-hyperbolic net $f=(x,\rho) : \Z^2 \to
\R^3 \times \R$ in the following sense.  We start with a Weingarten transform
$\tilde x$ of the supporting A-surface $x$. This gives a 2-layer 3D A-net $X :
\Z^2 \times \{0,1\} \to \R^3$, which is composed of the layers $X(\cdot,0) = x$
and $X(\cdot,1) = \tilde x$.  The additional data needed to specify a
B\"acklund transform $\tilde f = (\tilde x, \tilde \rho)$ of $f$ are the values
of $\tilde \rho$ at the vertices of one elementary quadrilateral $\tilde Q$ of
$\tilde x$ (cf.  Fig.~\ref{fig:cauchy_data_weak_baecklund}, left).  The
remaining values of $\tilde \rho$ are then determined by the $C^1$-condition
imposed on vertical layers, which implies that the $C^1$-condition is satisfied
for the resulting crisscrossed A-net $\tilde f$ (cf.
Remark~\ref{rem:weak_c1_on_three_sides}).  Equivalently, in geometric terms,
the Cauchy data needed to specify a B\"acklund transformation are cross
vertices on the ``vertical edges'' incident to the four vertices of one
elementary square $Q$ of $x$, which determines a unique Blaschke cube by virtue
of Lemma~\ref{lem:blaschke_cube} (cf.
Fig.~\ref{fig:cauchy_data_weak_baecklund}, right). Thus, we may say that a
B\"acklund transform of $f$ is determined by a Weingarten transform $\tilde x$
of $x$ and the extension of one cube of $X$ to a Blaschke cube.

\begin{figure}[htb]
\begin{center}
 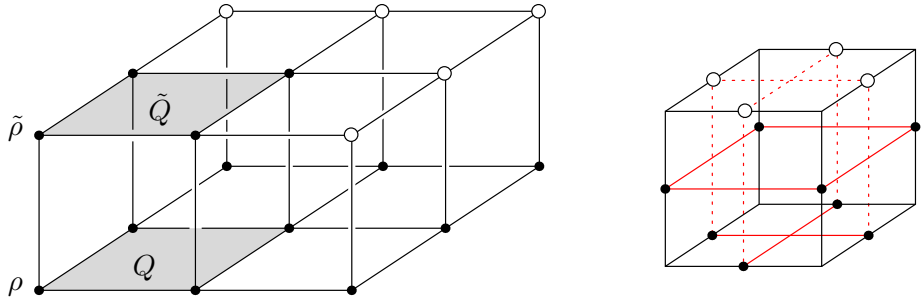 
\end{center}
\caption{ Left: The values of $\rho$ and $\tilde \rho$ at the black vertices
determine the values of $\tilde \rho$ at the white vertices by imposing the
$C^1$-condition on every 2-dimensional layer.  Right: The black cross
vertices of a Blaschke cube determine the remaining white vertices (cf.
Lemma~\ref{lem:blaschke_cube}).  }
\label{fig:cauchy_data_weak_baecklund}
\end{figure}

Since Cauchy data defining a B\"acklund transformation consist of the values
of $\tilde \rho$ at the four vertices of the quadrilateral $\tilde Q$, the
corresponding cross adapted to $\tilde Q$ can be chosen independently from the
cross adapted to $Q$. Thus, in general, the induced hyperboloids $h$ and
$\tilde h$ adapted to $Q$ and $\tilde Q$ do not form a classical Weingarten pair.

\begin{remark}
Proposition~\ref{prop:weak_c1_consistency} may be exploited to impose the
$C^1$-condition on all coordinate surfaces of a 3-layer crisscrossed A-net
$(x,\rho) : \Z^2 \times \left\{ 0,1,2 \right\} \to \R^3 \times \R$. In this
case, given initial data $\rho(\mathcal{S}^{12})$ that satisfy the
$C^1$-condition, $\rho_{33}$ already determines $\rho$ in the shifted
coordinate plane $\mathcal{S}^{12}_{33}$ by means of application of the $C^1$-condition
in all $(1,3)$- and $(2,3)$-coordinate planes.  However, 
$\rho_3,\rho_{13},\rho_{123},\rho_{23}$ can still be chosen independently as
described above. In this way, we obtain three pre-hyperbolic nets
$f^{(i)}=(x,\rho)(\cdot,i), i = 0,1,2$ such that both $(f^{(0)},f^{(1)})$ and
$(f^{(1)},f^{(2)})$ form a B\"acklund pair but also the nets
$f^{(0)},f^{(2)}$ are related in a particular manner. Furthermore, we may apply this
construction to an A-net $x$ consisting
of arbitrarily many layers and construct a special sequence
$f^{(0)},f^{(1)},f^{(2)},\cdots$ of B\"acklund transforms of pre-hyperbolic
nets that are adapted to $x$. To this end, we choose the first B\"acklund
transform $f^{(1)}$ of a given pre-hyperbolic net $f^{(0)}$ generically, while
all further B\"acklund transforms are determined uniquely (up to homogeneous
rescaling of $\rho$ in each ``horizontal'' layer) by the two initial nets.
By construction, the corresponding function $\rho$ obeys the BKP-type
equation \eqref{eq:modified_bkp_for_rho}. In this connection, it is observed
that if the layers $f^{(i)}$ are merely related by B\"acklund transformations
so that the $C^1$-condition is not necessarily satisfied on all ``vertical'' coordinate
surfaces then the expression \eqref{eq:general_solution_rho_weak_c1_3d} for
$\rho$ is still valid but $\gamma^{23},\gamma^{31}$ and $\gamma$ are now
arbitrary functions of $z_3$ and $\rho$ is governed by a slight generalization
of \eqref{eq:modified_bkp_for_rho}, which, generically, also depends on the
coefficient $\gamma$.
\label{rem:twofold_baecklund_trafo}
\end{remark}

\subsection{The notion of Weingarten transformations}
\label{subsec:weingarten_trafos}

We begin with a characterization of crosses adapted to opposite
faces of an elementary hexahedron of an A-net such that the corresponding
hyperboloids form a classical Weingarten pair.

\paragraph{Weingarten cubes.}
In the following, we use the term \emph{A-cube}
for an elementary hexahedron of an A-net, i.e., a cube with skew quadrilateral
faces and planar vertex stars.  We will show that for an A-cube with a
hyperboloid $h$ adapted to one face $Q$, there exists a unique hyperboloid
$\tilde h$ adapted to the opposite face $\tilde Q$, such that $h,\tilde h$ constitutes a
Weingarten pair.  In other words, there exists a unique Weingarten transformation
$T$ such that $\tilde h = T(h)$ is a hyperboloid adapted to $\tilde Q$.  The
geometric characterization of corresponding points $y \in h$ and $\tilde y =
T(h) \in \tilde h$ is that the line connecting $y$ and $\tilde y$ is the
intersection of the tangent planes to $h$ and $\tilde h$ in $y$ and $\tilde y$
respectively. We refer to this by saying that the points $y$ and $\tilde y$
satisfy the \emph{Weingarten property} (cf. Definition~\ref{def:wtrafo_anets}).

Consider an A-cube with crosses attached to opposite faces, labelled as in
Fig.~\ref{fig:weingarten_cube}, which determine hyperboloids $h$ and $\tilde
h$ adapted to the bottom and top faces respectively.

\begin{figure}[htb]
\begin{center}
 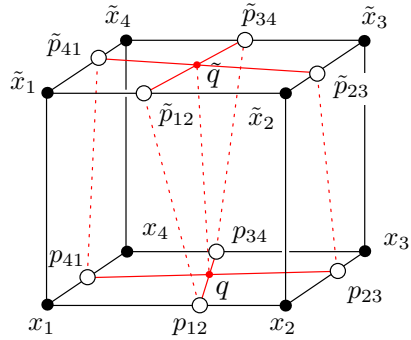 
\end{center}
\caption{Combinatorics of an A-cube with crosses adapted to two opposite faces.}
\label{fig:weingarten_cube}
\end{figure}

It is noted that, for any pair of adapted hyperboloids, the Weingarten property is
automatically satisfied at corresponding vertices $x_i,\tilde x_i$ due to the
geometry of A-cubes, i.e., since vertex stars are planar.  Now, we assume that
the crosses in Fig.~\ref{fig:weingarten_cube} define hyperboloids $h, \tilde
h$, which are related by a classical Weingarten transformation $T$.  Since
vertices $x_i,\tilde x_i$ are corresponding points and Weingarten
transformations preserve asymptotic lines, $T$ maps the asymptotic line
$\inc{x_i,x_{i+1}}$ (indices taken modulo 4) of $h$ to the asymptotic line
$\inc{\tilde x_i,\tilde x_{i+1}}$ of $\tilde h$, that is,
\begin{equation*}
T(\inc{x_i,x_{i+1}}) = \inc{\tilde x_i,\tilde x_{i+1}}.
\end{equation*}
In particular, each point $y \in \inc{x_i,x_{i+1}}$ has a corresponding point
$T(y) = \tilde y \in \inc{\tilde x_i,\tilde x_{i+1}}$.
By definition, two corresponding cross vertices $p_{ij}, \tilde p_{ij}$
satisfy the Weingarten property if and only if
\begin{equation}
\inc{p_{ij},\tilde p_{ij}} 
	= \inc{x_{i},x_{j},p_{kl}} 
			\cap \inc{\tilde x_{i},\tilde x_{j},\tilde p_{kl}},
\quad (j,k,l) = (i+1,i+2,i+3).
\label{eq:weingarten_property_cross_vertices}
\end{equation}
We make the crucial observation that
\eqref{eq:weingarten_property_cross_vertices} is equivalent to each of the
quadruples of points $(\tilde p_{ij},x_{i},x_{j},p_{kl})$ and $(p_{ij},\tilde
x_{i},\tilde x_{j},\tilde p_{kl})$ being coplanar, i.e., pairs $(p_{ij},\tilde
p_{kl})$ and $(\tilde p_{ij},p_{kl})$ of diagonally opposite cross vertices
being related by the $C^1$-condition with respect to both of the edges
$\inc{x_i,x_j}$ and $\inc{\tilde x_i,\tilde x_j}$ (see
Lemma~\ref{lem:local_c1}). 

Without loss of generality, we may assume that $T(p_{ij}) = \tilde p_{ij}$ for
one pair of corresponding cross vertices which, in turn, implies that the
opposite cross vertices $p_{kl}, \tilde p_{kl}$ also have to satisfy $T(p_{kl})
= \tilde p_{kl}$.  This follows from the fact that Weingarten transformations
preserve asymptotic lines, that is, $T(p_{ij}) = \tilde p_{ij}$ implies
$T(\inc{p_{ij},p_{kl}}) = \inc{\tilde p_{ij}, \tilde p_{kl}}$, and we may
conclude that
\begin{equation*}
\begin{split}
T(p_{kl})
	&= T(\inc{p_{ij},p_{kl}} \cap \inc{x_k,x_l})\\
  &= T(\inc{p_{ij},p_{kl}}) \cap T(\inc{x_k,x_l})
  = \inc{\tilde p_{ij},\tilde p_{kl}} \cap \inc{\tilde x_k,\tilde x_l}
	= \tilde p_{kl}.
\end{split}
\end{equation*}
The fact that it is actually possible to have simultaneously $T(p_{ij}) = \tilde
p_{ij}$ and $T(p_{kl}) = \tilde p_{kl}$ is due to the following important
\begin{lemma}[$C^1$-identity]
Consider diagonally opposite cross vertices $p_{ij}$ and $\tilde p_{kl}$ of an
A-cube equipped with crosses as in Fig.~\ref{fig:weingarten_cube}.  The points
$p_{ij}, \tilde x_i, \tilde x_j, \tilde p_{kl}$ are coplanar if and only if the
points $p_{ij}, x_k, x_l, \tilde p_{kl}$ are coplanar.
\label{lem:c1_identity}
\end{lemma}

\begin{remark}
One may interpret Lemma~\ref{lem:c1_identity} as follows: The point $\tilde
p_{kl}$ is the projection of $p_{ij}$ through both the line $\inc{\tilde
x_i,\tilde x_j}$ and the line $\inc{x_k,x_l}$
onto the line $\inc{\tilde x_k,\tilde x_l}$.
\label{rem:c1_identity_projection}
\end{remark}

\begin{titleproof}{Proof of Lemma~\ref{lem:c1_identity}}
The most elegant geometric proof of this lemma is via Cox' theorem
(Theorem~\ref{thm:cox}), applied in a suitable way
\cite{King:PrivateCommunication}.  Fig.~\ref{fig:cox_deformation}, left, shows
the initial A-cube, where a point and a plane are incident if they are
associated with adjacent vertices.  Given the crosses that are depicted in
Fig.~\ref{fig:weingarten_cube}, the initial A-cube may be modified by replacing
points $\tilde x_1$ and $x_4$ by $\tilde p_{12} \in \inc{\tilde x_1,\tilde
x_2}$ and $p_{34} \in \inc{x_3,x_4}$, and replacing planes $P_1$ and $\tilde
P_4$ by $E_1 = \inc{p_{34},x_1,x_2}$ and $\tilde E_4 = \inc{p_{34},\tilde
x_3,\tilde x_4}$. According to Cox' theorem, $\tilde p_{12} \in E_1$ if and
only if $\tilde p_{12} \in \tilde E_4$. 
\end{titleproof}

\begin{figure}[htb]
\begin{center}
 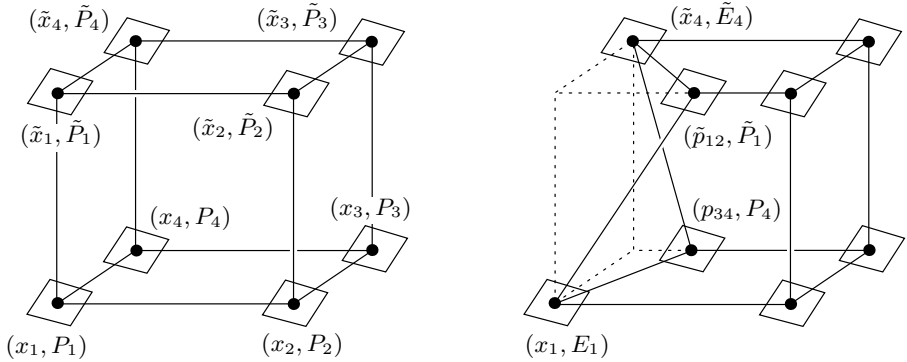 
\end{center}
\caption{Cox deformation of an A-cube.}
\label{fig:cox_deformation}
\end{figure}

If we start with the A-cube in Fig~\ref{fig:cox_deformation}, left and wish
to make sure that the configuration in Fig.~\ref{fig:cox_deformation}, right 
likewise constitutes an A-cube then we may choose either
$\tilde p_{12} \in \inc{\tilde x_1,\tilde x_2}$ or $p_{34} \in \inc{x_3,x_4}$.
Selecting, for example, $p_{34} \in \inc{x_3,x_4}$ yields the planes $E_1 =
\inc{p_{34},x_1,x_2}$ and $\tilde E_4 = \inc{p_{34},\tilde x_3,\tilde x_4}$.
By virtue of Cox' theorem, the points $\inc{\tilde x_1,\tilde x_2} \cap E_1$
and $\inc{\tilde x_1,\tilde x_2} \cap \tilde E_4$ coincide and define $\tilde
p_{12}$. We refer to this construction as a \emph{Cox deformation} of the original A-cube.

Remark~\ref{rem:c1_identity_projection} elucidates the propagation of cross
vertices induced by imposition of
\eqref{eq:weingarten_property_cross_vertices}.  The following
theorem shows that, for any given initial
cross $(p_{12},p_{23},p_{34},p_{41})$, this propagation indeed yields a cross
$(\tilde p_{12},\tilde p_{23},\tilde p_{34},\tilde p_{41})$ and, moreover, that
the Weingarten property is satisfied for the cross centres $q$ and $\tilde q$.
Having established this, it will be easy to demonstrate that
hyperboloids given by crosses that satisfy
$\eqref{eq:weingarten_property_cross_vertices}$ are indeed related by a classical
Weingarten transformation (Corollary~\ref{cor:w_cubes_existence}).

\begin{theorem}
Consider an A-cube $(x_1,\dots,x_4,\tilde x_1,\dots,\tilde x_4)$ as in
Fig.~\ref{fig:weingarten_cube}. Propagation of a cross
$(p_{12},p_{23},p_{34},p_{41})$ attached to the quadrilateral
$(x_1,x_2,x_3,x_4)$ according to the $C^1$-condition around horizontal edges (cf.
Remark~\ref{rem:c1_identity_projection}) yields a unique cross $(\tilde
p_{12},\tilde p_{23},\tilde p_{34},\tilde p_{41})$ attached to the
quadrilateral $(\tilde x_1,\tilde x_2,\tilde x_3,\tilde x_4)$.
The centres $q$ and $\tilde q$ of the crosses satisfy the Weingarten property, i.e.,
\begin{equation}
\inc{q,\tilde q} 
	= \inc{p_{12},p_{23},p_{34},p_{41}} 
		\cap \inc{\tilde p_{12},\tilde p_{23},\tilde p_{34},\tilde p_{41}}.
\label{eq:weingarten_property_centres}
\end{equation}
\label{thm:strong_c1_cube_consistency}
\end{theorem}

\begin{proof}
According to the $C^1$-identity (Lemma~\ref{lem:c1_identity}), propagation of
the initial cross $(p_{12}, p_{23}, p_{34}, p_{41})$ as described in
Remark~\ref{rem:c1_identity_projection} yields four well-defined points
$(\tilde p_{12},\tilde p_{23},\tilde p_{34},\tilde p_{41})$. We have to show
that these points are vertices of a cross, that is, that they are coplanar.

Consider Fig.~\ref{fig:weingarten_cube}. In a first step, we demonstrate that
the loop $(p_{12},\tilde p_{12},\tilde p_{34},p_{34})$ yields a refinement of
the original A-cube $\mathcal{C} = (x_1,\dots,x_4,\tilde x_1,\dots,\tilde x_4)$
into the two smaller A-cubes $\mathcal{C}_{41} = (x_1,p_{12},p_{34},x_4,$
$\tilde x_1,\tilde p_{12},\tilde p_{34},\tilde x_4)$ and $\mathcal{C}_{23} =
(x_2,p_{12},p_{34},x_3,\tilde x_2,\tilde p_{12},\tilde p_{34},\tilde x_3)$.
For symmetry reasons, it is sufficient to show that $\mathcal{C}_{23}$ is an
A-cube. This assertion, in turn, is true since $\mathcal{C}_{23}$ can be
obtained from $\mathcal{C}$ by applying two Cox deformations as illustrated in
Fig.~\ref{fig:cox_deformation_twice}.

\begin{figure}[htb]
\begin{center}
 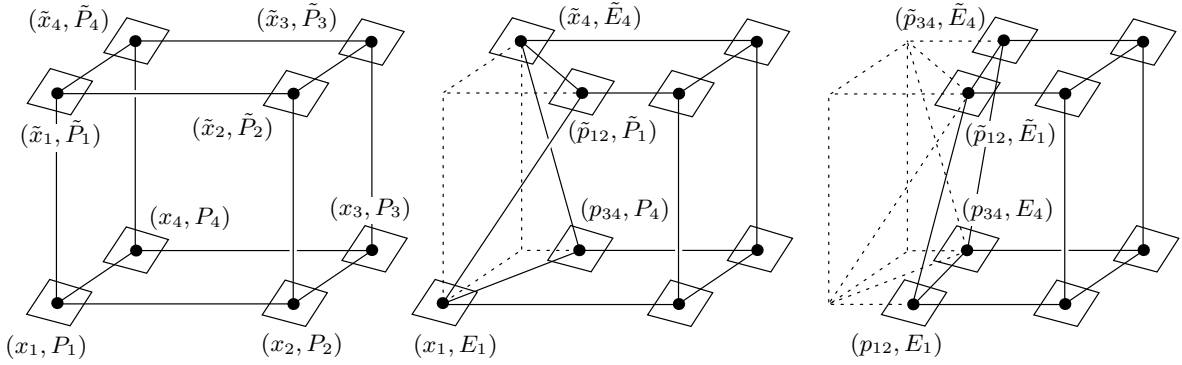 
\end{center}
\caption{Successive application of two Cox deformations.}
\label{fig:cox_deformation_twice}
\end{figure}

Next, we define
\begin{equation*}
\Pi = \inc{p_{12},p_{23},p_{34},p_{41}},
\quad
\tilde q = \Pi \cap \inc{\tilde p_{12},\tilde p_{34}},
\quad
\tilde \Pi = \inc{q,\tilde p_{12},\tilde p_{34}}.
\end{equation*}
The points $p_{23}$ and $\tilde
p_{41}$ are related by the $C^1$-identity with respect to the A-cube
$\mathcal{C}$, i.e., $\Pi_{41} = \inc{p_{23},x_1,x_4,\tilde p_{41}}$ is a
plane.  Since $p_{41},p_{23} \in \Pi_{41}$, we also have $q \in \Pi_{41}$.
This shows that $q$ and $\tilde p_{41}$ are related by the $C^1$-identity with
respect to the A-cube $\mathcal{C}_{41}$. Analogously, $q$ and $\tilde
p_{23}$ are related by the $C^1$-identity with respect to $\mathcal{C}_{23}$.
We may state this as
\begin{equation*}
\tilde p_{41} = \inc{\tilde x_1,\tilde x_4} \cap \tilde \Pi,
\quad
\tilde p_{23} = \inc{\tilde x_2,\tilde x_3} \cap \tilde \Pi,
\end{equation*}
which shows that
\begin{equation*}
\tilde \Pi = \inc{\tilde p_{12},\tilde p_{23},\tilde p_{34},\tilde p_{41}}.
\end{equation*}
In particular, the points $\tilde p_{12},\tilde p_{23},\tilde p_{34},\tilde p_{41}$
are coplanar and hence define a cross adapted to
$(\tilde x_1,\tilde x_2,\tilde x_3,\tilde x_4)$.
Now, we will demonstrate that, indeed,
\begin{equation*}
\tilde q = \inc{\tilde p_{12},\tilde p_{41}} \cap \inc{\tilde p_{23},\tilde p_{34}},
\end{equation*}
i.e., $\tilde q$ is the centre of the top cross as suggested by
Fig.~\ref{fig:weingarten_cube}.  For now, we denote the centre of the top cross by
$\tilde c$.  We have already established that each cross determines the other
cross according to \eqref{eq:weingarten_property_cross_vertices}, and that, if
we propagate from bottom to top, the center of the bottom cross is contained in
the plane of the top cross, that is, $q \in \tilde \Pi$.  For symmetry reasons, we also
have $\tilde c \in \Pi$ and, in particular,
\begin{equation}
\inc{q,\tilde c} = \Pi \cap \tilde \Pi.
\label{eq:w_cubes_intersection_cross_planes}
\end{equation}
Furthermore, we define $\bar q = \Pi \cap \inc{\tilde p_{23},\tilde p_{41}}$ and assume
that $\tilde q \ne \bar q$.  In this case, we have three distinct points
$\tilde c,\tilde q,\bar q \in \Pi$ that span the plane of the top cross, i.e.,
$\tilde \Pi = \inc{\tilde c,\tilde q,\bar q}$.  This, in turn, implies the
degenerate case $\Pi = \tilde \Pi$. Therefore, generically, $\tilde q = \bar q =
\tilde c$ and \eqref{eq:w_cubes_intersection_cross_planes} becomes
\eqref{eq:weingarten_property_centres}.
\end{proof}

\begin{definition}[Weingarten cube / Weingarten propagation of crosses]
An A-cube with two crosses adapted to opposite faces as in
Fig.~\ref{fig:weingarten_cube} is called a \emph{Weingarten cube} if the
Weingarten property \eqref{eq:weingarten_property_cross_vertices} is satisfied
by any pair $p_{ij},\tilde p_{ij}$ of corresponding cross vertices.  According
to Theorem~\ref{thm:strong_c1_cube_consistency}, an A-cube with a cross
attached to one face can be extended uniquely to a Weingarten cube by, for
instance, using the projections described in
Remark~\ref{rem:c1_identity_projection}. We refer to this extension as
\emph{Weingarten propagation} of the initial cross.
\label{def:weingarten_cube}
\end{definition}

\begin{remark}
A Weingarten cube determines a unique adapted hyperboloid for each face.  The
data needed to extend a Weingarten cube to a Blaschke cube are one point on a
``vertical'' (extended) edge (see Fig.~\ref{fig:extension_weingarten_to_blaschke}).
This yields four crosses adapted to the vertical faces that are
composed of asymptotic lines of the vertical hyperboloids.

\begin{figure}[htb]
\begin{center}
 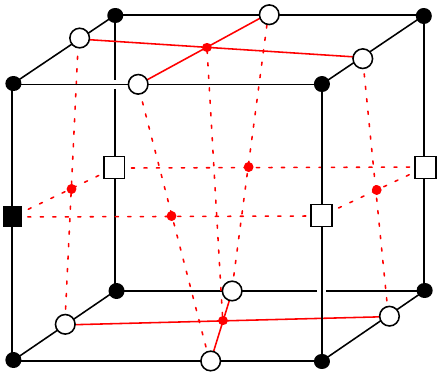 
\end{center}
\caption{\textup{Extension of a Weingarten cube to a Blaschke cube}.}
\label{fig:extension_weingarten_to_blaschke}
\end{figure}

According to Theorem~\ref{thm:strong_c1_cube_consistency}, it is possible to
have two ``$C^1$-loops'' of crosses (or adapted hyperboloids) around an A-cube.
However, it is not possible to have a third $C^1$-loop around the cube composed
of the crosses associated with the vertical faces.  The reason is that,
considering three faces adjacent to one vertex, one would have the
$C^1$-condition fulfilled around a vertex of degree 3. This would contradict
the fact that interior vertices of a pre-hyperbolic net have to be of even
degree.  In particular, Blaschke cubes that are Weingarten cubes in multiple
ways do not exist since at most one pair of opposite crosses can be related by
the Weingarten property.

Finally note that, if we prescribe which pair of opposite faces has to satisfy the
Weingarten property, a single hyperboloid of a Weingarten cube uniquely
determines all other hyperboloids according to the $C^1$-condition.
\label{rem:no_three_c1_loops}
\end{remark}

\begin{corollary}
Hyperboloids corresponding to crosses of a Weingarten cube form a (classical)
Weingarten pair.  In particular, the cross centres $q$ and $\tilde q$ are
corresponding points.
\label{cor:w_cubes_existence}
\end{corollary}

\begin{proof}
A cube with crosses attached to two opposite faces, top and bottom as in
Fig.~\ref{fig:weingarten_cube}, determines hyperboloids $h, \tilde h$ adapted
to those faces, and four hyperboloids adapted to vertical faces.  If
the crosses describe a Weingarten cube then any of the six adapted hyperboloids
determines the remaining ones uniquely (cf.
Remark~\ref{rem:no_three_c1_loops}).  Any point $q \in h$ is the centre of
a unique cross $c \subset h$ which determines, following asymptotic lines of the vertical
hyperboloids, a unique corresponding cross $\tilde c \subset \tilde h$ with
centre $\tilde q$. On noting that $c$ and $\tilde c$ are related by the Weingarten
propagation, the claim follows directly from
\eqref{eq:weingarten_property_centres}.
\end{proof}

It is evident that if both crosses of a Weingarten cube are internal then
Corollary~\ref{cor:w_cubes_existence} remains valid if one replaces
``hyperboloids'' by ``hyperboloid patches''.  Hence, the preceding analysis suggests 
regarding a Weingarten pair of (pre-)hyperbolic nets as being composed of
Weingarten cubes.

\begin{definition}[Weingarten transformation]
Two (pre-)hyperbolic nets $f=(x,\rho)$ and $\tilde f = (\tilde x, \tilde \rho)$
are said to be related by a \emph{Weingarten transformation} if 
\begin{enumerate}[\quad i)]
\item 
the supporting A-nets $x$ and $\tilde x$ form a (discrete) Weingarten pair, and
\item 
crosses adapted to corresponding quadrilaterals $Q$ and $\tilde Q$ of $x$ and $\tilde x$ are
related by the Weingarten propagation, i.e., $Q$ and $\tilde Q$ equipped with these crosses
are opposite faces of a Weingarten cube (see Definition~\ref{def:weingarten_cube})
so that the corresponding hyperboloids (patches) form a classical Weingarten pair.
\end{enumerate}
The nets $f$ and $\tilde f$ are said to form a \emph{Weingarten pair}
and $\tilde f$ is called a \emph{Weingarten transform} of $f$.
\label{def:weingarten_trafo}
\end{definition}

\paragraph{Relation between B\"acklund and Weingarten transformations
of pre-hyperbolic nets.}
Even though Definition~\ref{def:weingarten_trafo} is meaningful locally, {\em a priori}, it is
not evident that it is possible for all cubes to be simultaneously of Weingarten type. 
While the analysis of Weingarten pairs of hyperbolic nets is
more involved, the existence of Weingarten
pairs of pre-hyperbolic nets is easily established based on the B\"acklund transformation
introduced in the preceding. Here, the key is

\begin{proposition}
Let $(x,\rho)$ and $(\tilde x, \tilde \rho)$ be a B\"acklund pair of
pre-hyperbolic nets. If one cube with opposite faces $Q,\tilde Q$ as depicted
in Fig.~\ref{fig:cauchy_data_weak_baecklund} is a Weingarten cube, then
$(x,\rho)$ and $(\tilde x, \tilde \rho)$ form a Weingarten pair.  Conversely, 
if $(x,\rho)$ and $(\tilde x, \tilde \rho)$ form a Weingarten pair then 
$(\tilde x,\tilde \rho)$ and $(x,\rho)$
are B\"acklund-related.
\label{prop:relation_baecklund_weingarten_prehyp}
\end{proposition}

Proposition~\ref{prop:relation_baecklund_weingarten_prehyp} follows
almost immediately from the observation captured in the following

\begin{lemma}
In the sense of  Fig.~\ref{fig:transitive_quad_conditions},
the $C^1$-condition is transitive for crisscrossed quadrilaterals that
share an edge.
\label{lem:strong_c1_transitive}
\end{lemma}

\begin{figure}[htb]
\begin{center}
 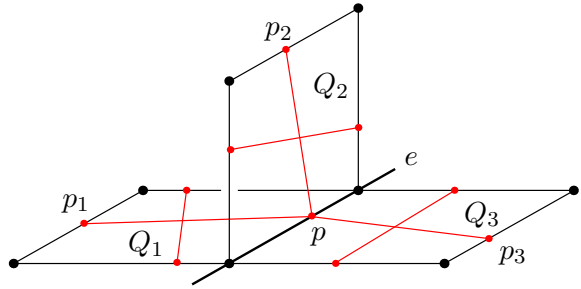 
\end{center}
\caption{ Crisscrossed quadrilaterals sharing an edge.  If hyperboloids adapted
to $Q_1$ and $Q_2$ are tangent along the extended edge $e$ and the same holds
for $Q_2$ and $Q_3$ then the hyperboloids adapted to $Q_1$ and $Q_3$ are likewise
tangent along $e$. In particular, the four cross vertices $p,p_1,p_2,p_3$ are coplanar. }
\label{fig:transitive_quad_conditions}
\end{figure}

\begin{remark}
Lemma~\ref{lem:strong_c1_transitive} shows that
Lemma~\ref{lem:weak_c1_two_cubes_consistency} is a direct consequence of
Theorem~\ref{thm:strong_c1_cube_consistency}. Furthermore, iterative
application of the Weingarten transformation leads to nets for which the
$C^1$-condition is satisfied in all coordinate planes.  Thus, in the case of
Weingarten transformations, Remark \ref{rem:twofold_baecklund_trafo} becomes
obsolete.
\end{remark}

\begin{remark}
It is not difficult to check, e.g., by considering parallel invariants,
that equi-twist is transitive in the same sense as the $C^1$-condition.
\label{rem:equi_twist_transitive}
\end{remark}

\begin{titleproof}{Proof of Proposition~\ref{prop:relation_baecklund_weingarten_prehyp}}
Consider adjacent cubes of a B\"acklund pair $(x,\rho), (\tilde x, \tilde \rho)$ as
depicted in Fig.~\ref{fig:propagation_of_weingarten_property} and denote the
constituent quadrilaterals by
\begin{equation*}
(x_1,x_2,x_5,x_6) \leftrightarrow Q_{1256},
\quad
(x_2,x_5,\tilde x_5,\tilde x_2) \leftrightarrow Q_{25\tilde5\tilde2},
\quad \text{etc.}
\end{equation*}

\begin{figure}[htb]
\begin{center}
 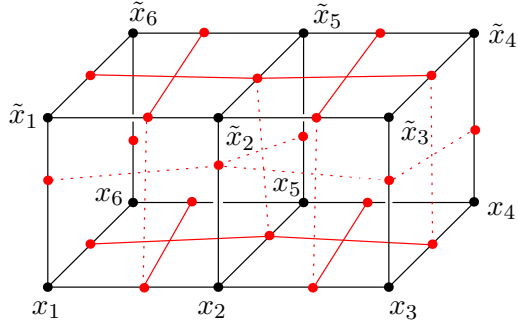 
\end{center}
\caption{Adjacent cubes of a B\"acklund pair of pre-hyperbolic nets.}
\label{fig:propagation_of_weingarten_property}
\end{figure}

The values of $\rho$ and $\tilde \rho$ at the 12 vertices yield a unique cross
for each of the 11 elementary quadrilaterals.  Assuming that the left cube with
crosses adapted to opposite faces $Q_{1256}$ and
$Q_{\tilde1\tilde2\tilde5\tilde6}$ is a Weingarten cube, we now show that the
right cube is a Weingarten cube as well.  According to
Lemma~\ref{lem:c1_identity}, it is sufficient to show that for each pair
$(Q_{2345}, Q_{25\tilde5\tilde2})$, $(Q_{2345}, Q_{23\tilde3\tilde2})$,
$(Q_{\tilde2\tilde3\tilde4\tilde5}$, $Q_{25\tilde5\tilde2})$, and
$(Q_{\tilde2\tilde3\tilde4\tilde5}$, $Q_{23\tilde3\tilde2})$ of quadrilaterals
the corresponding crosses satisfy the $C^1$-condition.
For symmetry reasons, we have to consider only the two pairs containing, for
example, the quadrilateral $Q_{\tilde2\tilde3\tilde4\tilde5}$.  By assumption,
the crosses attached to $Q_{\tilde1\tilde2\tilde5\tilde6}$ and
$Q_{\tilde2\tilde3\tilde4\tilde5}$ satisfy the $C^1$-condition.  Since the left
cube is a Weingarten cube, also the crosses attached to
$Q_{\tilde1\tilde2\tilde5\tilde6}$ and $Q_{25\tilde5\tilde2}$ satisfy the
$C^1$-condition. Therefore, Lemma~\ref{lem:strong_c1_transitive} implies that
the same holds for the crosses attached to $Q_{\tilde2\tilde3\tilde4\tilde5}$
and $Q_{25\tilde5\tilde2}$.  On the other hand, the pairs of crosses adapted to
$(Q_{12\tilde2\tilde1},Q_{\tilde1\tilde2\tilde5\tilde6})$,
$(Q_{12\tilde2\tilde1},Q_{23\tilde3\tilde2})$, and $(Q_{\tilde1\tilde2\tilde5\tilde6},
Q_{\tilde2\tilde3\tilde4\tilde5})$ each satisfy the $C^1$-condition.
Thus, according to Lemma~\ref{lem:4_cc_quads_consistency}, also the crosses
attached to $Q_{\tilde2\tilde3\tilde4\tilde5}$ and $Q_{23\tilde3\tilde2}$
satisfy the $C^1$-condition. Reversing the above argument
shows that, indeed, every Weingarten pair of pre-hyperbolic
nets constitutes a B\"acklund pair.
\end{titleproof}

Proposition~\ref{prop:relation_baecklund_weingarten_prehyp} shows that
Weingarten pairs of pre-hyperbolic nets exist since Cauchy data $\tilde \rho$
at one initial quadrilateral $\tilde Q$ for a B\"acklund transformation as
depicted in Fig.~\ref{fig:cauchy_data_weak_baecklund} can be chosen in such a manner that
the initial cube constitutes a Weingarten cube.  We obtain the following constructive
description of Weingarten transformations of pre-hyperbolic nets.

\begin{theorem}
Let $f=(x,\rho) : \Z^2 \to \R^3 \times \R$ be a pre-hyperbolic net and let
$\tilde x$  be a Weingarten transform of the A-surface $x$. The extension of
$\tilde x$ to a crisscrossed A-net $\tilde f = (\tilde x,\tilde \rho)$
according to Definition~\ref{def:weingarten_trafo}, ii), i.e.,
extending elementary hexahedra to Weingarten cubes, yields a Weingarten
transform $\tilde f$ of $f$. The Weingarten transform $\tilde f$ is uniquely
determined (modulo homogeneous rescaling of $\tilde \rho$) by its supporting
A-surface $\tilde x$.
\label{thm:weingarten_extension_is_consistent}
\end{theorem}

\begin{remark}
Proposition~\ref{prop:relation_baecklund_weingarten_prehyp} states that the
property of Blaschke cubes being Weingarten cubes propagates ``horizontally''
in a crisscrossed 2-layer 3D A-net whose two layers form a B\"acklund pair.
Moreover, Weingarten cubes propagate also ``vertically'' in the following
sense.  Consider a 3D crisscrossed A-net $(x,\rho)$ that satisfies the
$C^1$-condition in coordinate planes, i.e., $\rho$ is of the form
\eqref{eq:general_solution_rho_weak_c1_3d} and satisfies
\eqref{eq:modified_bkp_for_rho}.  Then, any of the three families of
``parallel'' crisscrossed coordinate surfaces may be interpreted as a family
$f^{(i)} = (x^{(i)},\rho^{(i)})$ of pre-hyperbolic nets such that for all $i$,
$f^{(i)}$ and $f^{(i+1)}$ form a B\"acklund pair and, in addition,
$f^{(i)},f^{(i+2)}$ are related in a particular manner (in contrast to an
arbitrary family of B\"acklund transforms (c.f.
Remark~\ref{rem:twofold_baecklund_trafo})).  Now, if we assume that a single
elementary cube of $(x,\rho)$ constitutes a Weingarten cube then one of the
three families of B\"acklund transformations is privileged (cf.
Remark~\ref{rem:no_three_c1_loops}) and the proof of
Proposition~\ref{prop:relation_baecklund_weingarten_prehyp} (the transitivity
of the $C^1$-condition) reveals that, in fact, with respect to this
distinguished family, all B\"acklund pairs represent Weingarten pairs.  In this
connection, we observe that even if we do not make the assumption of an initial
Weingarten cube then a (weaker) Weingarten connection still exists. Thus, since
for any family $f^{(i)}$ of B\"acklund transforms as defined above the
$C^1$-condition maps the layer $f^{(i)}$ uniquely and independently of
$f^{(i+1)}$ to the layer $f^{(i+2)}$ and a double application of the Weingarten
transformation implies the vertical $C^1$ property, one may interpret
$f^{(i+2)}$ as being generated from $f^{(i)}$ by a double Weingarten
transformation.
\label{rem:weingarten_hierarchy}
\end{remark}

\paragraph{Algebraic description of Weingarten pairs in terms of potentials
$\tau$ that parametrize Moutard coefficients of the underlying A-net: A geometric
interpretation of solutions of the dBKP equation.}
\label{par:rho_gleich_tau} Consider a supporting 2-layer 3D A-net $x : \Z^2
\times \left\{ 0,1 \right\} \to \R^3$.  With respect to a potential $\tau$ for
Moutard coefficients of $x$,
relation~\eqref{eq:general_solution_rho_weak_c1_3d} is the general parametrization of
$\rho : \Z^2 \times \left\{ 0,1 \right\} \to \R^3$ such that $(x,\rho)$ is a
B\"acklund pair of pre-hyperbolic nets adapted to the supporting A-net $x$.
By virtue of Proposition~\ref{prop:relation_baecklund_weingarten_prehyp}, if,
for a given $\rho : \Z^2 \times \left\{ 0 \right\} \to \R$ which describes the
extension of $x(\cdot,0)$ to a pre-hyperbolic net, one chooses the parameters
$\gamma,\gamma^{ij},f^{(i)}$ in \eqref{eq:general_solution_rho_weak_c1_3d}
with $\gamma^{23},\gamma^{31}$ and $\gamma$ interpreted as functions of $z_3$ as
indicated in Remark \ref{rem:twofold_baecklund_trafo} such
that one initial cube of $(x,\rho)$ becomes a Weingarten cube then $(x,\rho)$
describes a Weingarten pair.  It turns out that this reduces
\eqref{eq:general_solution_rho_weak_c1_3d} to
$\rho = \tau$ modulo a reparametrization of $\tau$ which
corresponds to black-white rescaling of Lelieuvre normals $n$ for $x$, possibly
combined with a change of Moutard coefficients to be parametrized by $\tau$.
The following lemma encapsulates this
relation for a single Weingarten cube.

\begin{lemma}
Consider a crisscrossed A-cube given by $(x,\rho)$ as in
Fig.~\ref{fig:equi_twisted_cube}, left and let $n$ be Lelieuvre normals of $x$
with Moutard coefficients $a^{ij}$ chosen as in Fig.~\ref{fig:equi_twisted_cube}, middle.
The pair $(x,\rho)$ governs a Weingarten cube with respect to the top and bottom faces
if and only if there exists a $\lambda \in \R$ such that
\begin{equation}
a^{ij} = \lambda \frac{\rho_i \rho_j}{\rho \rho_{ij}}, \quad
a^{ij}_k = \lambda^{-1} \frac{\rho_{ik} \rho_{jk}}{\rho_k \rho_{ijk}}, \quad
(i,j,k) \in \left\{ (2,1,3),(2,3,1),(3,1,2) \right\}.
\label{eq:w_cube_m_coeffs_in_rho}
\end{equation}
In particular, modulo a suitable black-white rescaling of $n$, either 
the Moutard coefficients $(a^{21},a^{23},a^{31})$ or $(a^{12},a^{32},a^{13})$
and their respective shifts are parametrized by $\rho$.
\label{lem:w_cube_m_coeffs_in_rho}
\end{lemma}

\begin{figure}[htb]
\begin{center}
 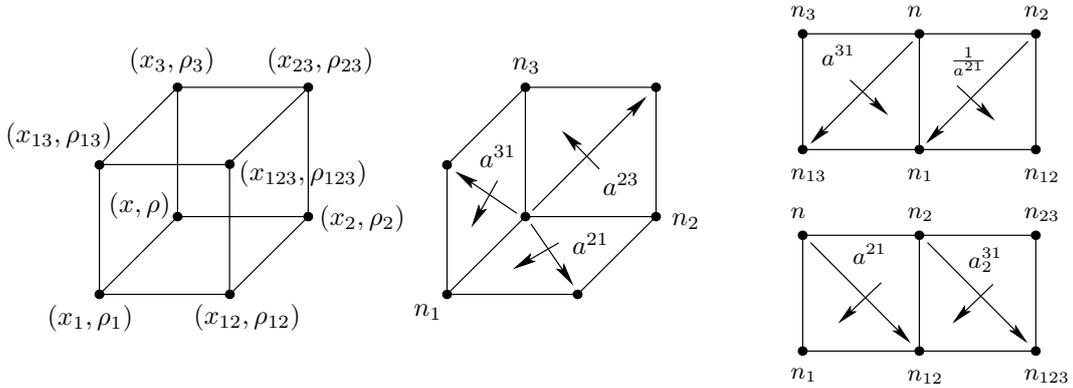 
\end{center}
\caption{Algebraic data of a crisscrossed A-cube.}
\label{fig:equi_twisted_cube}
\end{figure}

\begin{proof}
We consider the two pairs of adjacent Moutard coefficients in
Fig.~\ref{fig:equi_twisted_cube}, right associated with horizontal edges of
direction 1 and the two corresponding pairs associated with direction
2.  According to Lemma~\ref{lem:local_c1_rho} / Remark~\ref{rem:local_c1_rho_second_invariant}, the Weingarten
propagation of the cross determined by $(\rho,\rho_1,\rho_{12},\rho_2)$ is
described by the conditions
\begin{equation}
\frac{a^{31}}{a^{21}} = \frac{\rho_3 \rho_{12}}{\rho_2 \rho_{13}}, \quad
a^{21}a^{31}_2 = \frac{\rho_1 \rho_{23}}{\rho \rho_{123}}, \quad
\frac{a^{23}}{a^{21}} = \frac{\rho_3 \rho_{12}}{\rho_1 \rho_{23}}, \quad
a^{21}a^{23}_1 = \frac{\rho_2 \rho_{13}}{\rho \rho_{123}}.
\label{eq:weingarten_propagation_in_rho}
\end{equation}
Now, we define $\lambda$ such that 
\begin{equation*}
a^{21} = \lambda \frac{\rho_1 \rho_2}{\rho \rho_{12}}
\end{equation*}
so that the relations  \eqref{eq:weingarten_propagation_in_rho} become
\begin{equation*}
a^{23} = \lambda \frac{\rho_2 \rho_3}{\rho \rho_{23}}, \quad
a^{31} = \lambda \frac{\rho_1 \rho_3}{\rho \rho_{13}}, \quad
a^{23}_1 = \lambda^{-1} \frac{\rho_{12} \rho_{13}}{\rho_1 \rho_{123}}, \quad
a^{31}_2 = \lambda^{-1} \frac{\rho_{12} \rho_{23}}{\rho_2 \rho_{123}}.
\end{equation*}
Furthemore, the relation \eqref{eq:moutard_constant_ratio}, i.e., equality of
the three ratios of opposite Moutard coefficients, yields
\begin{equation*}
a^{21}_3 = \frac{a^{21} a^{23}_1}{a^{23}} = \frac{a^{21} a^{31}_2}{a^{31}}
	= \lambda^{-1} \frac{\rho_{13} \rho_{23}}{\rho_3 \rho_{123}}.
\end{equation*}
Finally, we observe that a black-white rescaling \eqref{eq:bw_rescaling} of $n$ by $\alpha$
at even vertices and by $\alpha^{-1}$ at odd vertices amounts to a rescaling of
$a^{21},a^{23},a^{31}$ by $\alpha^2$ and a rescaling of
$a^{21}_3,a^{23}_1,a^{31}_2$ by $\alpha^{-2}$.  Therefore, $\lambda = \pm 1$
can always be achieved. After this normalization, either the Moutard
coefficients $(a^{21},a^{23},a^{31})$ or $(a^{12},a^{32},a^{13})$ and their
respective shifts are parametrized by $\rho$, depending on the sign of
$\lambda$.
\end{proof}

\begin{remark}
The proof of Lemma~\ref{lem:w_cube_m_coeffs_in_rho} demonstrates the
consistency of \eqref{eq:weingarten_propagation_in_rho} regarded as
evolution equations for given $\rho,\rho_1,\rho_2,\rho_3,\rho_{12}$ and represents 
an algebraic proof of Theorem~\ref{thm:strong_c1_cube_consistency}.
\label{rem:weingarten_propagation_consistency_algebraic}
\end{remark}

Application of Lemma~\ref{lem:w_cube_m_coeffs_in_rho} to a 2-layer lattice now yields

\begin{theorem}
A crisscrossed A-net $(x,\rho) : \Z^2 \times \left\{ 0,1 \right\} \to \R^3
\times \R$ encodes a Weingarten pair of pre-hyperbolic nets if and only if
there exists a Lelieuvre normal field $n$ of $x$ such that $\tau = \rho$
parametrizes in the sense of \eqref{eq:moutard_param_by_tau} either the Moutard coefficients $(a^{21},a^{23},a^{31})$ or $(a^{12},a^{32},a^{13})$.
\label{thm:rho_gleich_tau}
\end{theorem}

\begin{proof}
If $\tau = \rho$ is a potential for Moutard coefficients then, according to
Lemma~\ref{lem:w_cube_m_coeffs_in_rho}, every elementary cube constitutes a Weingarten cube,
i.e., $(x,\rho)$ describes a Weingarten pair.
Conversely, we observe that it is possible to achieve
\begin{equation}
\frac{\rho_1 \rho_2}{\rho \rho_{12}} = \varepsilon a^{21}, \quad \varepsilon = \pm 1
\label{eq:rho_tau_pm}
\end{equation}
for one initial quadrilateral by applying a suitable black-white rescaling of Lelieuvre normals $n$ of $x$.
With respect to this normalization,
\begin{equation*}
\tau(\mathcal{S}^1) = \rho(\mathcal{S}^1), \quad
\tau(\mathcal{S}^2) = \rho(\mathcal{S}^2), \quad
\tau(1,1,0) = \rho(1,1,0), \quad \tau(0,0,1) = \rho(0,0,1)
\end{equation*}
are Cauchy data for a potential $\tau$ that parametrizes the coefficients
$\varepsilon(a^{21},a^{23},a^{31})$, whereby $\varepsilon = -1$ corresponds to
a parametrization of $(a^{12},a^{32},a^{13})$.  With respect to this unique
potential $\tau$, the $C^1$-condition \eqref{eq:weak_c1_in_tau} in the
coordinate plane $\Z^2 \times \left\{ 0 \right\}$ reduces to
$\tau(\mathcal{S}^{12}) = \rho(\mathcal{S}^{12})$, which is satisfied
by assumption. Now, since $\tau(0,0,1) = \rho(0,0,1)$, the Weingarten conditions
\eqref{eq:weingarten_propagation_in_rho} imply that $\tau = \rho$ on the entire
cube containing the initial quadrilateral. Iterative application of this
argument shows that $\tau = \rho$ everywhere.
\end{proof}

\begin{remark}
If we iterate the Weingarten transformation of a pre-hyperbolic net adapted to
one $(1,2)$-layer of a 3D A-net $x:\Z^3 \to \R^3$ then we obtain a crisscrossed A-net
$(x,\rho)$ and the above theorem implies that $\tau=\rho$ parametrizes either Moutard
coefficients $(a^{21},a^{23},a^{31})$ or $(a^{12},a^{32},a^{13})$
of a distinguished Lelieuvre representation $n$ of $x$.
Therefore, $\tau$ satisfies the discrete BKP equation in the form
\begin{equation}
\tau \tau_{123} - \tau_1 \tau_{23} - \tau_2 \tau_{13} + \tau_3 \tau_{12} = 0
\label{eq:bkp_Weingarten_12}
\end{equation}
which characterizes potentials $\tau$ for solutions
$\varepsilon(a^{21},a^{23},a^{31}), \varepsilon = \pm 1$ of \eqref{eq:ste} and constitutes
the analogue of equation \eqref{eq:bkp_lexicographic} characterizing
potentials for lexicographically ordered coefficients.
Conversely, if $\tau$ is a solution of \eqref{eq:bkp_Weingarten_12} then the
discrete Moutard equations \eqref{eq:moutard_minus_mD} together with the
discrete Lelieuvre formulae \eqref{eq:lelieuvre_mD} give rise to a class of 
corresponding A-nets $x$ with Lelieuvre normals for which either the Moutard coefficients
$(a^{21},a^{23},a^{31})$ or $(a^{12},a^{32},a^{13})$ are parametrized by
$\tau$ and $(x,\tau)$ completely encodes a family of Weingarten pairs of pre-hyperbolic 
nets with respect to the distinguished $(1,2)$-coordinate planes. In this manner, a novel
geometric intepretation of the discrete BKP equation \eqref{eq:bkp_Weingarten_12} is
uncovered in that the solution of the latter {\em directly} parametrizes the hyperboloids (patches) adapted 
to a corresponding A-net. In general, solutions of discrete BKP
equations of the type \eqref{eq:bkp_lexicographic} with two plus and two minus
signs on a 3-dimensional lattice correspond to families of Weingarten
transforms with respect to the associated distinguished $(i,k)$-coordinate planes.
\label{rem:bkp_geometric}
\end{remark}

\subsection{Equi-twisted 3D A-nets and Weingarten transformations of hyperbolic nets}

A pre-hyperbolic net $(x,\rho)$ is a hyperbolic net if all crosses are internal
and therefore describe hyperboloid patches. This is the case if and only if
$\rho$ is strictly positive or strictly negative.  Assuming that an A-surface
$x$ is simply connected, it is possible to extend $x$ to a hyperbolic net if
and only if $x$ is equi-twisted, i.e., if and only if all parallel invariants
of $x$ are positive (cf.
Proposition~\ref{prop:hypnet_extension_positive_invariants}).

The key aspect in the determination of Weingarten transformations of
hyperbolic nets is the analysis of the equi-twist properties of
multidimensional A-nets. In Theorem~\ref{thm:weingarten_extension_is_consistent}, 
it is justified to replace the term ``pre-hyperbolic net'' by ``hyperbolic net''
if we can ensure that all 3-dimensional A-cubes have the property that
the Weingarten propagation of crosses preserves internal crosses.
We begin with the discussion of cubes that have this property.

\begin{definition}[Equi-twisted A-cubes]
We call an A-cube \emph{equi-twisted with respect to a pair $(Q,\tilde Q)$ of
opposite faces} if both loops of edge-adjacent quadrilaterals containing $Q$
and $\tilde Q$ are equi-twisted.
\label{def:et_acube}
\end{definition}

\begin{lemma}
Let $\mathcal{C}$ be an A-cube with an internal cross attached to one face $Q$.
The Weingarten propagation of this cross to the opposite face $\tilde Q$
yields an internal cross adapted to $\tilde Q$ if and only if
$\mathcal{C}$ is equi-twisted with respect to $(Q,\tilde Q)$.
\label{lem:equi_twisted_a_cube}
\end{lemma}
\begin{proof}
Propagation of cross vertices according to the $C^1$-condition maps an internal
vertex to an internal vertex if and only if the corresponding two edge-adjacent
quadrilaterals are equi-twisted
(Lemma~\ref{lem:propagation_of_internal_vertices}). This implies that two
edge-adjacent quadrilaterals of an A-cube are equi-twisted if and only if the
opposite two quadrilaterals are equi-twisted (cf.
Remark~\ref{rem:c1_identity_projection}) and the claim of the lemma follows 
(cf.  Fig.~\ref{fig:diagonal_cross_vertex_propagation}).
\end{proof}

\begin{figure}[htb]
\begin{center}
 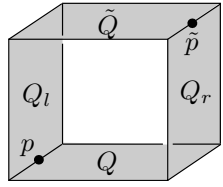 
\end{center}
\caption{An internal cross vertex $p$ is propagated to an internal cross vertex
$\tilde p$ if and only if $(Q,Q_r)$ are equi-twisted, which is equivalent to
$(Q_l,\tilde Q)$ being equi-twisted.  If both, $(Q,Q_l)$ and $(Q,Q_r)$ are
equi-twisted then the complete loop $(Q,Q_r,\tilde Q,Q_l)$ is equi-twisted.  }
\label{fig:diagonal_cross_vertex_propagation}
\end{figure}

Equi-twisted A-cubes exist in the following sense. We recall the algebraic description
\eqref{eq:weingarten_propagation_in_rho} of the Weingarten propagation
in the context of Fig.~\ref{fig:equi_twisted_cube} and observe, in analogy with
Lemma~\ref{lem:propagation_of_internal_vertices}, that an internal cross
adapted to the bottom quadrilateral of the cube in
Fig.~\ref{fig:equi_twisted_cube}, left is propagated to an internal cross adapted
to the top quadrilateral if and only if the four relevant parallel invariants are
positive, i.e.,
\begin{equation}
\frac{a^{31}}{a^{21}} > 0, \quad a^{21}a^{31}_2 > 0,
\quad \frac{a^{23}}{a^{21}} > 0, \quad a^{21}a^{23}_1 > 0.
\label{eq:et_wcube_condition}
\end{equation}
This means that the cube in question is equi-twisted with respect to the top and bottom quadrilaterals
if and only if \eqref{eq:et_wcube_condition} is satisfied and, hence, 
we have to demonstrate that it is possible to achieve \eqref{eq:et_wcube_condition}
for a single A-cube. In terms of arbitrarily chosen Moutard
coefficients $a^{21},a^{23},a^{31}$, the propagation \eqref{eq:ste}
of those coefficients on an A-cube reads
\begin{equation}
a^{ij}_k = \frac{a^{ij}}{a^{21} (a^{23} + a^{31}) - a^{23}a^{31}}.
\label{eq:ste_123}
\end{equation}
Now, it is possible to choose positive coefficients $a^{21},a^{23},a^{31}$ in such
a manner that the denominator in \eqref{eq:ste_123} is positive so that a
solution of \eqref{eq:et_wcube_condition} is obtained.
It is noted that the coefficient $a^{21}_3$ is then also positive, which corresponds to
the symmetry of equi-twisted A-cubes with respect to the two distinguished opposite
faces. Fig.~\ref{fig:et_acube_example} displays an example of an equi-twisted Weingarten cube
and the associated hyperboloid patches.

\begin{figure}[htb]
\begin{center}
\parbox{.28\textwidth}{\includegraphics[width=\linewidth]{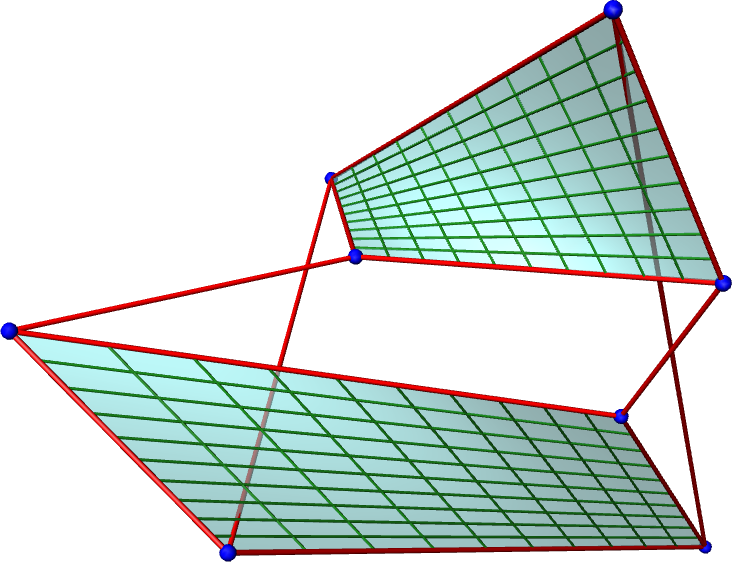}}
\hspace{.06\linewidth}
\parbox{.25\textwidth}{\includegraphics[width=\linewidth]{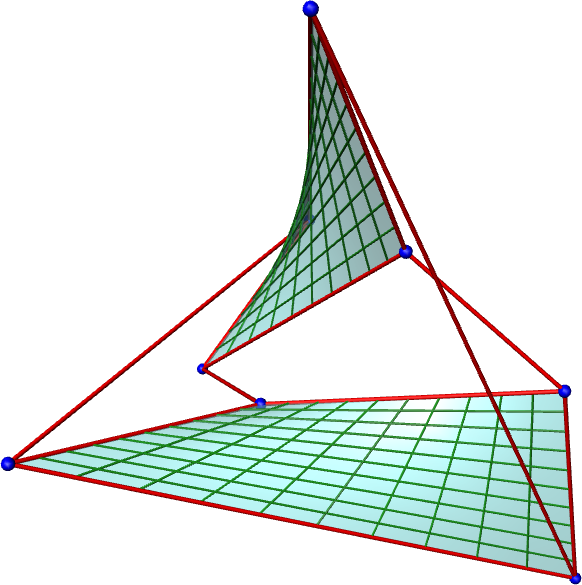}}
\hspace{.06\linewidth}
\parbox{.29\textwidth}{\includegraphics[width=\linewidth]{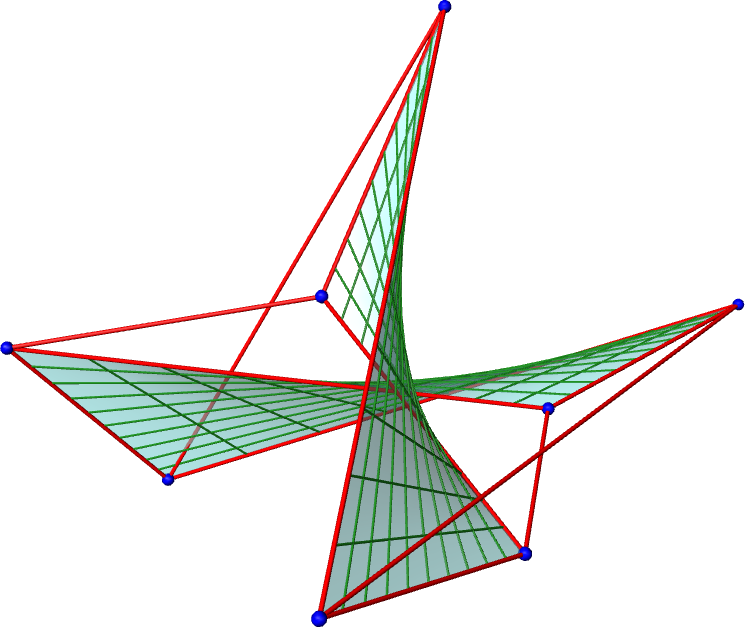}}
\end{center}
\begin{center}
\parbox{.28\textwidth}{\includegraphics[width=\linewidth]{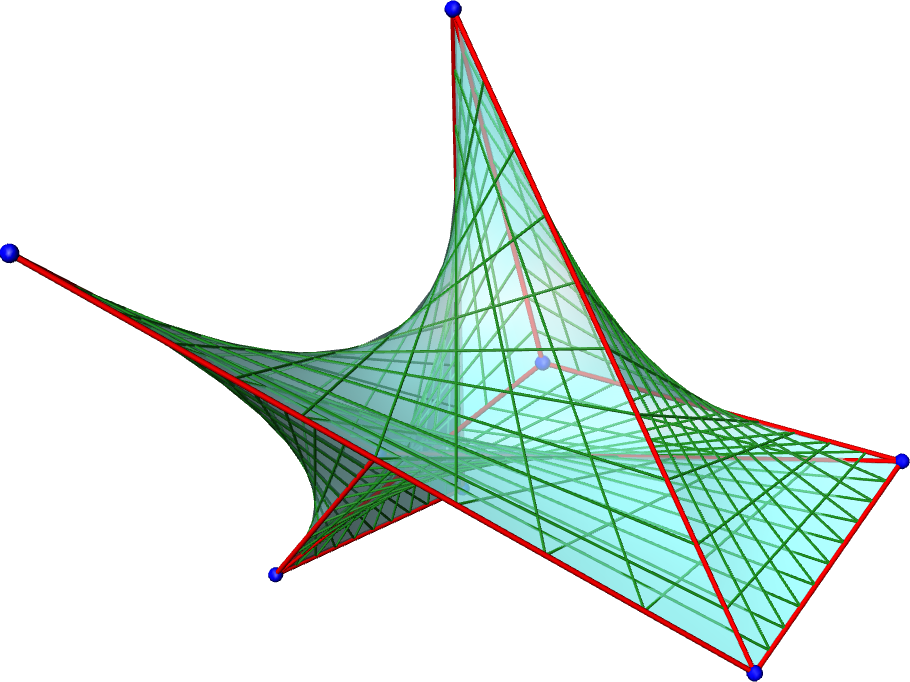}}
\hspace{.05\linewidth}
\parbox{.28\textwidth}{\includegraphics[width=\linewidth]{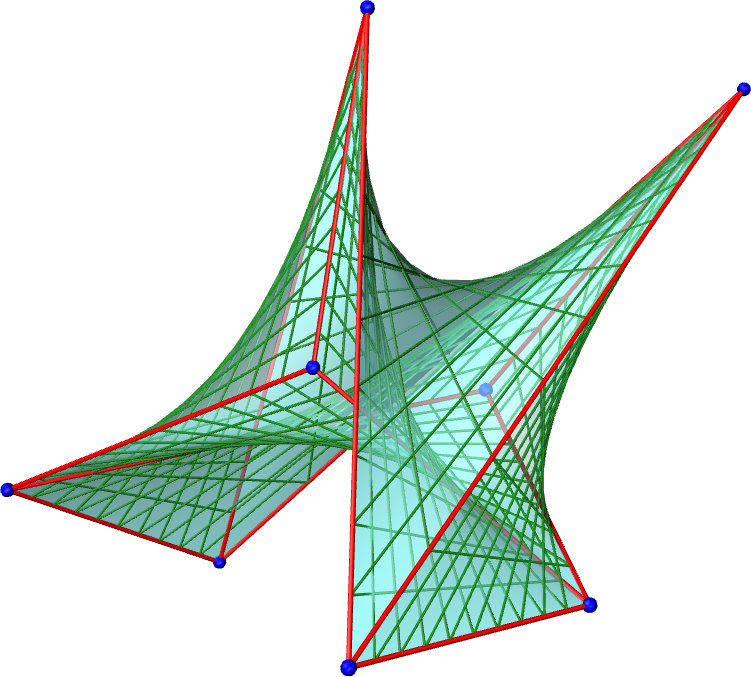}}
\end{center}
\caption{Top: An equi-twisted A-cube with a Weingarten pair of hyperboloid patches
adapted to opposite faces.
Bottom: The same cube completed with the unique patches that are adapted to vertical faces
and determined by the distinguished Weingarten pair via the $C^1$-condition.}
\label{fig:et_acube_example}
\end{figure}

\begin{remark}
The previous considerations show that it is impossible for all three loops of
an A-cube to be equi-twisted. Indeed, the equi-twist conditions associated with
the additional "horizontal" loop of quadrilaterals in
Fig.~\ref{fig:equi_twisted_cube} may be expressed as $a^{31}/a^{23} < 0$ and
$a^{31}a^{23}_1 < 0$, which cannot be satisfied simultaneously with
\eqref{eq:et_wcube_condition}.  An analogous argument may be used to show that
all interior vertices of an equi-twisted A-surface are of even degree.
\label{rem:no_three_loops_equi_twisted}
\end{remark}

Based on Definition~\ref{def:et_acube}, we say that an A-net $x : \Z^3 \to
\R^3$ is \emph{equi-twisted with respect to consecutive layers
$x^{(k)}=x(\cdot,k)$ and $x^{(k+1)}=x(\cdot,k+1)$} if the elementary hexahedra "between"
the restrictions of $x$ to $x^{(k)}$ and $x^{(k+1)}$ are equi-twisted with
respect to corresponding opposite quadrilaterals of $x^{(k)}$ and $x^{(k+1)}$.
As a consequence, each of the layers $x^{(k)},x^{(k+1)}$ is then equi-twisted
itself (see Remark~\ref{rem:equi_twist_transitive}).  Thus, by virtue of 
Lemma~\ref{lem:equi_twisted_a_cube}, the Weingarten propagation of (internal
crosses of) a hyperbolic net adapted to $x^{(k)}$ generates a Weingarten transform
adapted to $x^{(k+1)}$ if and only if $x$ is equi-twisted with respect to
$x^{(k)},x^{(k+1)}$.
The previous considerations, in particular, the algebraic equi-twist condition
\eqref{eq:et_wcube_condition}, combined with Theorem~\ref{thm:rho_gleich_tau} and
Remark~\ref{rem:bkp_geometric} lead to the following description of equi-twisted
A-nets and Weingarten transformations of adapted hyperbolic nets.

\begin{theorem}
\begin{enumerate}[i)]
\item 
A-nets $x : \Z^3 \to \R^3$ that are equi-twisted with respect to any two
consecutive $(1,2)$-layers are characterized by the property that, for any
Lelieuvre representation, either the Moutard coefficients
$(a^{21},a^{23},a^{31})$ or $(a^{12},a^{32},a^{13})$ may be parametrized by
positive solutions $\tau$ of the discrete BKP equation
\eqref{eq:bkp_Weingarten_12}.
\item
Hyperbolic nets that are adapted to the (1,2)-layers of an A-net $x:\Z^3 \to
\R^3$ and represented as $(x,\rho) : \Z^3 \to \R^3 \times \R$ constitute a family of
Weingarten transforms if and only if there exists a Lelieuvre normal field $n$
of $x$ such that $\tau = \rho$ parametrizes  according to
\eqref{eq:moutard_param_by_tau} either the Moutard coefficients
$(a^{21},a^{23},a^{31})$ or $(a^{12},a^{32},a^{13})$ of $n$.
\item
In particular, if $\tau$ is a positve solution of \eqref{eq:bkp_Weingarten_12} and $x$ is an
A-net with Moutard coefficients $(a^{21},a^{23},a^{31})$ or
$(a^{12},a^{32},a^{13})$ parametrized by $\tau$ then
$(x,\tau)$ encapsulates a family of Weingarten pairs of hyperbolic nets with
respect to the $(1,2)$-coordinate planes.
\end{enumerate}
\label{thm:bkp_geometric_positive}
\end{theorem}

\begin{remark}
The trivial solution $\tau \equiv 1$ of \eqref{eq:bkp_Weingarten_12}
corresponds to Weingarten transformations of the affine minimal surfaces
analysed in
\cite{Craizer:2010:AffineMinimalSurfaces,KaeferboeckPottmann:2012:DiscreteAffineMinimal}
since $\tau_1 \tau_2 / \tau \tau_{12} = \rho_1 \rho_2 / \rho
\rho_{12} = 1$ corresponds to hyperbolic paraboloid patches adapted to
(1,2)-quadrilaterals (cf.
Remark~\ref{rem:hyperbolic_paraboloids_shape_parameter}).
\end{remark}

We conclude with the following statement about general Weingarten pairs of hyperbolic nets.

\begin{proposition}
The class of Weingarten pairs of hyperbolic nets may be parametrized in terms of 
one-dimensional Cauchy data and a function of two variables 
(encoding one of the two hyperbolic A-nets) which is locally bounded below.
\label{prop:existence_positive}
\end{proposition}

\begin{proof}
According to Theorem~\ref{thm:bkp_geometric_positive}, the analysis of
Weingarten pairs of hyperbolic nets is equivalent to the analysis of A-nets $x:
\Z^2 \times \left\{ 0,1 \right\} \to \R^3$ that are equi-twisted with respect
to the two layers denoted by $x^{(0)}=x(\cdot,0)$ and $x^{(1)} = x(\cdot,1)$.
A necessary and sufficient condition for $x$ being
equi-twisted with respect to $x^{(0)},x^{(1)}$ is that for
every elementary hexahedron the corresponding Moutard coefficients satisfy
\eqref{eq:et_wcube_condition}. Therefore, it is sufficient to confine
ourselves to the consideration of Moutard coefficients and related Cauchy
problems. If we think of $x^{(0)}$ as a given equi-twisted layer then
it is convenient to regard the relevant Moutard coefficients as functions of
the ``horizontal'' variables $z_1$ and $z_2$, that is,
\begin{equation}
  a^{23},a^{31},a,a_3:\Z^2\to\R^3,
\end{equation}
where $a := a^{21}$. Hence, since $x^{(0)}$ is equi-twisted, we may assume without 
loss of generality that $a > 0$. It is therefore required to show that
\eqref{eq:et_wcube_condition} holds for all A-cubes ``between'' 
the two layers $x^{(0)}$ and $x^{(1)}$, that is,
\begin{equation}
a_3, a^{23}, a^{31} > 0.
\label{eq:POS}
\end{equation}
The latter Moutard coefficients are determined by the evolution equations 
\eqref{eq:ste_123} with Cauchy data consisting of coefficients $a$ for quadrilaterals 
of the $(1,2)$-plane and $a^{23},a^{31}$ for the ``vertical'' quadrilaterals over the
coordinate axes of the $(1,2)$-plane, i.e.,
\begin{equation}
a(\mathcal{S}^{12}),
\quad a^{23}(\left\{ 0 \right\} \times \Z),
\quad a^{31}(\Z \times \{0\}).
\label{eq:cauchy_data_moutard_coeffs_3d}
\end{equation}
As a necessary condition, the above Cauchy data have to be chosen positive.

In order to proceed, we now cast the Cauchy problem into a form which reflects 
the privileged role of the $(1,2)$-coordinate planes. Thus, the relation 
\begin{equation*}
\frac{a^{23}}{a^{23}_1} =
\frac{a^{31}}{a^{31}_2}, 
\label{eq:moutard_constant_ratio_3d}
\end{equation*}
which is a consequence of \eqref{eq:ste_123}, guarantees the existence of a 
potential $\Phi$ such that
\begin{equation}
a^{23} = \frac{\Phi}{\Phi_2}, \quad a^{31} = \frac{\Phi}{\Phi_1}.
\label{eq:potential_phi_3d}
\end{equation}
With respect to this potential, the evolution equations  \eqref{eq:ste_123} may then be written as
\begin{equation}
\Phi_{12} = a(\Phi_1 + \Phi_2) - \Phi,\quad a_3 = \frac{\Phi_1 \Phi_2}{\Phi \Phi_{12}} a.
\label{eq:evolution_phi_3d}
\end{equation}
Accordingly, in terms of $\Phi$, the Cauchy data 
\eqref{eq:cauchy_data_moutard_coeffs_3d} translate into the Cauchy data
\begin{equation*}
a(\mathcal{S}^{12}),
\quad \Phi(\left\{ 0 \right\} \times \Z),
\quad \Phi(\Z \times \left\{ 0 \right\}).
\label{eq:cauchy_data_phi_3d}
\end{equation*}
The parametrization \eqref{eq:potential_phi_3d} now shows that the positivity 
of the Moutard coefficients $a^{23}$ and $a^{31}$ leads to the key condition
\begin{equation*}
\Phi>0.
\end{equation*}
The latter is satisfied if the Cauchy data $\Phi(\left\{ 0 \right\} \times \Z)$ and 
$\Phi(\Z \times \left\{ 0 \right\})$ are positive and, at each step of the iteration, 
the Moutard coefficient $a$ in the four-point equation 
\eqref{eq:evolution_phi_3d}$_1$ is chosen in such a way that
\begin{equation}
a > \frac{\Phi}{\Phi_1 + \Phi_2}.
\label{eq:positive_phi_12}
\end{equation}
Finally, relation \eqref{eq:evolution_phi_3d}$_2$ shows that the positivity of 
$\Phi$ implies the positivity of the remaining coefficient $a_3$.
\end{proof}

\begin{remark}
In broad terms, if, in the context of the above proof, we regard the A-surface
$x^{(0)}$ as a discretization of a continuous 
surface then the condition \eqref{eq:positive_phi_12} imposes a constraint on 
the ``quality'' of the discretization of any particular surface rather than the 
surface itself. More precisely, in the context of the continuum limit alluded to 
in Section \ref{subsec:a-surfaces}, the functions $a$ and $\Phi$ admit the 
expansions
\begin{equation*}
 a = 1 + \frac{1}{2}\epsilon_1\epsilon_2A+\cdots,\quad 
\Phi_1 = \Phi + \epsilon_1\partial_1\Phi+\cdots,\quad 
\Phi_2 = \Phi + \epsilon_2\partial_2\Phi+\cdots
\end{equation*}
and the discrete Moutard equation \eqref{eq:evolution_phi_3d}$_1$ formally 
reduces to the classical Moutard equation
\begin{equation*}
\partial_1\partial_2\Phi = A\Phi
\end{equation*}
in the limit $\epsilon_i\to0$. The inequality \eqref{eq:positive_phi_12} then 
adopts the form
\begin{equation*}
1 + \epsilon_1\frac{\partial_1\Phi}{\Phi} + \epsilon_2\frac{\partial_2\Phi}{\Phi} 
+ \epsilon_1\epsilon_2A + \cdots > 0
\end{equation*}
which may be met for sufficiently ``small'' discretization parameters 
$\epsilon_i$ subject to appropriate boundedness assumptions on the 
functions $A$ and $\Phi$.
\end{remark}

\section{Perspectives}

The preceding analysis has revealed that Weingarten transformations for
hyperbolic nets are algebraically encoded in positive solutions of the discrete BKP
equation  \eqref{eq:bkp_Weingarten_12}. A single application of a Weingarten transformation
corresponds to a positive solution $\tau: \Z^2\times \{0,1\}\to\R$ of the discrete BKP
equation and the existence of such solutions has been proven in Proposition  
\ref{prop:existence_positive}. Iterated Weingarten transformations for hyperbolic nets
correspond to positive solutions $\tau$ of the discrete BKP equation which are defined on 
larger domains. In which sense these exist in both geometric and algebraic terms is currently
being investigated. In this context, it is also natural to examine the permutability theorems
associated with Weingarten transformations. Furthermore, it is necessary to inquire as to whether 
the positivity of $\tau$ may be preserved by the standard B\"acklund transformation for 
the discrete BKP equation. In geometric terms, this is closely related to the consideration of four- or 
higher-dimensional dimensional hyperbolic nets. In this connection, the application of Weingarten transformations to special discrete surfaces such as the discrete K-surfaces alluded to in the 
Introduction should be pursued. All these items will be addressed in a separate 
publication.

\ommit{
\input{appendix}
}

\bibliographystyle{amsalpha}
\bibliography{wtrafos}

\end{document}